\documentclass[11pt]{article}
\usepackage{graphicx} 
\usepackage{subcaption}

\usepackage[a4paper, margin=2cm]{geometry}
\usepackage{mystyle}
\usepackage[colorlinks=true,linkcolor=blue,citecolor=blue,urlcolor=blue]{hyperref}

\newcommand{\OT}{\mathtt{OT}}
\newtheorem{assumption}[thm]{Assumption}
  
\title{Second-order geometry of optimal transport: \\ cost identifiability and stability}
\date{\today}

\newcommand{\KL}{\mathrm{KL}}

\author{
Gabriel Peyr\'e 
\thanks{CNRS and ENS, PSL Universit\'e, Paris, France,
\texttt{gabriel.peyre@ens.fr}}
\and
Clarice Poon \thanks{
Mathematics Institute, University of Warwick, Coventry, UK,
\texttt{clarice.poon@warwick.ac.uk}}
\and
Oscar Tron
\thanks{
CentraleSupélec, Paris, France,
\texttt{oscar.tron@student-cs.fr}}
}

\begin{document}

\maketitle

\begin{abstract}
We study the optimal transport value as a functional of the ground cost. Although this functional is an infimum of linear functionals and is therefore polyhedral for finitely supported marginals, it develops genuine second-order geometry in the smooth continuous setting. Under twist, nondegeneracy, and regularity assumptions, we prove twice differentiability with respect to the cost in H\"older spaces, derive a variational formula for the Hessian, and characterize its null space. The resulting curvature is strict transverse to the cost perturbations that preserve the observed optimal plan.

This structure yields identifiability and quantitative stability for inverse optimal transport. For bilinear costs $c_A(x,y)=-x^\top Ay$, a full spanning condition on the Hessians of the Kantorovich potential implies that every nonzero matrix compatible with a single optimal coupling is a positive multiple of the true matrix. We show that this condition can be attained by arbitrarily small compactly supported smooth perturbations of the transport potential that preserve the target marginal's support and density class. For $d\ge2$, uniquely optimal discrete plans and same-base elliptical models, including Gaussians, retain nontrivial degeneracies. Finally, we quantify the regularization bias and sampling error of entropic inverse optimal transport as the regularization parameter vanishes. Thus the second-order geometry of the forward transport problem provides a means towards identifiability, stability, and statistical cost recovery.
\end{abstract}

\section{Introduction}

Fixing two probability measures $\alpha$ and $\beta$, we regard the optimal transport value $c\mapsto\OT_{\alpha,\beta}(c)$ as a functional of the ground cost. This viewpoint raises a basic geometric question: how does the optimal value respond to perturbations of the cost? Since it is the infimum of linear functionals of $c$, the OT value is concave. For finitely supported marginals it is piecewise linear, with the flat geometry coming from OT's correspondence to linear programming. Our starting observation is that the smooth continuous Monge regime behaves differently: regular motion of the optimal map creates genuine second-order curvature in the cost variable.

Our first main result makes this phenomenon precise for general smooth costs. Under twist, nondegeneracy, and regularity assumptions, we prove that the OT value is twice differentiable in H\"older spaces and derive a variational characterization of its Hessian. We also identify its null directions exactly: the second variation vanishes precisely when the cost perturbation induces a gradient field along the optimal map. These directions preserve the observed map under sufficiently small perturbations. The Hessian therefore separates the intrinsic plan-preserving deformations of the cost from directions in which the optimal coupling changes. This result is a forward sensitivity theorem, but it also supplies the local geometry needed to invert the transport problem.

We apply this geometry to inverse optimal transport (iOT), where the cost is unknown and an optimal coupling is observed. For the bilinear family $c_A(x,y)=-x^\top Ay$, we prove that a single smooth coupling identifies every nonzero compatible cost matrix up to positive scale whenever the Hessians of its Kantorovich potential span the full space $\mathbb S_d$ of symmetric matrices. We further show that this spanning condition can be attained by arbitrarily small compactly supported smooth perturbations of the transport potential that preserve the target marginal's support and regularity class. For $d\ge2$, the same condition yields local quadratic growth of the inverse loss transverse to scaling and hence quantitative stability. In contrast, uniquely optimal discrete plans and same-base elliptical models, including Gaussians, exhibit non-identifiability beyond scale in these dimensions.

The final part of the paper connects this unregularized geometry to statistical cost recovery. Much of the existing methodology for iOT uses entropic formulations at a fixed regularization level $\varepsilon>0$ \cite{dupuy2019estimating,carlier2023sista,andradesparsistency}. Entropy makes computation and statistical analysis more tractable, but it biases recovery when the observations come from unregularized OT; reducing $\varepsilon$ mitigates this bias while worsening sampling bounds. We quantify both effects and use the curvature of the unregularized inverse loss to derive a joint choice of $\varepsilon$ and sample size. Thus a single second-order mechanism links sensitivity of the forward OT problem to identifiability, stability, and finite-sample recovery in the inverse problem.

\subsection{Problem setting}
\paragraph{Optimal transport as a functional of the cost}
Let $X,Y\subset\RR^d$. For probability measures $\alpha\in\Pp(X)$ and $\beta\in\Pp(Y)$ and a cost $c\in\Cc(X\times Y)$, optimal transport is defined by
\begin{equation*}
\OT_{\alpha,\beta}(c)
\coloneqq\inf_{\pi\in\Pi(\alpha,\beta)}\int c(x,y)\,d\pi(x,y).
\tag{OT}\label{eq:OT}
\end{equation*}
Here $\Pi(\alpha,\beta)$ is the set of couplings with marginals $\alpha$ and $\beta$. Given an optimal coupling $\hat\pi$, the inverse problem is to determine which costs make it optimal. Positive rescaling and additive perturbations $a(x)+b(y)$ leave optimal plans unchanged, so recovery must account for these invariances. We focus on the bilinear family
\[
c_A(x,y)\coloneqq-x^\top Ay,\qquad A\in\RR^{d\times d}.
\]
On open domains this family contains no nonzero additive cost directions, leaving positive scaling as the unavoidable ambiguity. The choice $A=\Id$ is equivalent to the quadratic cost $\frac12\|x-y\|^2$, up to additive terms. More generally, a symmetric positive-definite $A$ defines a Mahalanobis metric.

The unregularized optimality gap is
\begin{equation*}
\Ll_0(A;\hat\pi)
\coloneqq\dotp{c_A}{\hat\pi}-\OT_{\alpha,\beta}(c_A).
\tag{$\Ll_0$}\label{eq:L0}
\end{equation*}
It is convex, nonnegative, and vanishes exactly when $\hat\pi$ is optimal for $c_A$. Its zero set describes identifiability, while its growth away from that set controls stability. Since its first term is linear in $A$, its Hessian is the negative Hessian of the OT value restricted to bilinear costs.

\paragraph{Sampling and entropic regularization}
For statistical recovery, we observe independent samples $(x_i,y_i)\sim\hat\pi$, $i=1,\ldots,n$, from a single unregularized optimal coupling. Following the entropic approach to iOT, we consider the regularized OT value
\begin{equation*}
\OT^\epsilon_{\alpha,\beta}(c)
\eqdef\inf_{\pi\in\Pi(\alpha,\beta)}
\Big(\dotp{c}{\pi}+\epsilon\mathrm{KL}(\pi\mid\alpha\otimes\beta)\Big)
\tag{$\mathrm{OT}_\epsilon$}\label{eq:OT_epsilon}
\end{equation*}
and its associated gap loss
\begin{equation*}
\Ll_\epsilon(A;\hat\pi)
\coloneqq\dotp{c_A}{\hat\pi}-\OT^\epsilon_{\alpha,\beta}(c_A).
\tag{$\Ll_\epsilon$}\label{eq:Leps}
\end{equation*}
At fixed $\epsilon>0$, this is the convex maximum-likelihood objective used in entropic iOT \cite{dupuy2019estimating,andradesparsistency}, and Sinkhorn algorithms evaluate its value and gradient \cite{carlier2023sista}. We study its empirical version as $\epsilon\to0$, with the constraint $A\succeq0$, $\operatorname{tr}(A)=1$ to fix the scale of the recovered cost. Our estimates separate the deterministic regularization bias from sampling error and give a joint choice of $\epsilon$ and $n$.

\subsection{Main contributions and outline}
The paper develops one structural result and follows its consequences from qualitative identifiability to quantitative recovery.
\begin{itemize}
    \item \textbf{Cost-space curvature and its null directions.}
    Theorem~\ref{thm:curvature} (Section~\ref{sec:curvature_OT}) establishes twice differentiability of the OT value in H\"older spaces and gives a variational formula for its Hessian. The formula characterizes the null space by a gradient-field condition along the optimal map. Corollary~\ref{cor:zero-curvature-directions} then shows that these and only these infinitesimal null directions are locally plan-preserving. This identifies the precise sense in which smooth continuous OT is curved despite its definition as an infimum of linear functionals of the cost.

    \item \textbf{Cost identifiability from a single coupling.}
    For bilinear costs, Theorem~\ref{thm:uniqueness} (Section~\ref{sec:uniqueness}) shows that the full spanning condition~\eqref{eq:hessian-span-symmetric} on $\mathbb S_d$ reduces the admissible set to the nonnegative ray generated by the true matrix; equivalently, every nonzero compatible cost is identified up to positive scale. Proposition~\ref{prop:generic_identifiability} shows that an arbitrarily small compactly supported smooth perturbation of the transport potential can enforce this condition while preserving the target marginal's support and density class. The discrete and elliptical examples in Sections~\ref{sec:discrete} and~\ref{sec:Gaussian} expose the contrasting degeneracy mechanisms.

    \item \textbf{Inverse stability and entropic recovery.}
    Theorem~\ref{thm:local_curvature_L0} (Section~\ref{sec:bilinear}) turns the spanning condition into local quadratic growth of the inverse optimality gap transverse to scaling. This gives the quantitative stability estimate needed for noisy recovery. Section~\ref{sec:sample_setting} then analyzes a computable entropic inverse loss as $\varepsilon\to0$: Theorem~\ref{thm:bias} bounds its deterministic bias, while Theorem~\ref{thm:statistical} controls sampling error and gives an explicit regularization schedule.
\end{itemize}
\noindent Code and data reproducing all numerical figures, together with exploratory notebooks and interactive experiments, will be available at \url{https://github.com/gpeyre/iot-curvature}.

\subsection{Related works}
\paragraph{Entropic methods for inverse OT}
Inverse OT is a special case of inverse linear optimization \cite{zhang1996calculating,ahuja2001inverse}. A substantial line of work instead uses the entropic formulation, developed in economics and matching models through maximum likelihood \cite{galichon2010matching,galichon2022cupid,dupuy2019estimating}. The resulting objective is, up to constants independent of $A$, the gap~\eqref{eq:Leps} \cite{dupuy2019estimating,andradesparsistency}, whose convexity extends to any linear cost parametrization. This formulation supports computational methods for cost learning \cite{carlier2023sista,liu2019learning,ma2020learning,li2019learning} and statistical guarantees at fixed regularization, including sparse recovery \cite{andradesparsistency}. Fenchel--Young gap losses extend this approach to inverse problems over measures, including inverse unbalanced OT and inverse JKO learning \cite{andrade2025learning}.

These methods benefit from the computational and analytical theory of forward entropic OT. Sinkhorn iterations \cite{CuturiSinkhorn}, based on classical matrix scaling \cite{Sinkhorn64}, make the forward problem tractable; connections with Schr\"odinger bridges \cite{leonard2012schrodinger} and stochastic control \cite{nutz2021introduction} provide further interpretations. Our interest in vanishing regularization arises from the use of these established entropic methods to recover costs from unregularized transport data.

\paragraph{Identifiability from transport observations}
To understand that limit, one must first determine what the unregularized coupling identifies. The geometry of discrete inverse OT is studied in \cite{chiu2022discrete}. Other identifiability results use richer observations, including multiple marginals or value data \cite{gonzalez2024nonlinear} and collections of optimal maps on manifolds \cite{zhai2025inverse}. We study a single coupling and give conditions on its Monge map that ensure recovery of a bilinear cost up to positive scale. The discrete and elliptical examples clarify why uniqueness of the forward plan alone is insufficient.

\paragraph{Forward sensitivity and inverse stability}
Identifiability must be supplemented by quantitative stability to control recovery from approximate observations. Forward stability describes how changes in the marginals or cost affect optimal plans and potentials. For unregularized OT, estimates for changes in the marginals include \cite{merigot2020quantitative,delalande2021quantitative,letrouit2024gluing}, with related results for discretized plans \cite{li2021quantitative} and smooth density perturbations \cite{caja2026stability}. Stability of maps under perturbations of both the cost and densities is studied in \cite{chen2016stability}. For entropic OT, stability estimates are available for plans and potentials \cite{carlier2024displacement,eckstein2022quantitative,deligiannidis2024ipfp,nutz2023stability}. Our second-variation formula describes sensitivity to the cost and identifies the directions to which the observed plan is insensitive. For bilinear costs, positive curvature transverse to scaling yields the inverse estimate needed for stable recovery.

\paragraph{Sampling and vanishing regularization}
The statistical analysis combines this inverse estimate with forward approximation bounds. For unregularized OT, empirical cost estimates exhibit a strong dependence on dimension \cite{weed2019sharp}, while map-estimation rates depend on regularity of the Brenier potential \cite{deb2021rates,manole2021plugin,muzellec2021near}. At fixed entropic regularization, sample-complexity bounds are available for Sinkhorn divergences and OT values \cite{genevay2019sample,mena2019statistical}. Their dependence on $\varepsilon$ becomes essential when regularization vanishes: bounds for entropic map estimation can deteriorate exponentially \cite{rigollet2022sample}, while stronger smoothness assumptions yield sharper control \cite{pooladian2021entropic,divol2024tight}. We combine forward bounds with the curvature of the inverse loss to control regularization bias and sampling error jointly, leading to an explicit schedule $\varepsilon=\varepsilon_n\to0$.

\section{Curvature of optimal transport}\label{sec:curvature_OT}

We study the second-order structure of the optimal transport functional
\[
c\mapsto \OT_{\alpha,\beta}(c).
\]
While first-order properties of optimal transport are classical, its second-order structure with respect to the cost variable has received little attention. 
For inverse OT, this Hessian controls whether a perturbation of the cost is
visible from a fixed optimal plan.

\paragraph{Dual formulation of OT}
Define the semi-dual functional  
\begin{equation*}\label{eq:semidual}
    F(c,\phi) \coloneqq - \int_X \phi\,d\alpha \;+\;\int_Y \phi^c\,d\beta,
\qquad
\phi^c(y)\coloneqq\inf_{x\in X}\{c(x,y)+ \phi(x)\}. \tag{S}
\end{equation*}
Then,~\eqref{eq:OT} has the dual formulation
$$
\OT_{\alpha,\beta}(c) = \sup_{\phi\in C(\bar X)}F(c,\phi),
$$
and an optimal solution $\phi$ is a Kantorovich potential.

For $k\in\NN_0$ and $\kappa\in(0,1)$, we use the usual H\"older spaces $C^{k,\kappa}(\bar X)$, locally H\"older spaces $C^{k,\kappa}_{\mathrm{loc}}(\bar X)$ and their vector-valued counterparts, with $k$ specified in each result. Since $F(c,\phi)$ is invariant under addition of constants to $\phi$, we work on $C^{k,\kappa}(\bar X)/\RR$, equipped with the quotient norm
\[
\norm{\phi}_{C^{k,\kappa}/\RR}
\coloneqq\inf_{a\in\RR}\norm{\phi+a}_{C^{k,\kappa}}.
\]
We impose the following assumptions on the domains, densities and base cost $c_0$.

\begin{assumption}\label{ass:domain}
$X,Y$ are bounded, connected, uniformly convex open domains whose closures are compact and whose boundaries are of class $C^{k+2,\kappa}$.
\end{assumption}

\begin{assumption}\label{ass:densities}
$\alpha,\beta$ admit $C^{k,\kappa}$ densities on $\bar X$ and $\bar Y$, respectively, that are bounded above and bounded below by positive constants.
\end{assumption}
As a slight abuse of notation, since $\al,\beta$ are assumed to admit densities, we wrote $\alpha(x)$ and $\beta(y)$ to denote their densities.

We will establish regularity of the maps $c\mapsto \phi_c$ and $c \mapsto\OT(c)$ locally around a cost $c_0$ satisfying the bi-twist and nondegeneracy conditions:
\begin{assumption}\label{ass:twist}
 $c_0 \in C^{k+2,\kappa}(\bar X\times\bar Y)$ is bi-twisted and nondegenerate: for every $x\in\bar X$ the map $y\mapsto\partial_xc_0(x,y)$ is injective, for every $y\in\bar Y$ the map $x\mapsto\partial_yc_0(x,y)$ is injective, and $\partial^2_{x,y}c_0(x,y)$ is invertible on $\bar X\times\bar Y$.
\end{assumption}

For bilinear and quadratic costs, Assumptions~\ref{ass:domain} and~\ref{ass:densities} are standard in Caffarelli's global regularity theory \cite[Theorems~12.50 and~12.51]{villani2008optimal}. For general costs, we explicitly assume the required regularity of the primal and dual potentials at $c_0$. Proposition~\ref{prop:regularity_potential} then constructs optimal potentials for nearby costs using the implicit function theorem.

In Assumption~\ref{ass:twist}, forward twist ensures uniqueness of the optimal plan, dual twist defines the reverse contact map, and nondegeneracy gives local invertibility \cite[p.~234]{villani2008optimal}. For bilinear costs, all three conditions hold when $A$ is invertible. Under our density and connectedness assumptions, the Kantorovich potential is unique up to a constant \cite[Remark~10.30]{villani2008optimal}. We use $\phi$ for both its quotient class and its representative normalized by $\int\phi\,d\alpha=0$.

\paragraph{Main curvature result}

The following theorem establishes twice differentiability of $\OT$ under standard regularity assumptions and gives an explicit characterization of its second variation. 
This result applies to general smooth costs and is of independent interest. In the following, $T^c$ and $\phi_c$ denote the optimal map and the \textit{zero-mean} Kantorovich potential associated with $c$.

\begin{thm}[Curvature of OT for general costs] \label{thm:curvature}
Assume~\ref{ass:domain}–\ref{ass:densities}, and let 
$c_0$ satisfy~\ref{ass:twist} with  $k\in\NN$, $k\ge 1$, and $\kappa\in(0,1)$. 
Let $\phi_0$ denote the mean-zero Kantorovich potential associated to $c_0$. 
Write $T_0\coloneqq T^{c_0}$ for the associated optimal map.
Assume $\phi_0\in  C^{k+2,\kappa}(\bar X)$, its $c_0$-transform $\phi_0^{c_0}$ belongs to $C^{k+2,\kappa}(\bar Y)$, and that there exist $\nu,\mu>0$ such that
\begin{equation}\label{ass:curved_cphi}
    \nu I_d
\preceq
\partial_{xx}c_0(x,T_0(x)) + \nabla^2\phi_0(x)
\preceq
\mu I_d.
\end{equation}
Assume additionally that the base duality gap has no contacts outside the graph, including at the boundary:
\begin{equation}\label{ass:strict-contact}
c_0(x,y)+\phi_0(x)-\phi_0^{c_0}(y)>0
\quad\text{for }(x,y)\in\bar X\times\bar Y\text{ with }y\ne T_0(x).
\end{equation}
Fix $0<\theta<\kappa$. For costs $c$ in a sufficiently small neighborhood of $c_0$ (which may depend on $\theta$), let $T^c$ and $\phi_c$ denote the optimal map and normalized Kantorovich potential.
Then, the following hold.
\begin{enumerate}
\item[(i)] $c\mapsto\OT_{\alpha,\beta}(c)$ is $ C^2$ in a $C^{k+2,\kappa}$-neighborhood of $c_0$;
\item[(ii)] the map $c \mapsto \phi_c$ is $ C^1$ as a map from this $C^{k+2,\kappa}$-neighborhood to $C^{k+1,\theta}(\bar X)/\RR$;
\item[(iii)] for any perturbation $\delta c\in C^{k+2,\kappa}(\bar X\times\bar Y)$,
$$
    \partial_{cc}\OT_{\al,\beta}(c)[\delta c,\delta c] = \sup_{\psi\in  C^{k+1,\theta}(\bar X)/\RR}-\int\norm{\partial_x \delta c(x,T^c(x)) + \nabla\psi(x)}^2_{(\partial_{xx} c(x,T^c(x))+\nabla^2\phi_c(x))^{-1}}\,d\alpha(x)
    $$
    In particular,
$\partial_{cc}\OT_{\alpha,\beta}(c)[\delta c,\delta c]=0$ 
if and only if the vector field
$x \mapsto \partial_x \delta c(x,T^c(x))$
is a gradient field on $X$.
\end{enumerate}

\end{thm}

\begin{rem}
 For a bilinear base cost $c_0(x,y)=-x^\top A_0y$, \eqref{ass:strict-contact} is a consequence of \eqref{ass:curved_cphi}: since $\partial_{xx}c_0=0$, ~\eqref{ass:curved_cphi} implies that $\phi_0$ uniformly convex on $\bar X$. Writing $S_0\coloneqq T_0^{-1}$, stationarity and integration along segments in the convex domain give
\[
c_0(x,y)+\phi_0(x)-\phi_0^{c_0}(y)
\ge\frac\nu2\lVert x-S_0(y)\rVert^2
\qquad((x,y)\in\bar X\times\bar Y).
\]
\end{rem}

The curvature formula identifies its null directions as cost perturbations whose $x$-gradient is a gradient field along the optimal map. The following corollary shows that these directions preserve optimality of the map under sufficiently small perturbations of the cost.

\begin{cor}[Plan-specific zero-curvature directions]\label{cor:zero-curvature-directions}
Assume the hypotheses of Theorem~\ref{thm:curvature}, and fix a cost
\(c\) in the neighborhood where the theorem applies. Let
\(\delta c\in C^{k+2,\kappa}(\bar X\times\bar Y)\) satisfy
\[
\partial_x\delta c(x,T^c(x))=\nabla\psi(x)
\]
for some \(\psi\). Set \(c_\epsilon\coloneqq c+\epsilon\delta c\) and
\(\phi_\epsilon\coloneqq\phi_c-\epsilon\psi\). Then, for all sufficiently small
\(|\epsilon|\), \(\phi_\epsilon\) is \(c_\epsilon\)-convex and \(T^c\)
is an optimal transport map from \(\alpha\) to \(\beta\) for the cost
\(c_\epsilon\).
\end{cor}
\begin{proof}
 Write $T\coloneqq T^c$ and $S\coloneqq T^{-1}$. The base regularity and twist assumptions give a diffeomorphism $T:\bar X\to\bar Y$, as recalled in Lemma~\ref{lem:Tmap}.
Recall that $T$, $c$ and $\phi_c$ are related by
$
\partial_x  c(x, T(x)) = -\nabla \phi_c(x).
$
If $\partial_x  \delta c(x,T(x)) = \nabla \psi(x)$ is a gradient field, then the perturbed cost  $c_\epsilon=  c+ \epsilon \delta c$  and perturbed potential   $\phi_\epsilon = \phi_c -\epsilon \psi$ satisfy
\begin{equation}\label{eq:tmp_opt_alt}
\partial_x  c_\epsilon(x,T(x)) = -\nabla \phi_\epsilon(x).
\end{equation}
Note also that the potential $\psi$ is $C^2$ up to the boundary, since its gradient is $\partial_x\delta c(\cdot,T(\cdot))$.

We now show that $T$ remains the transport map for cost $c_\epsilon$. It is enough to show that $\phi_\epsilon$ is $c_\epsilon$-convex in the sign convention of \eqref{eq:semidual}, meaning that there exists a function $\psi_\epsilon$ such that
\[
\phi_\epsilon(x)=\sup_{y\in Y}\{\psi_\epsilon(y)-c_\epsilon(x,y)\},
\]
with equality at $y=T(x)$.
To this end, define
\[
\psi_\epsilon(y) \eqdef \inf_{x\in\bar X}\{\phi_\epsilon(x)+c_\epsilon(x,y)\}.
\]
At $\epsilon=0$, the Hessian in $x$ of $x\mapsto\phi_c(x)+c(x,y)$ at $x=S(y)$ is
\[
\nabla^2\phi_c(S(y))+\partial_{xx}c(S(y),y),
\]
which is uniformly positive definite by~\eqref{ass:curved_cphi} on the graph of $T$.
Hence, for all $|\epsilon|$ small enough, the Hessian in $x$ of $x\mapsto\phi_\epsilon(x)+c_\epsilon(x,y)$ is uniformly positive definite in a fixed neighborhood of $x=S(y)$.
Moreover, by~\eqref{eq:tmp_opt_alt},
\[
\partial_xc_\epsilon(S(y),y)+\nabla\phi_\epsilon(S(y))=0
\]
for all $y\in\bar Y$.
It follows that $S(y)$ is a strict local minimizer of $x\mapsto\phi_\epsilon(x)+c_\epsilon(x,y)$, uniformly in $y$.
Choose $r>0$ small enough so that this local minimality holds on $B_r(S(y))\cap\bar X$, uniformly in $y$ and small $|\epsilon|$. Convexity of $\bar X$ allows integration of the Hessian along segments, including when $S(y)\in\partial X$.
To see that it is a global minimizer, notice that
\[
c_\epsilon(x,y)+\phi_\epsilon(x)-\bigl(c_\epsilon(S(y),y)+\phi_\epsilon(S(y))\bigr)
\]
is strictly bounded away from zero at $\epsilon=0$ on the compact set
\[
\{(x,y)\in\bar X\times\bar Y:\lVert x-S(y)\rVert\ge r\},
\]
if nonempty, by~\eqref{ass:strict-contact}. Uniform convergence preserves positivity there for all sufficiently small $|\epsilon|$.
Thus $S(y)$ is the unique minimizer of $x\mapsto\phi_\epsilon(x)+c_\epsilon(x,y)$ on $\bar X$.

Hence, for all $(x,y)\in\bar X\times\bar Y$,
\[
\psi_\epsilon(y)=\phi_\epsilon(S(y))+c_\epsilon(S(y),y)
\le\phi_\epsilon(x)+c_\epsilon(x,y),
\]
with equality at $y=T(x)$. Therefore,
\[
\phi_\epsilon(x)=\sup_{y\in\bar Y}\{\psi_\epsilon(y)-c_\epsilon(x,y)\}.
\]
Since $T_\#\alpha=\beta$, equality along $T$ gives
\[
\int_Xc_\epsilon(x,T(x))\,d\alpha(x)
=-\int_X\phi_\epsilon\,d\alpha+\int_Y\psi_\epsilon\,d\beta,
\]
so primal--dual equality proves optimality.

\end{proof}

\paragraph{Overview} The remainder of this section is dedicated to the proof of Theorem~\ref{thm:curvature}. 
To prove this result, we first provide the first variation formula for OT. Then, we establish that $c\mapsto \phi_c$ is $ C^1$ locally. Finally, in Section~\ref{sec:proof_curvature}, we complete the proof of Theorem~\ref{thm:curvature}  by deriving the second variation formula for OT with respect to $c$.

\subsection{First order variation}
We  prove in the following proposition that the OT value is Frechet differentiable at each Hölder cost under the sole assumption that it generates a unique optimal plan. 
\begin{prop}[First variation]\label{prop:first_variation_ot}
    Let $c_0\in  C^{0,\kappa}(X\times Y)$ be a Hölder continuous cost function such that there exists a unique plan $\pi_0$ for the transport between $\alpha$ and $\beta$. Then, the map $c\mapsto \OT_{\alpha,\beta}(c)$ is Frechet differentiable at $c_0$, and for every $h\in C^{0,\kappa}(X\times Y)$, 
    $$
    \partial_c\OT_{\al,\beta}(c_0)[h] = \int h (x,y) d\pi_0(x,y)
    $$
    Moreover, if $\pi_0=(Id,T^{c_0})\#\alpha$ where $T^{c_0}$ is the  optimal map for $c_0$, then 
    $$
    \partial_c\OT_{\al,\beta}(c_0)[h] = \int h (x,T^{c_0}(x)) d\alpha(x)
    $$
\end{prop}
\begin{proof}
Let $\pi_0$ be the unique optimal plan for $c_0$, and, for each
$h\in C^{0,\kappa}(\bar X\times\bar Y)$, let $\pi_h$ be an optimal plan for $c_0+h$.

By compactness of $\Pi(\alpha,\beta)$ and uniqueness of $\pi_0$ as an OT plan for $c_0$, we have $\pi_h\rightharpoonup\pi_0$ as
 $\|h\|_{C^{0,\kappa}}\to0$.

Optimality of $\pi_h$ and $\pi_0$ gives
\[
\int h\,d(\pi_h-\pi_0)
\le
\mathrm{OT}_{\alpha,\beta}(c_0+h)
-\mathrm{OT}_{\alpha,\beta}(c_0)
-\int h\,d\pi_0
\le 0.
\]
Consequently,
\[
\left|
\mathrm{OT}_{\alpha,\beta}(c_0+h)
-\mathrm{OT}_{\alpha,\beta}(c_0)
-\int h\,d\pi_0
\right|
\le
\|h\|_{C^{0,\kappa}}
\sup_{\|u\|_{C^{0,\kappa}}\le1}
\left|\int u\,d(\pi_h-\pi_0)\right|.
\]
Since $\kappa>0$, the unit ball of
$C^{0,\kappa}(\bar X\times\bar Y)$ is relatively compact in the
uniform norm. Hence weak convergence of $\pi_h$ to $\pi_0$ is
uniform over this ball, and the right-hand side is
$o\bigl(\|h\|_{C^{0,\kappa}}\bigr)$.
\end{proof}


\subsection{Regularity of Kantorovich potentials}

In this section, we show that $c\mapsto \phi_c$ is $ C^1$ locally.

\begin{prop}[Regularity of Kantorovich potentials]\label{prop:regularity_potential}

     Assume~\ref{ass:domain} and~\ref{ass:densities} with $k\in\NN$, $k\ge 1$,  and $\kappa\in(0,1)$, and let $c_0\in C^{k+2,\kappa}(\bar X\times\bar Y)$ satisfy Assumption~\ref{ass:twist}. Let $T_0$ be the optimal map associated with $c_0$. Let $\phi_0\in C^{k+2,\kappa}(\bar X)/\RR$ be the unique maximizer of $F(c_0,\cdot)$. Assume that $\phi_0^{c_0}\in C^{k+2,\kappa}(\bar Y)$ and that there exist $\nu,\mu>0$ such that
    $$
    \nu I_d\preceq \partial_{xx}c_0(x,T_0(x)) + \nabla^2\phi_0(x)\preceq \mu I_d.
    $$ 
    Assume also the strict-contact condition~\eqref{ass:strict-contact} on $\bar X\times\bar Y$, and fix $0<\theta<\kappa$.
    There exist a neighborhood $\Vv\subset C^{k+2,\kappa}(\bar X\times\bar Y)$ of $c_0$ and a neighborhood $\Ww\subset C^{k+1,\theta}(\bar X)/\RR$ of $\phi_0$ such that:
    \begin{enumerate}
        \item for every $c\in\Vv$, the functional $F(c,\cdot)$ has a unique maximizer $\phi_c\in\Ww$, and the map $c\mapsto\phi_c$ belongs to $C^1(\Vv,\Ww)$;
        \item the associated optimal map $T^c$ extends to a $C^{k,\theta}$ diffeomorphism from $\bar X$ onto $\bar Y$.
    \end{enumerate}
\end{prop}

The proof follows the classical PDE viewpoint on the second boundary value problem for Monge--Amp{\`e}re type equations \cite{urbas1997second,trudinger2009second}.  In perturbative regimes, Chen and Figalli \cite{chen2016stability} show that smoothness of optimal maps is stable under small perturbations of the cost and of the densities. These results provide the classical regularity background for the base solution and for nearby smooth solutions. What is needed here is slightly different: differentiability of the potential with respect to the cost parameter. For this we use a \textit{Banach-space implicit function theorem} applied directly to the nonlinear transport equation with its natural boundary condition. Our proof is close in spirit to the recent work of Gonz{\'a}lez-Sanz and Sheng \cite{gonzalez_sanz_sheng_2024_linearization}, who prove differentiability of optimal maps with respect to perturbations of the marginals by linearizing the Monge--Amp{\`e}re equation with natural oblique boundary conditions. Our argument uses a similar linearized-oblique-operator mechanism, but the parameter being varied is the cost rather than the target density.

We next construct a diffeomorphism $T_{c,\phi}$ given functions $c$ and $\phi$. We will later establish that the IFT branch is globally $c$-convex and hence coincides with the optimal transport map.
\begin{lem}\label{lem:Tmap}
Let $c_0$ and $\phi_0$ satisfy the hypotheses of Proposition~\ref{prop:regularity_potential}, and fix $0<\theta<\kappa$.
Then there exists a neighborhood $\Vv\times\Ww\subset C^{k+2,\kappa}(\bar X \times \bar Y)\times \bigl(C^{k+1,\theta}(\bar X)/\RR\bigr)$ of $(c_0,\phi_0)$ such that for every $(c,\phi)\in \Vv\times \Ww$, there exists a map $T_{c,\phi}$ that is a $C^{k,\theta}$ embedding of $\bar X$ into $\RR^d$ and a diffeomorphism between $X$ and its image, such that for all $x\in\bar X$,
$$
\partial_x c(x,T_{c,\phi}(x)) + \nabla \phi(x) = 0.
$$
Moreover, $(c,\phi)\mapsto T_{c,\phi}$ is Fr\'echet $C^1$ with values in $C^{k,\theta}(\bar X;\RR^d)$.
\end{lem}
\begin{proof}
Under the hypotheses of Proposition~\ref{prop:regularity_potential}, standard $c$-exponential theory applied to the primal and dual potentials gives mutually inverse forward and reverse optimal maps; see \cite[Section~2, especially (2.2) and (2.6)--(2.7)]{chen2016stability}. The assumed H\"older regularity, bi-twist, and the uniformly positive graph Hessian then show from the stationary equation that $T_0$ extends to a $C^{k+1,\kappa}$ diffeomorphism from $\bar X$ onto $\bar Y$.

Use a fixed bounded linear extension of the cost to a slightly larger product of bounded convex domains containing $\bar X\times\bar Y$; such H\"older extensions exist by~\cite[Lemma~6.37]{gilbarg_trudinger_2001}. This allows nearby maps to take values outside $\bar Y$ before the boundary condition is imposed. On an open neighborhood of $(c_0,\phi_0,T_0)$, define
\[
\begin{split}
&\mathcal F:\;C^{k+2,\kappa}(\bar X\times\bar Y)
\times\bigl(C^{k+1,\theta}(\bar X)/\RR\bigr)
\times C^{k,\theta}(\bar X;\RR^d)\longrightarrow C^{k,\theta}(\bar X;\RR^d),
\\& \mathcal F(c,\phi,T)\coloneqq\partial_xc(\cdot,T)+\nabla\phi,
\end{split}
\]
where $\partial_xc(\cdot,T)$ refers to the extended cost $c$.

We first check that $\Ff$ is Fréchet differentiable.
By Lemma~\ref{lem:holder-composition-gap} with $m=k$, applied to the graph map $x\mapsto(x,T(x))$, for $a\in C^{k,\kappa}$ we have locally
\[
\|a(\cdot,T+V)-a(\cdot,T)\|_{C^{k,\theta}}
\le C\|a\|_{C^{k,\kappa}}\|V\|_{C^{k,\theta}}^{\kappa-\theta}.
\]
Taking $a=\partial_{xy}c$ for the extended cost, Taylor's integral formula and the H\"older product estimate therefore give
\[
\begin{split}
\|\partial_xc(\cdot,T+V)-\partial_xc(\cdot,T)
-\partial_{xy}c(\cdot,T)V\|_{C^{k,\theta}} &= \norm{ \int_0^1 (\partial_{xy} c(\cdot, T+sV) - \partial_{xy}c(\cdot, T)) V  \mathrm{d}s }_{C^{k,\theta}}\\
&\qquad\le C\|V\|_{C^{k,\theta}}^{1+\kappa-\theta}
=o(\|V\|_{C^{k,\theta}}).
\end{split}
\]
The strict inequality $\theta<\kappa$ makes the remainder superlinear.
The same estimates, uniformly for cost increments in bounded subsets of $C^{k+2,\kappa}$, give joint Fr\'echet $C^1$ regularity and
\[
D\mathcal F(c,\phi,T)[h,\xi,V]
=\partial_xh(\cdot,T)+\nabla\xi+\partial_{xy}c(\cdot,T)V.
\]

At the base point $(c_0,\phi_0,T_0)$, we have $\Ff(c_0,\phi_0,T_0) = 0$ and $\partial_T \Ff(c_0,\phi_0,T_0)[V] = \partial_{xy} c_0(\cdot,T_0(\cdot))V$. By the bi-twist assumption,
 multiplication by $\partial_{xy}c_0(\cdot,T_0)$ is an isomorphism on $C^{k,\theta}$. The implicit function theorem therefore gives the locally unique $C^1$ map $(c,\phi)\mapsto T_{c,\phi}$. It is well defined on the quotient because $\mathcal F$ depends only on $\nabla\phi$. Since $k\ge1$, closeness in $C^{k,\theta}$ implies closeness in $C^1$. Stability of embeddings then proves the claimed geometric properties.
\end{proof}

\begin{proof}[Proof of Proposition~\ref{prop:regularity_potential}]
    Let $T_{c,\phi}$ be the map defined in Lemma~\ref{lem:Tmap}. Since $Y$ is uniformly convex and has $C^{k+2,\kappa}$ boundary, choose a uniformly convex defining function $\rho_Y\in C^{k+2,\kappa}$ on a neighborhood of $\bar Y$ such that
    \[
    Y=\{\rho_Y<0\},
    \qquad
    \rho_Y=0\ \text{ and }\ \nabla\rho_Y\ne0\quad\text{on }\partial Y.
    \]
    On the neighborhoods supplied by Lemma~\ref{lem:Tmap}, define
    \[
    \begin{split}
    \Hh:\;&C^{k+2,\kappa}(\bar X\times\bar Y)
    \times\bigl(C^{k+1,\theta}(\bar X)/\RR\bigr)\times\RR\\
    &\longrightarrow C^{k-1,\theta}(\bar X)\times C^{k,\theta}(\partial X)
    \end{split}
    \]
    by
    \[
    \Hh(c,\phi,\lambda)
    \coloneqq\left(\sigma_0\beta(T_{c,\phi})\det DT_{c,\phi}-\lambda\alpha,
    \;\rho_Y(T_{c,\phi})\restriction_{\partial X}\right).
    \]
    where we extend $\beta$ outside $Y$ with the same H\"older regularity. Here $\sigma_0\coloneqq\mathrm{sign}\det DT_0$ is constant, and $\sigma_0\det DT_{c,\phi}=|\det DT_{c,\phi}|$ for nearby maps.

    We verify that $\Hh$ is $C^1$ in these spaces. For $T\in C^{k,\theta}$, the determinant map $T\mapsto\det DT$ is a polynomial in first derivatives and hence smooth into $C^{k-1,\theta}$. Since the extended density satisfies $\nabla\beta\in C^{k-1,\kappa}$, Lemma~\ref{lem:holder-composition-gap} with $m=k-1$ and Taylor's integral formula yield
    \[
    \begin{split}
    &\|\beta(T+V)-\beta(T)-(\nabla\beta\circ T)\cdot V\|_{C^{k-1,\theta}}\le C\|V\|_{C^{k,\theta}}^{1+\kappa-\theta}
    =o(\|V\|_{C^{k,\theta}}).
    \end{split}
    \]
    Thus $\theta<\kappa$ again gives differentiability, and the composition estimate also gives continuity of the derivative. Likewise $T\mapsto\rho_Y(T)|_{\partial X}$ is $C^1$ into $C^{k,\theta}(\partial X)$, because $\rho_Y\in C^{k+2,\kappa}$. Products in the H\"older spaces are continuous, and Lemma~\ref{lem:Tmap} gives $C^1$ dependence of $T_{c,\phi}$. The chain rule therefore proves the required joint $C^1$ regularity of $\Hh$.

    Observe that $\Hh(c_0,\phi_0, 1) = 0$. We will apply the implicit function theorem to obtain a map $c\mapsto (\phi_c, \lambda_c)$, then show that $\lambda_c =1$ and $\phi_c$ is the maximizer of $F(c,\cdot)$. 
    The defining function $\rho_Y$ produces an oblique boundary condition. This idea goes back to the regularity theory of Urbas and Trudinger--Wang \cite{urbas1997second,trudinger2009second}, and the present proof is inspired by \cite{gonzalez_sanz_sheng_2024_linearization}.

\medskip

\noindent\textbf{First variation of $T_{c,\phi}$:} Given a perturbation $\dot \phi \in \Cc^{k+1,\theta}(\bar X)/\RR$, let $\dot T \coloneqq \left.\frac{d}{dt}\right|_{t=0}T_{c,\phi+t \dot \phi}$.
    Since $\partial_x c(x,T_{c,\phi + t\dot \phi}(x))  + \nabla (\phi +t\dot \phi)(x) = 0$, we have
    \begin{align*}
        &\partial_{xy} c(x, T_{c,\phi}(x)) \dot T + \nabla \dot \phi(x) = 0\\
        &\partial_{11} c(x, T_{c,\phi}(x)) + \partial_{xy} c(x,T_{c,\phi}(x))  DT_{c,\phi}(x) + \nabla^2  \phi(x) = 0,
    \end{align*}
    where the first equation is obtained by differentiating the stationary equation with respect to $t$
 and the second is by differentiating with respect to $x$. 

 Letting $\Aa_{c,\phi}(x) \coloneqq\partial_{11} c(x, T_{c,\phi}(x)) + \nabla^2  \phi(x) $, and since this is invertible uniformly on $x$ at $c_0, \phi_0$ , we have
 \begin{equation}\label{eq:dotT}
 \dot T_0 = D T_{c_0,\phi_0} \Aa_{c_0,\phi_0}^{-1} \nabla \dot \phi(x).
 \end{equation}
Denote $V_0(x)\eqdef \Aa_{c_0,\phi_0}^{-1} \nabla \dot \phi(x)$.
 
\medskip

\noindent\textbf{First variation of $\Hh$:} We now compute at $(c_0, \phi_0,\lambda)$,
$$
\partial_{\phi,\lambda}\Hh(c_0,\phi_0,\lambda)[\dot \phi,\dot \lambda] = \pa{  \sigma_0\partial_\phi( \beta\circ T_{c_0,\phi_0} \det(D T_{c_0,\phi_0}) )[\dot \phi] -  \dot \lambda \alpha  ,\; \partial_\phi(\rho_Y \circ T_{c_0,\phi_0}\restriction_{\partial X})[\dot \phi] }.
$$

Note that at $c_0, \phi_0$, $T_{c_0,\phi_0} = T_0$ is an optimal transport map and satisfies $T_0(\partial X) = \partial Y$. So, taking any parameterization  $\gamma_X$ of $\partial X$, for all parameter values $s$, $\rho_Y(T_0(\gamma_X(s))) = 0$. Differentiating in $s$ gives
\begin{equation}\label{eq:normal}
    DT_0^\top \nabla \rho_Y(T_0(\gamma_X(s))) \cdot \gamma_X'(s) = 0.
    \end{equation}
So, for all $x\in \partial X$,
$DT_0^\top \nabla \rho_Y(T_0(x)) = b_0(x)n_X(x)$, where $n_X$ is the outward normal to $\partial X$.  Moreover,
$b_0\in C^{k,\kappa}(\partial X)$ and $b_0(x)>0$.  Indeed, $T_0$ maps interior points to $Y=\{\rho_Y<0\}$, so the derivative in the outward normal direction has positive sign; the invertibility of $DT_0$ and $\nabla\rho_Y\ne0$ exclude vanishing.  Compactness therefore gives $b_0\ge b_*>0$.

 It follows from~\eqref{eq:dotT} and~\eqref{eq:normal} that
 $$
 \partial_\phi(\rho_Y \circ T_{c,\phi}\restriction_{\partial X})\big|_{(c_0,\phi_0)}[\dot \phi]  = \dot T_0^\top \nabla \rho_Y(T_0) = b_0 (\Aa_{c_0,\phi_0}^{-1} \nabla \dot \phi(x))^\top n_x.
 $$

 For the first part of the variation,
 \begin{align*}
     \partial_\phi( \beta\circ T_{c_0,\phi_0} \det(D T_{c_0,\phi_0}) )[\dot \phi] &= \left.\frac{d}{dt}\right|_{t=0} ( \beta(T_0+ t\dot T) \det(D (T_0+ t\dot T)) )\\
     &= \det(D (T_0))\nabla \beta(T_0) \cdot \dot T + \beta(T_0)\left.\frac{d}{dt}\right|_{t=0} \det(D (T_0+ t\dot T)) \\
     &=\det(D (T_0))\nabla \beta(T_0) \cdot \dot T + \beta(T_0) \det(DT_0) \tr(DT_0^{-1} D\dot T),
 \end{align*}
 and in the last line, we can further expand by differentiating~\eqref{eq:dotT}.  Explicitly, for $v\in\RR^d$,
 \[
 D\dot T_0(x)[v]
 =D^2T_0(x)[v,V_0(x)]+DT_0(x)DV_0(x)[v].
 \]
	 Finally, recalling that $\alpha = \sigma_0 \beta(T_0) \det(DT_0)$ and differentiating this in direction $V_0$, we have
	 $$
	 \nabla \alpha \cdot V_0 = \sigma_0 \nabla \beta(T_0)\cdot DT_0 V_0 \det(DT_0) + \sigma_0\beta(T_0) \det(DT_0) \tr\!\left( DT_0^{-1} D^2 T_0[\,\cdot\,,V_0]\right).
 $$
 It follows that $\sigma_0 \partial_\phi( \beta\circ T_{c_0,\phi_0} \det(D T_{c_0,\phi_0}) )[\dot \phi] = \mathrm{div}(\alpha V_0) = \mathrm{div}(\alpha \Aa_{c_0,\phi_0}^{-1} \nabla \dot\phi)$.

 To summarise,
 $$
	 \partial_{\phi,\lambda}\Hh(c_0,\phi_0,\lambda)[\dot \phi,\dot \lambda] = \pa{ \mathrm{div}(\alpha \Aa_{c_0,\phi_0}^{-1} \nabla \dot\phi) -  \dot \lambda \alpha ,\; b_0 n_x^\top \Aa_{c_0,\phi_0}^{-1} \nabla \dot \phi }.
 $$
 
\medskip
 \noindent\textbf{Isomorphism of $\partial_{\phi,\lambda}\Hh(c_0,\phi_0,\lambda)$:} It remains to check that $\partial_{\phi,\lambda}\Hh(c_0,\phi_0,\lambda)$ is an isomorphism from $C^{k+1,\theta}(\bar X)/\RR \times  \RR$ to $C^{k-1,\theta}(\bar X)\times C^{k,\theta}(\partial X)$. Since multiplication of the boundary component by the non-vanishing function $\alpha/b_0$ is an isomorphism on the boundary target space, it is enough to check the following rescaled operator:
	 \begin{equation}\label{eq:linearized_conormal_operator}
	 \Gg(\dot \phi,\dot \lambda) \coloneqq\pa{\mathrm{div}(\alpha \Aa_{c_0,\phi_0}^{-1} \nabla \dot\phi) -  \dot \lambda  \alpha ,\; \alpha n_x^\top\Aa_{c_0,\phi_0}^{-1} \nabla \dot \phi }.
		 \end{equation}
	 Put
	 \[
	 B(x)\coloneqq \alpha(x)\Aa_{c_0,\phi_0}(x)^{-1}.
	 \]
	 By the lower and upper bounds on $\Aa_{c_0,\phi_0}$ and on $\alpha$, $B$ is uniformly elliptic. Moreover, by~\ref{ass:densities} and the assumed regularity of $c_0$ and $\phi_0$, we have $B\in C^{k,\kappa}(\bar X;\RR^{d\times d})$.
	 \begin{itemize}
	     \item Injectivity. Suppose $\Gg(\dot \phi,\dot\lambda) = 0$. Integrating the first component over $X$ and using the vanishing conormal component gives
	     \[
	     0=\int_X \mathrm{div}(B\nabla\dot\phi)\,dx-\dot\lambda\int_X\alpha\,dx
	     =-\dot\lambda\int_X\alpha\,dx,
	     \]
	     hence $\dot\lambda=0$. Multiplying the first equation by $\dot\phi$ and integrating by parts then yields
	     \[
	     0
	     =-\int_X \nabla\dot\phi^\top B\nabla\dot\phi\,dx
	     +\int_{\partial X}\dot\phi\, B\nabla\dot\phi\cdot n_X\,dS
	     =-\int_X \nabla\dot\phi^\top B\nabla\dot\phi\,dx.
	     \]
	     Uniform ellipticity gives $\nabla\dot\phi=0$, so $\dot\phi$ is constant and therefore vanishes in the quotient $C^{k+1,\theta}(\bar X)/\RR$.
	     \item Surjectivity. Given $f\in C^{k-1,\theta}(\bar X)$ and $g\in C^{k,\theta}(\partial X)$, choose
	 \[
	 s \coloneqq\frac{ \int_{\partial X} g\,dS - \int_X f\,dx}{\int_X \alpha\,dx}.
	 \]
	 Then the data $(f+s\alpha,g)$ satisfy the compatibility condition
	 \[
	 \int_X (f+s\alpha)\,dx=\int_{\partial X}g\,dS.
	 \]
	 It remains to solve
	     \begin{equation}\label{eq:conormal_problem}
	     \begin{cases}
	                       \mathrm{div}(B\nabla \xi)= f+s\alpha , & \text{in }X,\\
	       B\nabla \xi \cdot n_X  = g, &\text{on } \partial X.
	             \end{cases}
	          \end{equation}
		 We justify this by the method of continuity. For $t\in[0,1]$, set
	 \[
	 B_t\coloneqq(1-t)I+tB,\qquad L_tu\coloneqq\mathrm{div}(B_t\nabla u),\qquad N_tu\coloneqq B_t\nabla u\cdot n_X.
	 \]
	 The matrices $B_t$ are uniformly elliptic and uniformly bounded in $C^{k,\kappa}$, and the boundary fields $B_tn_X$ are uniformly oblique since
	 \[
	 (B_tn_X)\cdot n_X=n_X^\top B_t n_X\ge \vartheta>0
	 \]
	 with $\vartheta$ independent of $t$. The Schauder estimates \cite[Theorem 6.30]{gilbarg_trudinger_2001}, see also \cite{agmon1959estimates}, for uniformly oblique boundary problems therefore give, uniformly in $t$,
	 \begin{equation}\label{eq:uniform_schauder}
	 \|u-\bar u\|_{C^{k+1,\theta}(\bar X)}
	 \le C\left(
	 \|L_tu\|_{C^{k-1,\theta}(\bar X)}
	 +
	 \|N_tu\|_{C^{k,\theta}(\partial X)}
	 \right),
	 \end{equation}
		 where $\bar u\coloneqq|X|^{-1}\int_Xu\,dx$. Indeed, the usual estimate contains an additional $\|u\|_{C^0}$ term; on the quotient by constants this term is removed by the standard compactness argument, see for example \cite[Theorem 4.1]{nardi2015schauder} or \cite[Lemma 3]{peetre1961another}, because the homogeneous problem $L_tu=0$, $N_tu=0$ has only constants in its kernel. At $t=0$, solvability is precisely the Neumann problem for the Laplacian, which follows from \cite{nardi2015schauder}. The standard method of continuity now applies: openness follows from continuity in operator norm and stability of invertibility under sufficiently small operator perturbations, while closedness follows from the uniform estimate~\eqref{eq:uniform_schauder}. Therefore~\eqref{eq:conormal_problem} is solvable at $t=1$, uniquely modulo constants. Hence $\Gg(\xi,s)=(f,g)$.
	 \end{itemize}
 By the implicit function theorem, there exists a locally unique $C^1$ map $c\mapsto (\phi_c,\lambda_c)$ from a neighborhood of $c_0$ in $C^{k+2,\kappa}$ to a neighborhood of $(\phi_0,1)$ in $(C^{k+1,\theta}/\RR)\times\RR$ such that 
 $$
 \Hh(c, \phi_c,\lambda_c) = 0.
 $$
 
Recall from Lemma~\ref{lem:Tmap} that $T_{c,\phi_c}$ is an embedding and a diffeomorphism onto its image. The boundary component of $\Hh=0$ gives $T_{c,\phi_c}(\partial X)\subset\partial Y$. For $d\ge2$, invariance of domain on the connected manifold $\partial Y$ shows that the image of the boundary embedding is both open and closed in $\partial Y$, and hence equals $\partial Y$; for $d=1$ the same conclusion is immediate from closeness to $T_0$. The Jordan--Brouwer separation theorem then gives $T_{c,\phi_c}(\bar X)=\bar Y$. Thus $T_{c,\phi_c}:\bar X\to\bar Y$ is a diffeomorphism.

Using $\sigma_0\det DT_{c,\phi_c}=|\det DT_{c,\phi_c}|$, the first component of $\Hh=0$ and the change-of-variables formula give
\[
1=\int_X\beta(T_{c,\phi_c}(x))|\det DT_{c,\phi_c}(x)|\,dx
=\lambda_c\int_X\alpha(x)\,dx=\lambda_c.
\]
Hence $\lambda_c=1$, and the same identity shows that $(T_{c,\phi_c})_{\#}\alpha=\beta$.

It remains to show that  $T_c\coloneqq T_{c,\phi_c}$ is indeed an optimal transport map. For this, we can apply the same argument as in Corollary~\ref{cor:zero-curvature-directions}, since the argument  there allows one to conclude that  for $c$ and $\phi_c$ sufficiently close to $c_0$ and $\phi_0$ in $C^2$, and $S_c=T_c^{-1}$ sufficiently close to $S_0$ uniformly, if $(T_c)_\sharp \alpha = \beta$ with $\partial_x c(x,T_c(x)) + \nabla \phi_c(x) = 0$, then $T_c$ is an optimal transport map.


Indeed, the local-minimality and compactness argument in the proof of Corollary~\ref{cor:zero-curvature-directions} also applies to any nearby cost and potential converging in $C^2$ to $(c_0,\phi_0)$, with stationary inverse maps $S_c\to S_0$ uniformly. In that case, take the compact set $\lVert x-S_0(y)\rVert\ge r/2$ at the base cost: for nearby $c$, it contains all pairs with $\lVert x-S_c(y)\rVert\ge r$. Thus the same argument establishes unique global contact along the varying map, using only stationarity, convergence, and the base Hessian and strict-contact assumptions. 
 \end{proof}

\subsection{Second order variation}\label{sec:proof_curvature}
\begin{proof}[Proof of Theorem~\ref{thm:curvature}]
Recall from Proposition~\ref{prop:first_variation_ot} that
$$
\partial_c \OT(c_0)[h] = \int h(x,T_0(x)) d\alpha(x).
$$
Fix $0<\theta<\kappa$. By Lemma~\ref{lem:Tmap} and Proposition~\ref{prop:regularity_potential}, $c\mapsto T_c$ is Fr\'echet $C^1$ from a neighborhood in $C^{k+2,\kappa}$ into $C^{k,\theta}$. Evaluation of $h$ along $T_c$ and integration against $\alpha$ are continuously differentiable in these spaces, uniformly for $h$ in bounded subsets of $C^{k+2,\kappa}$. Hence the first-derivative map
\[
c\longmapsto D\OT(c),\qquad
D\OT(c)[h]=\int_Xh(x,T_c(x))\,d\alpha(x),
\]
is Fr\'echet $C^1$ as a map into the dual of $C^{k+2,\kappa}(\bar X\times\bar Y)$.  In particular,
\[
D^2\OT(c)[h_1,h_2]
=\int_X \partial_yh_1(x,T_c(x))\cdot D_cT_c[h_2](x)\,d\alpha(x)
\]
is a continuous bilinear form depending continuously on $c$. 
Let $S_c\coloneqq T_c^{-1}$ and write $S_0\coloneqq S_{c_0}$. We differentiate inverse maps only in the uniform norm: the $C^1$ dependence of $T_c$ into $C^{k,\theta}$, with $k\ge1$, and uniform invertibility of $DT_c$ give
\[
D_cS_c[h]=-(DT_c\circ S_c)^{-1}(D_cT_c[h]\circ S_c)
\quad\text{in }C^0(\bar Y;\RR^d).
\]
To work out the variational formula, 
\begin{align*}
    \partial_{cc}\OT(c_0)[h,h] = \int \partial_x h(S_0(y),y) \dot S(y) d\beta(y) 
\end{align*}
where $\dot S \coloneqq \frac{d}{dt}_{\restriction t=0} S_{c_0+th}$. By differentiating $\partial_x(c_0+th)(S_{c_0+th}(y),y) + \nabla \phi_{c_0+th}(S_{c_0+th}(y)) = 0$ with respect to $t$ at $t=0$, we obtain
$$
\dot S(y) = -\Aa_{c_0,\phi_0}(S_0(y))^{-1}(\nabla \dot \phi(S_0(y)) + \partial_x h(S_0(y),y)),
$$
where $\dot \phi \coloneqq \frac{d}{dt}_{\restriction t=0} \phi_{c_0+th}$. Plugging this into the expression for $\partial_{cc}\OT$, we obtain
\begin{align*}
    \partial_{cc}\OT(c_0)[h,h] &= -\int \partial_x h(S_0(y),y)  \Aa_{c_0,\phi_0}(S_0(y))^{-1}(\nabla \dot \phi(S_0(y)) + \partial_x h(S_0(y),y)) d\beta(y) \\
    &=-\int \partial_x h(x,T_0(x))  \Aa_{c_0,\phi_0}(x)^{-1}(\nabla \dot \phi(x) + \partial_x h(x,T_0(x))) d\alpha(x)
\end{align*}

Now, recall the IFT function $\Hh$ from the proof of Proposition~\ref{prop:regularity_potential} that in a neighbourhood around $c_0$,
$$
\Hh(c, \phi_c, 1) = 0.
$$
Differentiating this with respect to $c$ in direction $h$ at $c_0$ yields
$$
\partial_c \Hh(c_0,\phi_0,1)[h] + \partial_\phi \Hh(c_0,\phi_0,1)[\dot \phi] = 0.
$$
By following the argument in the proof of Proposition~\ref{prop:regularity_potential}, this implies that
\begin{equation}\label{eq:linearized_pde}
    \begin{cases}
    \mathrm{div}(\alpha \Aa_{c_0,\phi_0}^{-1} (\partial_x h(x,T_0(x)) + \nabla \dot \phi(x))) = 0 &\text{in }X,\\
    n_x \cdot \Aa_{c_0,\phi_0}^{-1} (\partial_x h(x,T_0(x)) + \nabla \dot \phi(x)) = 0 &\text{on }\partial X.
    \end{cases}
\end{equation}
Thus, multiplying the interior PDE equation by $\dot \phi$ and integrating on $X$ yields
$$
\partial_{cc}\OT(c_0)[h,h] =  -\int \norm{\nabla \dot \phi(x) + \partial_x h(x,T_0(x))}_{\Aa_{c_0,\phi_0}^{-1}}^2 d\alpha(x).$$

Finally, to see the equivalence to the variational form, the Euler Lagrange equation for
$$
\inf_{\psi\in H^1/\RR} \int \norm{\nabla  \psi(x) + \partial_x h(x,T_0(x))}_{\Aa_{c_0,\phi_0}^{-1}}^2 d\alpha(x)
$$
is, for all $\xi \in H^1(X)/\RR$,
$$
\int  \nabla \xi \cdot \pa{ \Aa_{c_0,\phi_0}^{-1}(\nabla  \psi(x) + \partial_x h(x,T_0(x)))} d\alpha(x) = 0
$$
and~\eqref{eq:linearized_pde} shows that its weak solution is $\dot\phi\in C^{k+1,\theta}/\RR$. Thus the infimum is attained in the test space appearing in the theorem. The same variational identity holds at every nearby cost $c$: differentiating
$\int_Y f\,d\beta=\int_X f(T_c(x))\,d\alpha(x)$ for $f\in C^1(\bar Y)$ gives the weak linearized transport equation. With $f=\xi\circ S_c$, stationarity yields
\[
\int_X\nabla\xi\cdot\Aa_{c,\phi_c}^{-1}
\bigl(\partial_xh(\cdot,T_c)+\nabla D_c\phi_c[h]\bigr)\,d\alpha=0.
\]
Initially this holds for $\xi\in C^1(\bar X)$ and then for $\xi\in H^1(X)$ by density. It identifies $D_c\phi_c[h]\in C^{k+1,\theta}/\RR$ as the minimizer and gives the claimed formula throughout the neighborhood. In particular, zero curvature is equivalent to the stated gradient-field condition.
\end{proof}

\section{Well-posedness of iOT solutions for the bilinear cost}\label{sec:uniqueness}

We now address the inverse optimal transport problem of recovering a bilinear cost while observing an optimal plan.
For a given $A$ and the bilinear cost $c_A(x,y) = -x^\top A y$, let
\[
\pi_0(A)\coloneqq 
\mathop{\argmin}_{\pi\in\Pi(\alpha,\beta)} 
\int c_A\,d\pi.
\]
Given $\hat\pi\in \pi_0(\hat A)$,  the admissible set of cost matrices is defined as
\begin{equation*}
    \Ss_0(\hat \pi)
    \eqdef
    \{A : \hat \pi \in \pi_0(A) \}.
    \tag{$\Ss_0$}
\end{equation*}
Positive rescaling does not change optimal transport plans, and the zero cost makes every coupling optimal. Hence
$\{\lambda\hat A:\lambda\geq0\}\subset\Ss_0(\hat\pi)$.
In this section, we show that if the induced transport map has sufficient curvature, then
\begin{enumerate}
    \item the full admissible set is the nonnegative ray
    $\Ss_0(\hat\pi)=\{\lambda\hat A:\lambda\geq0\}$; equivalently, every nonzero admissible cost is recoverable up to positive scale. See Theorem~\ref{thm:uniqueness}.
\item the set $\Ss_0(\hat \pi)$ is stable. See Theorem~\ref{thm:local_curvature_L0} and Corollary~\ref{cor:stabilty_S0}.
\end{enumerate}
In general, $\pi_0(A)$ is a set, but in Section~\ref{sec:brenier}, we will recall conditions under which the transport plan is unique and hence $\pi_0(A)$ is a singleton. Our discussion on the uniqueness of the iOT problem will be specialized to the setting where the forward problem has a unique solution.

\subsection{Uniqueness}

A unique forward transport plan need not identify its cost: discrete and elliptical examples are given in Sections~\ref{sec:discrete} and~\ref{sec:Gaussian}. Theorem~\ref{thm:uniqueness} gives a sufficient condition in terms of the almost-everywhere Hessian of a convex Brenier potential. We work with finite-second-moment measures on $\RR^d$ under the following assumptions. Proposition~\ref{prop:generic_identifiability} then shows how small smooth perturbations can restore the spanning condition.

\begin{defn}[Essential Hessian spanning]\label{def:essential-hessian-spanning}
Let $X\subset\RR^d$ be nonempty and open, and suppose that $\varphi:X\to\RR$ is twice differentiable in the Alexandrov sense almost everywhere on $X$, with a positive-semidefinite Hessian. We say that the Hessian of $\varphi$ essentially spans $\mathbb{S}_d \eqdef \enscond{B\in\RR^{d\times d}}{B^\top=B}$ if
\begin{equation}\label{eq:hessian-span-symmetric}
\Span\{\nabla^2\varphi(x):x\in X\setminus N\}=\mathbb{S}_d
\end{equation}
for every Lebesgue-null set $N\subset X$ containing the points where $\varphi$ is not twice differentiable in the Alexandrov sense. Thus the Hessians must span $\mathbb S_d$ after any further null set is discarded.
\end{defn}

\begin{thm}[Identifiability from a Brenier map]\label{thm:uniqueness}
Let $\alpha,\beta\in\mathcal P_2(\RR^d)$.  Suppose that
$X\coloneqq\operatorname{int}(\operatorname{Spt}\alpha)$ is nonempty,
$\operatorname{Spt}\alpha=\bar X$, and $\alpha$ has a density that is positive almost everywhere on $X$.
Let $\hat A\in GL_d(\RR)$, and let
$\hat\pi=(\Id,T_{\hat A})_\#\alpha$ be the corresponding optimal coupling for
$c_{\hat A}(x,y)=-x^\top\hat Ay$.  Assume that
the map has the Brenier representation for a convex function $\Psi_{\hat A}:\RR^d\to\RR$,
\[
T_{\hat A}(x)=\nabla\Psi_{\hat A}(\hat A^\top x),
\qquad
\text{for $\alpha$-almost every }x.
\]
Set $\phi_{\hat A}(x)\coloneqq\Psi_{\hat A}(\hat A^\top x)$, so that $\nabla\phi_{\hat A}=\hat A T_{\hat A}$ for $\alpha$-almost every $x$. Assume that its almost-everywhere Alexandrov Hessian essentially spans $\mathbb S_d$ in the sense of Definition~\ref{def:essential-hessian-spanning}.
Then $\Ss_0(\hat\pi)=\{\lambda\hat A:\lambda\geq0\}$.
\end{thm}

\begin{rem}[One-dimensional identifiability]\label{rem:oned-identifiability}
When $d=1$ and $\beta$ also has a Lebesgue density, condition~\eqref{eq:hessian-span-symmetric} for $\phi_{\hat A}$ is automatic under the remaining assumptions of Theorem~\ref{thm:uniqueness}. Indeed, writing $f,g$ for the densities of $\alpha,\beta$, we have
\[
g(T_{\hat A}(x))\,|T_{\hat A}'(x)|=f(x)>0
\qquad\text{for a.e. }x\in X.
\]
Since $\phi_{\hat A}''=\hat A T_{\hat A}'\ge0$ by convexity, it follows that $\phi_{\hat A}''>0$ almost everywhere on $X$, so its values essentially span $\mathbb S_1=\RR$. In this setting, the problem of cost identifiability is only non-trivial for $d\geq 2$.
\end{rem}

For Gaussian marginals, the Brenier potential is quadratic, so its Hessian is constant and~\eqref{eq:hessian-span-symmetric} fails for $d\ge2$. Section~\ref{sec:Gaussian} illustrates this degeneracy and how perturbing the target can restore identifiability. We now recall the Brenier representation and prove the theorem.

\subsubsection{Preliminaries on the Brenier map}\label{sec:brenier}

We use $\nabla^2u$ for the Alexandrov Hessian of a convex function $u$. Its gradient belongs to $BV_{\mathrm{loc}}$ on the interior of its effective domain, and its distributional Hessian is a matrix-valued Radon measure with decomposition
\[
D^2u=\nabla^2u\,\mathcal L^d+(D^2u)^s.
\]
In particular, the density of the absolutely continuous part is the almost-everywhere Hessian; see \cite[Theorem~3.2 and the discussion following it]{dephilippis2014monge}. Below, we differentiate almost-everywhere identities distributionally and then take their absolutely continuous parts.

Bilinear transport reduces to quadratic transport by applying $A$ to the target. This keeps the source absolutely continuous even when $A$ is singular.

\begin{prop}[Brenier representation for bilinear costs]\label{prop:brenier-map}
Let $\alpha,\beta\in\mathcal P_2(\RR^d)$ with $\alpha\ll\mathcal L^d$, and let $A\in\RR^{d\times d}$. Let $\nabla\phi_A$ be the Brenier map from $\alpha$ to $A_\#\beta$, for a proper lower-semicontinuous convex potential $\phi_A$.
A coupling $\pi\in\Pi(\alpha,\beta)$ is optimal for $c_A$ if and only if
\[
Ay=\nabla\phi_A(x)\qquad\text{for $\pi$-almost every }(x,y).
\]
In particular, any optimal map $T$ satisfies $AT=\nabla\phi_A$ $\alpha$-almost everywhere. If $A$ is invertible, the unique optimal coupling is induced by
\[
T_A(x)\coloneqq A^{-1}\nabla\phi_A(x)=\nabla\Psi_A(A^\top x),
\]
where $\nabla\Psi_A$ is the Brenier map from $(A^\top)_\#\alpha$ to $\beta$.
\end{prop}
\begin{proof}
Put $\nu_A\coloneqq A_\#\beta$. Projecting a coupling by $(\Id,A)$ preserves its bilinear cost, and
\[
-x\cdot z=\frac12\|x-z\|^2-\frac12\|x\|^2-\frac12\|z\|^2.
\]
The last two terms have fixed integrals under couplings of $\alpha$ and $\nu_A$.
Conversely, disintegrate $\beta=\int\beta_z\,d\nu_A(z)$, with $\beta_z$ supported on $\{y:Ay=z\}$. Every $\gamma\in\Pi(\alpha,\nu_A)$ lifts to the coupling
\[
\pi_\gamma(dx,dy)\coloneqq\int\beta_z(dy)\,\gamma(dx,dz)
\]
of $\alpha$ and $\beta$, with the same bilinear cost. Thus $\pi$ is optimal for $c_A$ if and only if $(\Id,A)_\#\pi$ is optimal for quadratic transport from $\alpha$ to $\nu_A$. Brenier's theorem \cite[Theorem~10.28]{villani2008optimal} gives the unique reduced optimal coupling $(\Id,\nabla\phi_A)_\#\alpha$; no density assumption on $\nu_A$ is needed. This proves the characterization for every $A$.

If $A$ is invertible, the characterization gives the unique map $T_A=A^{-1}\nabla\phi_A$. Alternatively, changing source variables to $z=A^\top x$ reduces the cost to $-z\cdot y$, and quadratic transport from $(A^\top)_\#\alpha$ to $\beta$ gives $T_A(x)=\nabla\Psi_A(A^\top x)$.
\end{proof}

\begin{rem}\label{rem:build-ot-convex}
For invertible $A$, the same reduction constructs optimal maps: if $\Psi$ is convex and $T_A(x)=\nabla\Psi(A^\top x)$ has a finite second moment under $\alpha$, then $T_A$ is optimal from $\alpha$ to $(T_A)_\#\alpha$ for $c_A$.
\end{rem}

\begin{prop}[Regularity]\label{prop:regularity_OT}
    Suppose~\ref{ass:domain} holds and  $\al,\beta$ satisfy~\ref{ass:densities} with $k\in\NN_0$.
    Given an invertible matrix $A_0$, the Brenier potential $\Psi_{A_0}$ satisfies the following properties
    \begin{itemize}
        \item[i)] $\Psi_{A_0}\in\Cc^{k+2,\kappa}(\overline{A_0^\top X})$,
        \item[ii)] $\nu_{A_0}I_d\preceq \nabla^2\Psi_{A_0}(A_0^\top x)\preceq \mu_{A_0} I_d$ for all $x\in X$ where $\nu_{A_0},\mu_{A_0}>0$.
        \item[iii)] $T_{A_0}$ extends to a $C^{k+1,\kappa}$ diffeomorphism from $\bar X$ onto $\bar Y$.
    \end{itemize}
    
\end{prop}
\begin{proof}
For i), Proposition~\ref{prop:brenier-map} identifies $\Psi_{A_0}$ as the Brenier potential for quadratic transport from $(A_0^\top)_\#\alpha$ to $\beta$. Assumptions~\ref{ass:domain} and~\ref{ass:densities} and \cite[Theorem~12.50]{villani2008optimal} give $\Psi_{A_0}\in C^{k+2,\kappa}(\overline{A_0^\top X})$.

For ii), the Monge-Ampere equation for $T_A\#\alpha=\beta$ with $T_{A_0}(x) = \nabla\Psi_{A_0}(A_0^\top x)$ leads to 
$$
|\det A_0|\det\nabla^2\Psi_{A_0}(A_0^\top x) = \frac{\alpha(x)}{\beta(\nabla\Psi_{A_0}(A_0^\top x))}
$$
Since the densities are bounded from above and below, we get $C\ge \det\nabla^2\Psi_{A_0}(A_0^\top x)\ge c$ for all $x\in X$ where $C,c>0$. The $C^2$ regularity on the compact set gives a uniform upper bound on all eigenvalues of $\nabla^2\Psi_{A_0}$; combining this bound with the positive determinant lower bound gives a uniform positive lower bound on every eigenvalue.

For iii), the lower Hessian bound makes $z\mapsto\nabla\Psi_{A_0}(z)$ injective on the convex domain $A_0^\top X$, while the Jacobian is everywhere invertible. Since $(T_{A_0})_\#\alpha=\beta$ and both densities have full support on their domains, continuity gives $T_{A_0}(\bar X)=\bar Y$. The inverse function theorem, compactness, and the regularity in i) give the claimed global diffeomorphism.

\end{proof}
\begin{rem}[Relation between the potentials]\label{rem:kanto-brenier-link}
For invertible $A$, the Kantorovich and Brenier potentials satisfy
\[
\phi_A(x)=\Psi_A(A^\top x),
\qquad
\nabla^2\phi_A(x)=A\nabla^2\Psi_A(A^\top x)A^\top
\quad\text{for a.e. }x\in X.
\]
Thus Proposition~\ref{prop:regularity_OT} also gives the H\"older regularity and positive Hessian bounds for $\phi_A$. Since $\partial_{xx}c_A=0$, these bounds, together with the corresponding dual regularity, verify the potential assumptions of Theorem~\ref{thm:curvature} when $k\ge1$. Moreover, congruence by $A$ is an automorphism of $\mathbb S_d$, and $x\mapsto A^\top x$ preserves null sets. Hence the Hessian of $\phi_A$ essentially spans $\mathbb S_d$ on $X$ if and only if that of $\Psi_A$ does so on $A^\top X$.
\end{rem}

\subsubsection{Proof of Theorem~\ref{thm:uniqueness}} \label{sec:identifiability-proof}
\begin{proof}[Proof of Theorem~\ref{thm:uniqueness}]
Let $A\in\Ss_0(\hat\pi)$ be arbitrary and set $C\coloneqq A\hat A^{-1}$. By Proposition~\ref{prop:brenier-map}, there is a proper convex potential $\phi_A$ such that
\begin{equation}\label{eq:grad-equality}
\nabla\phi_A=AT_{\hat A}=C\nabla\phi_{\hat A}
\qquad\text{for $\alpha$-almost every }x.
\end{equation}
Since $\alpha$ has a positive density almost everywhere on $X$, this identity holds Lebesgue-almost everywhere there. The effective domain of $\phi_A$ contains a full-measure subset of $X$, so convexity places $X$ in its interior. Both potentials are therefore locally Lipschitz on $X$, and differentiating~\eqref{eq:grad-equality} distributionally gives
\[
D^2\phi_A=C\,D^2\phi_{\hat A}.
\]
Put $H(x)\coloneqq\nabla^2\phi_{\hat A}(x)$. Taking the absolutely continuous parts and using convexity yields
\begin{equation}\label{eq:commutation}
CH(x)=H(x)C^\top,\qquad CH(x)\succeq0
\qquad\text{for a.e. }x\in X.
\end{equation}
By essential Hessian spanning, the symmetry identity implies
\[
CS=SC^\top\qquad\text{for every }S\in\mathbb S_d.
\]
Taking $S=I_d$ makes $C$ symmetric. Commutation with diagonal matrices then makes $C$ diagonal, and commutation with $e_ie_j^\top+e_je_i^\top$ makes its diagonal entries equal. Hence $C=\lambda I_d$ and $A=\lambda\hat A$ for some $\lambda\in\RR$.

If $\lambda<0$, the inequalities $H\succeq0$ and $\lambda H\succeq0$ from~\eqref{eq:commutation} force $H=0$ almost everywhere, contradicting essential spanning. Thus $\lambda\ge0$. Conversely, positive scaling preserves optimality, and every coupling is optimal for the zero cost. Therefore
\[
\Ss_0(\hat\pi)=\{\lambda\hat A:\lambda\ge0\}.
\]
\end{proof}

We next construct smooth perturbations that enforce the spanning condition.
\begin{prop}[Compact perturbations on $\RR^d$]
\label{prop:generic_identifiability}
Let $k\in\NN_0$ and $\kappa\in(0,1)$.  Let $\alpha\in\mathcal P_2(\RR^d)$ have a strictly positive $C^{k,\kappa}_{\mathrm{loc}}(\RR^d)$ density, and let $A\in GL_d(\RR)$. Suppose
\[
T_A(x)=\nabla\Psi_A(A^\top x),
\qquad
\Psi_A\in C^{k+2,\kappa}_{\mathrm{loc}}(\RR^d),
\qquad
\nabla^2\Psi_A\succeq\nu I_d
\]
for some $\nu>0$, and assume that $\beta\coloneqq(T_A)_\#\alpha$ belongs to $\mathcal P_2(\RR^d)$.
Then there exist $h\in C_c^\infty(\RR^d)$ and $\varepsilon_0>0$ such that, for
\[
\Psi_\varepsilon\coloneqq\Psi_A+\varepsilon h,
\qquad
T_\varepsilon(x)\coloneqq\nabla\Psi_\varepsilon(A^\top x),
\qquad
\beta_\varepsilon\coloneqq(T_\varepsilon)_\#\alpha,
\]
the following hold whenever $|\varepsilon|<\varepsilon_0$:
\begin{enumerate}
    \item $\Psi_\varepsilon$ is uniformly convex, $T_\varepsilon$ is a global $C^{k+1,\kappa}_{\mathrm{loc}}$ diffeomorphism of $\RR^d$, and it is the optimal map from $\alpha$ to $\beta_\varepsilon$ for $c_A$;
    \item $\beta_\varepsilon\in\mathcal P_2(\RR^d)$ and has a strictly positive $C^{k,\kappa}_{\mathrm{loc}}(\RR^d)$ density;
    \item
    \[
    \Ww_2(\beta,\beta_\varepsilon)
    \le |\varepsilon|\,\|\nabla h(A^\top\cdot)\|_{L^2(\alpha)}.
    \]
\end{enumerate}
For all but finitely many $\varepsilon\in(-\varepsilon_0,\varepsilon_0)$, one additionally has
\[
\Span\left\{\nabla^2\Psi_\varepsilon(A^\top x):x\in\RR^d\right\}
=\mathbb S_d.
\]
Consequently, every Wasserstein neighborhood of $\beta$ contains a target marginal with a strictly positive $C^{k,\kappa}_{\mathrm{loc}}(\RR^d)$ density for which the cost is identifiable up to positive scale.
\end{prop}
\begin{proof}
Let $p\coloneqq\dim\mathbb S_d=d(d+1)/2$. Choose distinct points $z_1,\ldots,z_p\in\RR^d$ and a basis $M_1,\ldots,M_p$ of $\mathbb S_d$. Take pairwise disjoint balls $B_i$ centered at $z_i$ and cutoffs $\chi_i\in C_c^\infty(B_i)$ satisfying $\chi_i=1$ near $z_i$. Define
\[
h_i(z)\coloneqq\frac12\chi_i(z)(z-z_i)^\top M_i(z-z_i),
\qquad
h\coloneqq\sum_{i=1}^p h_i.
\]
Then $h\in C_c^\infty(\RR^d)$ and $\nabla^2h(z_i)=M_i$ for every $i$; the disjoint supports eliminate all cross terms.

Put $C_h\coloneqq\sup_z\|\nabla^2h(z)\|_{\mathrm{op}}$. After reducing $\varepsilon_0$ so that $\varepsilon_0C_h<\nu/2$, we have
\[
\nabla^2\Psi_\varepsilon
=\nabla^2\Psi_A+\varepsilon\nabla^2h
\succeq\frac\nu2 I_d
\qquad\text{on }\RR^d.
\]
Thus $\nabla\Psi_\varepsilon$ is strongly monotone and locally invertible. Moreover,
\[
\langle\nabla\Psi_\varepsilon(z)-\nabla\Psi_\varepsilon(0),z\rangle
\ge\frac\nu2\|z\|^2,
\]
so it is proper. A proper local diffeomorphism of $\RR^d$ that is injective by strong monotonicity is a global diffeomorphism onto $\RR^d$. Hence the same is true of $T_\varepsilon=\nabla\Psi_\varepsilon\circ A^\top$. Its Brenier form proves optimality for $c_A$.

Change of variables gives
\[
\beta_\varepsilon(y)
=
\frac{\alpha(T_\varepsilon^{-1}(y))}
{\left|\det DT_\varepsilon(T_\varepsilon^{-1}(y))\right|}.
\]
This density is strictly positive and belongs to $C^{k,\kappa}_{\mathrm{loc}}(\RR^d)$. Since $\nabla h$ is bounded,
$T_\varepsilon-T_A=\varepsilon\nabla h(A^\top\cdot)$ belongs to $L^2(\alpha)$; therefore $\beta_\varepsilon\in\mathcal P_2(\RR^d)$. The coupling $(T_A,T_\varepsilon)_\#\alpha$ gives the displayed Wasserstein bound.

For spanning, put $x_i\coloneqq A^{-\top}z_i$ and $H_i\coloneqq\nabla^2\Psi_A(z_i)$. For a linear isomorphism $L:\mathbb S_d\to\mathbb R^p$, define
\[
Q(\varepsilon)
\coloneqq
\bigl[L(H_1+\varepsilon M_1)\mid\cdots\mid
L(H_p+\varepsilon M_p)\bigr].
\]
The determinant of $Q(\varepsilon)$ is a polynomial of degree at most $p$ whose leading coefficient is
\[
\det\bigl[L(M_1)\mid\cdots\mid L(M_p)\bigr]\ne0.
\]
It has only finitely many roots. For every other $\varepsilon$, the matrices
$\nabla^2\Psi_\varepsilon(A^\top x_i)=H_i+\varepsilon M_i$ form a basis of $\mathbb S_d$.  Continuity of these Hessians makes this pointwise spanning equivalent to~\eqref{eq:hessian-span-symmetric}. Choosing a nonexceptional $\varepsilon$ arbitrarily close to zero proves the final assertion by Theorem~\ref{thm:uniqueness}.
\end{proof}

\begin{rem}[Bounded domains]
The identifiability theorem also applies on bounded domains under its convexity and essential spanning assumptions. For the smooth perturbation construction, Assumptions~\ref{ass:domain} and~\ref{ass:densities} supply the additional regularity needed here. Choose the points in $A^\top X$ and multiply the local quadratic perturbations by smooth cutoffs compactly supported away from the boundary. Then $T_\varepsilon=T_A$ near $\partial X$. For sufficiently small $|\varepsilon|$, uniform convexity and stability of embeddings preserve a diffeomorphism $\bar X\to\bar Y$; the change-of-variables formula preserves the positive $C^{k,\kappa}$ density bounds.
\end{rem}

\subsection{Stability}\label{sec:bilinear}
Theorem~\ref{thm:uniqueness}  establishes when the set of feasible costs $\Ss_0(\hat \pi)$ is precisely a ray.
We now specialize the curvature result of  Theorem~\ref{thm:curvature}  to understand the stability of the admissible set $\Ss_0(\hat\pi)$ for the bilinear family
$c_A(x,y)=-x^\top A y$.

Recall the unregularized gap loss $\Ll_0$ defined in~\eqref{eq:L0}. As discussed in the introduction, it is a convex loss in $A$ and its zero set is precisely the admissible set $\Ss_0(\hat\pi)$. In particular, $\Ll_0\geq 0 $ and $\Ll_0(A;\hat\pi) = 0 $
 iff $A \in\Ss_0(\hat \pi)$. Observe that, whenever it exists, the Hessian of $\Ll_0$ is precisely the negative Hessian of OT  with respect to the cost matrix $A$.
Throughout this subsection, matrix inner products and norms are Frobenius:
$\langle A,B\rangle_F\coloneqq\operatorname{Tr}(A^\top B)$ and
$\|A\|_F^2\coloneqq\langle A,A\rangle_F$; we suppress the subscript $F$.
Using the curvature theorem of OT, we obtain a quantitative stability estimate for this zero-temperature gap.

 \begin{thm}[Local quadratic stability]\label{thm:local_curvature_L0}
Assume $d\ge2$, that~\ref{ass:domain} and~\ref{ass:densities} hold with $k=1$, and let $A_0\in\Ss_0(\hat\pi)$ be invertible.
     Then there exist $r>0$, $\mu>0$ such that for all $\norm{A-A_0}\leq r$,
     $$
     \Ll_0(A;\hat \pi)-\Ll_0(A_0;\hat \pi) \geq \frac{m_*[A_0]}{8\mu}\min\pa{\norm{A}^2,\norm{A_0}^2}  \pa{1- \pa{\frac{\dotp{ A}{ A_0}}{\norm{ A}{\norm{A_0}}}}^2 }
     $$
     where 
$$\displaystyle m_*[A]\coloneqq\inf\enscond{f_A(B)}{\norm{B}=1, \dotp{B}{A}=0} 
$$
and
$$
    f_A(B)\coloneqq \inf_{\psi\in\Cc^{2,\theta}(\bar X)/\RR}\int \norm{\nabla\psi(x)-BT_A(x)}^2\;d\alpha(x).
    $$
     Moreover, if~\eqref{eq:hessian-span-symmetric} holds at $A_0$, then $m_*[A_0]>0$, i.e., there is curvature orthogonally to the ray $\Ss_0(\hat\pi)$.
 \end{thm}
For $d=1$, let $A_0\in\Ss_0(\hat\pi)$ be nonzero and $0<r<|A_0|$. Every $A\in B_r(A_0)$ is a positive multiple of $A_0$, so $\Ll_0(A;\hat\pi)=0$ and $\mathrm{dist}(A,\RR_+A_0)=0$; there is no transverse direction to estimate. The positive-semidefinite trace-one scalar parameter set is the singleton $\{1\}$.

 For $d\ge2$, if $m_*[A_0]>0$---in particular, under the spanning condition---then locally around $A_0$, the feasible set $\Ss_0(\hat\pi)$ reduces to a single nonnegative ray, and the optimality gap grows quadratically in directions transverse to scaling. As a corollary, we have the following result about the local stability of the set $\Ss_0(\hat \pi)$. We denote the cross covariance of $\pi_A$ by $\Sigma_A \coloneqq \int x y^\top d\pi_A(x,y)$.

\begin{cor}[Local isolation of the admissible ray]\label{cor:stabilty_S0}
Under the assumptions of Theorem~\ref{thm:local_curvature_L0}, assume that
$m_*[A_0]>0$. Then there exists $r>0$ such that the only admissible costs in
$B_r(A_0)$ are the nonnegative scalar multiples of $A_0$, i.e.
\[
\Ss_0(\hat\pi)\cap B_r(A_0)=\RR_+ A_0\cap B_r(A_0).
\]
Moreover, for every $A\in B_r(A_0)$,
\[
\mathrm{dist}(A,\RR_+ A_0)^2 \le C\,\Ll_0(A;\hat\pi),
\]
and hence, for every optimal plan $\pi_A$ associated with $A$,
\[
\mathrm{dist}(A,\RR_+ A_0)^2 \le C'\,\|\Sigma_A- \Sigma_{A_0}\|.
\]
\end{cor}
\begin{proof}
Reduce $r$ so that $r<\|A_0\|$.  If $\pi_A$ is an optimal plan for $c_A$, then $\Ll_0(A;\pi_A)=0$ and $\Ll_0(A_0;\hat\pi)=0$. Thus,
\begin{align*}
        \Ll_0(A,\hat \pi)-\Ll_0(A_0, \hat \pi)
       & =\Ll_0(A,\hat \pi)-\Ll_0( A,  \pi_A)\\
       & 
= \dotp{c_A}{\hat \pi - \pi_A} \leq \norm{A} \norm{\Sigma_A -  \Sigma_{A_0}}
\end{align*}
The result follows by combining with Theorem~\ref{thm:local_curvature_L0}, and observing that
$$
\mathrm{dist}(A, \RR_+ A_0)^2=\mathrm{dist}(A, \RR A_0)^2
= \norm{A}^2 \pa{1-\dotp{\frac{A}{\norm{A}}}{\frac{A_0}{\norm{A_0}}}^2},
$$
because the projection coefficient onto $A_0$ is positive on this ball.
\end{proof}

\subsubsection{Proof of Theorem~\ref{thm:local_curvature_L0}}\label{sec:proof_bilinear}

Since $\hat \pi$ is fixed, we simply write $\Ll_0(A) \coloneqq \Ll_0(A;\hat \pi)$.
 Observe that one can rewrite~\eqref{eq:L0} as
 $$ 
 \Ll_0(A) = \sup_{\pi\in\Pi(\alpha,\beta)}\dotp{A}{\Sigma_\pi - \Sigma_{\hat\pi}}\qwhereq \Sigma_\pi \coloneqq \int xy^\top\;d\pi(x,y) 
 $$
 This formulation makes it simple to state the following proposition. 
 \begin{prop}\label{prop:l0-c1}
     The map $A\mapsto\Ll_0(A)$ is convex on $\RR^{d\times d}$ and its subgradient is 
     $$
     \partial\Ll_0(A) = \enscond{\Sigma_{\pi_0(A)}-\Sigma_{\hat\pi}}{\pi_0(A)\in\underset{\pi\in\Pi(\alpha,\beta)}{\argmin}\dotp{c_A}{\pi}}
     $$
     When $A$ is invertible, $\pi_0(A)\in\argmin_{\pi\in\Pi(\alpha,\beta)}\dotp{c_A}{\pi}$ is unique and $\Ll_0$ is $\Cc^1$ around $A$.
 \end{prop}
\begin{proof}
    Since $A\mapsto c_A$ is linear, the convexity of the gap loss follows from the concavity of $c\mapsto \OT(c)$. The envelope theorem describes the subgradient of $\Ll_0$. To prove $\Cc^1$ regularity it remains to prove that $\pi_0(A)$ is unique and that $A\mapsto \Sigma_{\pi_0(A)}$ is continuous.

    If $A$ is invertible, $c_A$ satisfies the TWIST condition and therefore, $\pi_0(A)$ is unique so $\Ll_0$ is differentiable at $A$ with 
    $$
    \nabla\Ll_0(A) = \Sigma_{\pi_0(A)}-\Sigma_{\hat\pi}
    $$
    Take $A_n$ a sequence converging to $A$ on the open set of invertible matrices. $c_{A_n}$ satisfies the TWIST condition so there exists a unique optimal $\pi_n=\pi_0(A_n)$ minimizing $\dotp{c_{A_n}}{\pi}$ over $\pi\in\Pi(\alpha,\beta)$. According to \cite[Lemma 4.4]{villani2008optimal} the set $\Pi(\alpha,\beta)$ is sequentially weakly compact. Therefore, $\pi_n$ has at least one weak accumulation point denoted $\pi^*$. By weak convergence, $\pi^*$ satisfies 
    $$
    \dotp{c_A}{\pi^*}\le \dotp{c_A}{\pi}\quad \forall\, \pi\in\Pi(\alpha,\beta)
    $$
    Then, by uniqueness of the transport plan, every accumulation point of $\pi_n$ is equal to $\pi_0(A)$. Thus $\pi_0(A_n)$ converges weakly to $\pi_0(A)$ and 
    $$
    \Sigma_{\pi_0(A_n)}\to \Sigma_{\pi_0(A)}
    $$
    which translates to $\nabla\Ll_0(A_n)\to\nabla\Ll_0(A)$.
\end{proof}
    
 When $\al,\beta$ are discrete measures, $\Ll_0$ is piecewise linear. Moreover, since $\Ll_0(\lambda A_0) = 0$ for all $\lambda\geq0$ and $A_0 \in\Ss_0(\hat \pi)$, it is clear that $\Ll_0$ is flat along certain directions. Nonetheless, we show that under regularity assumptions on the marginals and the domain, $\Ll_0$ is curved in orthogonal directions to its zero set.

     


\begin{proof}[Proof of Theorem~\ref{thm:local_curvature_L0}]

We proceed in three steps.

\paragraph{Step 1: Hessian coercivity in orthogonal directions.}

Let $A_0\in S_0(\hat\pi)$. From Proposition~\ref{prop:regularity_OT} with $k=1$, the potential $\phi_{A_0}$ satisfies
\[
\nu_{A_0} I_d \preceq \nabla^2\phi_{A_0}(x)\preceq \mu_{A_0} I_d.
\]
Hence there exists a neighborhood $\Vv\subset \Cc^{3,\kappa}(\bar X\times\bar Y)$ of $c_{A_0}$ such that $c\mapsto\phi_c$ is $\Cc^1$ with values in $C^{2,\theta}(\bar X)/\RR$ for a fixed $0<\theta<\kappa$. Choosing $r>0$ small enough so that $\|A-A_0\|<r\Rightarrow c_A\in\Vv$ and every $A\in B_r(A_0)$ is invertible, Theorem~\ref{thm:curvature} ensures that $\Ll_0$ is $\Cc^2$ on $B_r(A_0)$.

Moreover, by continuity of $A\mapsto \phi_A$, up to reducing $r$, we have
\[
\frac{\nu_{A_0}}{2}I_d\preceq \nabla^2\phi_A(x)\preceq 2\mu_{A_0}I_d,
\]
so that for all $A\in B_r(A_0)$ and $H\in\RR^{d\times d}$,
\[
\begin{aligned}
D^2\Ll_0(A)[H,H]
&=\inf_{\psi}\int
\norm{\nabla\psi-HT_A}^2_{\nabla^2\phi_A^{-1}}\,d\alpha\\
&\ge\frac{1}{2\mu_{A_0}}
\inf_{\psi}\int\norm{\nabla\psi-HT_A}^2\,d\alpha.
\end{aligned}
\]

Since $AT_A=\nabla\phi_A$, adding a multiple of $A$ to $H$ leaves $f_A(H)$ unchanged. Hence, by homogeneity and the definition of $m_*[A]$,
\[
f_A(H)=f_A(P_A^\perp H)
\ge m_*[A]\norm{P_A^\perp H}^2,
\qquad
P_A^\perp H\coloneqq H-\frac{\langle H,A\rangle}{\norm A^2}A.
\]
This also holds when $P_A^\perp H=0$, since both sides vanish.

\paragraph{Step 2: Definition and properties of $m_*[A]$.}

Define
\[
f_A(B)= \inf_{[\psi]\in H^1(X)/\RR}\int \norm{\nabla\psi(x)-BT_A(x)}^2\,d\alpha(x),
\qquad
\Tt_A\coloneqq \{B:\;\ps{B}{A}=0,\;\norm{B}=1\},
\]
and
\[
m_*[A]=\inf_{B\in \Tt_A}f_A(B).
\]
This Sobolev infimum is attained.  Indeed, on $H^1(X)/\RR$ the bilinear form
\[
a(u,v)\coloneqq\int_X\nabla u\cdot\nabla v\,d\alpha
\]
is coercive by the positive density bound and Poincar\'e's inequality.  Lax--Milgram gives a unique class $[\psi_{A,B}]$ satisfying
\[
\int_X\nabla\psi_{A,B}\cdot\nabla\xi\,d\alpha
=\int_XBT_A\cdot\nabla\xi\,d\alpha
\quad\text{for every }[\xi]\in H^1(X)/\RR.
\]
The corresponding conormal Schauder estimate gives
$[\psi_{A,B}]\in C^{2,\kappa}(\bar X)/\RR$.  Thus this definition agrees with the H\"older-space infimum in the theorem.

The map $(A,B)\mapsto f_A(B)$ is continuous.  To see this, the $C^1$ dependence of the transport branch gives $T_{A_n}\to T_A$ in $L^2(\alpha)$ whenever $A_n\to A$.  Applying the preceding coercivity estimate to the difference of the two weak equations shows
$\nabla\psi_{A_n,B_n}\to\nabla\psi_{A,B}$ in $L^2(\alpha)$ when also $B_n\to B$, and substitution in the energy proves joint continuity.  Since $\Tt_A$ is compact, $m_*[A]$ is attained.

From Step 1 we obtain, for all $A\in B_r(A_0)$ and $H$,
\[
D^2\Ll_0(A)[H,H]
\ge
\frac{1}{2\mu_{A_0}}\|P_A^\perp H\|^2\, m_*[A].
\]

We next show that $A\mapsto m_*[A]$ is continuous. Let $A_n\to A$.  A sequence of minimizers $B_n\in\Tt_{A_n}$ has a convergent subsequence, and every limit lies in $\Tt_A$; joint continuity gives the lower bound for the limit. Conversely, for each $B\in\Tt_A$, projecting $B$ onto $\{A_n\}^{\perp}$ and normalizing gives $B_n\in\Tt_{A_n}$ with $B_n\to B$, which gives the upper bound. Therefore
\[
m_*[A_n]\to m_*[A].
\]

Assume now that $\Span\{\nabla^2\phi_{A_0}(x),\; x\in X\}=\mathbb{S}_d$. If $f_{A_0}(B)=0$ for some $B\in\Tt_{A_0}$, then the attained minimizer satisfies
\[
\nabla \psi_{A_0,B}(x)=BT_{A_0}(x).
\]
Differentiating gives
\[
\nabla^2\psi_{A_0,B}(x)=BA_0^{-1}\nabla^2\phi_{A_0}(x).
\]
By symmetry,
\[
BA_0^{-1}\nabla^2\phi_{A_0}(x)=\nabla^2\phi_{A_0}(x)(BA_0^{-1})^\top.
\]
By the spanning assumption this identity holds with $\nabla^2\phi_{A_0}(x)$ replaced by every $S\in\mathbb S_d$. Taking $S=I_d$ first shows that $BA_0^{-1}$ is symmetric; it then commutes with every symmetric matrix and must be a scalar multiple of $I_d$. Hence $B=\lambda A_0$, which contradicts $\ps{B}{A_0}=0$ and $\norm{B}=1$. Therefore $m_*[A_0]>0$.

\paragraph{Step 3: Lower bound on the loss.}

Let $A\in B_r(A_0)$ and define $A_t\coloneqq A_0+t(A-A_0)$. Since $\nabla\Ll_0(A_0)=0$,
\[
\Ll_0(A)-\Ll_0(A_0)
=
\int_0^1\int_0^t D^2\Ll_0(A_s)[A-A_0,A-A_0]\,ds\,dt.
\]
Using Step 2,
\[
D^2\Ll_0(A_s)[A-A_0,A-A_0]
\ge
\frac{1}{2\mu_{A_0}}\|P_{A_s}^\perp(A-A_0)\|^2\, m_*[A_s].
\]

From Lemma~\ref{lem:projection-at},
\[
\|P_{A_s}^\perp(A-A_0)\|^2
\ge
\min\{\|A\|^2,\|A_0\|^2\}
\left(1-\ps{\tfrac{A}{\|A\|}}{\tfrac{A_0}{\|A_0\|}}^2\right).
\]

By continuity of $m_*[A]$, up to reducing $r$, we have $m_*[A_s]\ge m_*[A_0]/2$. Combining the above estimates yields
\[
\Ll_0(A) - \Ll_0(A_0)
\ge
\frac{m_*[A_0]}{8\mu_{A_0}}
\min\{\|A\|^2,\|A_0\|^2\}
\left(1-\ps{\tfrac{A}{\|A\|}}{\tfrac{A_0}{\|A_0\|}}^2\right).
\]

Finally, $m_*[A_0]$ is invariant along the ray $S_0(\hat\pi)$ (since $T_{A_0}$ is), hence does not depend on the choice of $A_0$.

\end{proof}

\subsection{Beyond bilinear costs}
Extending the identifiability and stability results of this section beyond bilinear costs requires understanding conditions for the strict curvature of optimal transport. Denoting the gap as
$$
\ell(c;\hat \pi) \coloneqq \dotp{c}{\hat \pi} - \inf_{\pi \in \Pi(\al,\beta)}\dotp{c}{\pi},
$$
the set of costs compatible with $\hat\pi$, within the chosen linear cost family, is
$$
\Ss_0(\hat \pi) \eqdef \enscond{c}{\ell(c,\hat \pi) = 0}.
$$
For a compatible cost $c$, every admissible cost is collinear with $c$ if $\ell(c+\delta c;\hat\pi)>0$ for all $\delta c\not\parallel c$ in this family. By convexity of the gap, it is sufficient to show that $\nabla^2 \ell(c;\hat \pi)[\delta c,\delta c] = -\partial_{cc}\OT(c)[\delta c,\delta c]>0$ for all $\delta c \not \parallel  c$. The previous section showed that when restricted to bilinear costs, this condition is implied by the spanning condition \eqref{eq:hessian-span-symmetric}. Additive perturbations $a(x)+b(y)$ automatically lie in the curvature kernel. Although a detailed study of identifiability for general costs is beyond the scope of this article, we make two remarks here.
\begin{enumerate}
    \item Consider costs of the form $c(x,y) = h(x-y)$.  Then, $\partial_{cc}\OT(h)[\delta h,\delta h] = 0$ if and only if there exists $\psi$ such that
    $$
    \nabla \delta h(x- T(x)) = -\nabla \psi(x)
    $$
    which holds if and only if
    $$
     \nabla^2 \delta h(x- T(x)) DT(x) 
    $$
    is symmetric. Thus, the question of identifiability corresponds to understanding conditions on $T$ under which  $\nabla^2 \delta h(x-T(x)) DT(x) =DT(x)^\top  \nabla^2 \delta h(x-T(x))$ for all $x$ would imply that $\delta h$ is proportional to $h$ modulo affine functions.
    \item Suppose we are interested in parameterized costs of the form $c_\theta(x,y) \coloneqq  \sum_{i=1}^m \theta_i h_i(x,y)$. In this case, $\partial_{cc} \OT(c_{\theta})[c_{\delta \theta}, c_{\delta \theta}] = 0 $ if and only if 
    $$\sum_i \delta \theta_i \partial_{11} h_i(x, T(x)) +  \sum_i \delta \theta_i \partial_{12} h_i(x,T(x)) DT(x)
    $$
    is symmetric for all $x$.  Since $\sum_i \delta \theta_i \partial_{11} h_i(x, T(x))$ is symmetric, this is equivalent to saying that for all $x$,
    $$\sum_i \delta \theta_i  S_i(x) = 0,\qwhereq S_i(x)\eqdef \mathrm{Skew}(\partial_{12} h_i(x,T(x)) DT(x)),$$
    Identifiability, therefore, corresponds to understanding when this happens only for $\delta\theta \in \RR\theta$.
\end{enumerate}


\section{Entropic regularization and sample complexity}\label{sec:sample_setting}

We study recovery from an unregularized optimal coupling using the entropic gap loss. We fix the scale by imposing $\operatorname{tr}(A)=1$ on positive-semidefinite matrices. Section~\ref{subsec:entropic-value-bounds} collects the pointwise first-order formula and a uniform upper bound for the forward entropic OT bias. Section~\ref{subsec:well-posed-regularized-objective} explains the need for normalization and the log-determinant limit problem. 
 Finally, Sections~\ref{subsec:bias-full-data} and~\ref{subsec:variance_control} quantify the entropic bias and combine it with sampling error to obtain finite-sample recovery guarantees.

\paragraph{Entropic gap loss.}
Recall the entropic gap  from~\eqref{eq:Leps}, and write $\pi_\varepsilon(A)$ for the entropic optimal coupling for cost $c_A$. Define the forward value bias
\[
b_\varepsilon(A)\coloneqq\OT^\varepsilon_{\alpha,\beta}(c_A)-\OT_{\alpha,\beta}(c_A)\ge0,
\qquad \Ll_\varepsilon(A)=\Ll_0(A)-b_\varepsilon(A).
\]
For convenience,  we remove terms independent of  $A$ by setting:
\[
a_\varepsilon\coloneqq\frac d2\varepsilon\log(2\pi\varepsilon)
+\frac\varepsilon2\bigl(\KL(\alpha\mid dx)+\KL(\beta\mid dy)\bigr),
\qquad \bar\Ll_\varepsilon\coloneqq\Ll_\varepsilon-a_\varepsilon.
\]
This shift leaves minimizers unchanged.

\subsection{The entropic value bias}\label{subsec:entropic-value-bounds}
We quantify how entropy perturbs the forward optimal value, since this bias controls the distortion of the inverse loss as $\varepsilon\to0$.

\paragraph{Pointwise first-order formula.}
Under Assumptions~\ref{ass:domain} and~\ref{ass:densities} with $k\ge1$, for each fixed $A\succ0$, ~\cite[Theorem~1]{pal2019entropiccost} gives
\begin{equation}\label{eq:pointwise-pal}
b_\varepsilon(A)=\frac d2\varepsilon\log(1/\varepsilon)+\varepsilon g(A)+o_A(\varepsilon),
\end{equation}
where
\[
g(A)\coloneqq-\frac d2\log(2\pi)-\frac12\bigl(\KL(\alpha\mid dx)+\KL(\beta\mid dy)\bigr)+\frac12\log\det A.
\]
Indeed, this follows from the quadratic-cost formula of~\cite[Theorem~1]{pal2019entropiccost} applied to $(A^{1/2})_\#\alpha$ and $(A^{1/2})_\#\beta$. The transformed supports satisfy the cone condition, their densities are continuous on their supports and bounded above and below, and the dual divergence has continuous uniformly elliptic Hessian. Marginal quadratic terms cancel from the bias, and the entropy of each transformed marginal decreases by $\frac12\log\det A$.

\paragraph{An upper bound.}
The following bound holds for arbitrary bilinear costs under second-moment assumptions, including when $A$ is singular. Bounds with the same leading coefficient $d/2$ are known for smooth costs and compactly supported marginals with bounded densities~\cite[Proposition~3.4]{carlier2023entropicrate}, and for quadratic costs with arbitrary marginals having finite $(2+\delta)$ moments, $\delta>0$~\cite[Corollary~3.14]{eckstein2024quantization}. Appendix~\ref{app:entropic-value-bounds} gives a proof for general costs that satisfy a \emph{cross-difference criterion}, the proof is elementary and based on conditional Gaussian resampling. The proposition stated here is a specialization of that result to bilinear costs $c_A$  with explicit constants and only second moments. 

\begin{prop}[Entropic value upper bound]\label{prop:entropic-value-bounds}
Let $\alpha,\beta\in\mathcal P_2(\RR^d)$, and set
\[
m_\alpha\coloneqq\int x\,d\alpha(x),\qquad
s_\alpha^2\coloneqq\int\|x-m_\alpha\|^2\,d\alpha(x),
\]
with analogous notation for $\beta$. For every $A\in\RR^{d\times d}$ and $\varepsilon>0$,
\begin{equation}\label{eq:explicit-bilinear-upper}
0\le b_\varepsilon(A)\le\frac{d\varepsilon}{2}
\left[1+\log\left(1+\frac{2\|A\|_{op}s_\alpha s_\beta}{d\varepsilon}\right)\right].
\end{equation}
In particular, for every $B<\infty$ there is $C_B<\infty$ such that, for $0<\varepsilon\le1$,
\begin{equation}\label{eq:global-entropic-bias}
b_\varepsilon(A)\le\frac d2\varepsilon\log(1/\varepsilon)+C_B\varepsilon
\qquad(\|A\|_F\le B).
\end{equation}
\end{prop}

\subsection{Scale normalization and the limit problem}\label{subsec:well-posed-regularized-objective}
Positive rescaling preserves unregularized optimal plans, so cost recovery must fix the scale before taking the vanishing-entropy limit.

\paragraph{Escape along an admissible ray.}
Suppose $\hat\pi$ is optimal for $c_{\hat A}$ with $\hat A\succ0$. Positive scaling preserves optimal plans and rescales the entropic parameter, so
\[
\Ll_\varepsilon(t\hat A)=t\Ll_{\varepsilon/t}(\hat A)=-t b_{\varepsilon/t}(\hat A).
\]
By the pointwise formula~\eqref{eq:pointwise-pal} at $\hat A$, for fixed $\varepsilon>0$ and $t\to\infty$,
\[
\Ll_\varepsilon(t\hat A)=-\frac{d\varepsilon}{2}\log(t/\varepsilon)+O(\varepsilon)\longrightarrow-\infty.
\]
Thus the entropic inverse loss is unbounded below without control of scale, even when the admissible direction is identifiable.

\begin{assumption}\label{ass:trace-ground-truth}
The observed coupling $\hat\pi$ is optimal for $c_{\hat A}$ with $\hat A\in\mathbb S_{++}^d$. We normalize $\hat A$ so that $\operatorname{tr}(\hat A)=1$.
\end{assumption}

We consider optimization over the compact convex constraint set
\begin{equation}\label{eq:trace-class}
\Kk\coloneqq\{A\in\mathbb S_+^d:\operatorname{tr}(A)=1\}.
\end{equation}
 Define
\begin{equation*}\label{eq:J_eps}
\Jj_\varepsilon(A)\coloneqq\bar\Ll_\varepsilon(A)+\iota_{\Kk}(A),
\qquad A_\varepsilon\in\underset{A\in\Kk}{\argmin}\,\Ll_\varepsilon(A),
\tag{$\Jj_\varepsilon$}
\end{equation*}
where $\iota_{\Kk}$ is zero on $\Kk$ and $+\infty$ elsewhere. 
Throughout,  all matrix norms and inner products are Frobenius. In particular, $\|A\|_F\le1$ on $\Kk$.

\begin{prop}[Curvature and existence of the full-data minimizer]\label{prop:uniqueness-aeps}
Assume~\ref{ass:domain},~\ref{ass:densities}, and~\ref{ass:trace-ground-truth}. For every $\varepsilon>0$, $\bar\Ll_\varepsilon$ is $C^2$ and locally strongly convex on $\mathbb S_d$. It has a unique minimizer $A_\varepsilon$ on $\Kk$. If $A_\varepsilon\succ0$, its constrained first-order condition is
\[
\dotp{\operatorname{Sym}\bigl(\Sigma_{\pi_\varepsilon(A_\varepsilon)}-\Sigma_{\hat\pi}\bigr)}{H} = 0
\]
for every symmetric matrix $H$ with $\tr(H)=0$.
\end{prop}
\begin{rem}[Marginal covariances]\label{rem:marginal-covariances}
The entropic gap is locally strongly convex with respect to $A$ provided that the covariance matrices for $\alpha,\beta$ are nondegenerate. 
For a probability measure $\mu$ with finite second moment, set
\[
m_\mu\coloneqq\int z\,d\mu(z),\qquad
\Gamma_\mu\coloneqq\int(z-m_\mu)(z-m_\mu)^\top\,d\mu(z).
\]
The domain and density assumptions give $\Gamma_\alpha,\Gamma_\beta\succ0$, since neither marginal is supported on an affine hyperplane.
\end{rem}
\begin{proof}[Proof of Proposition~\ref{prop:uniqueness-aeps}]
The Hessian of the entropic loss is the squared distance of a cost perturbation from additive functions~\cite{andradesparsistency}:
\[
D^2\bar\Ll_\varepsilon(A)[H,H]
=\frac1\varepsilon\inf_{f\in L^2(\alpha),\,g\in L^2(\beta)}
\int(x^\top Hy-f(x)-g(y))^2\,d\pi_\varepsilon(A).
\]
On bounded supports, the density of $\pi_\varepsilon(A)$ relative to $\alpha\otimes\beta$ has a positive lower bound $b_{A,\varepsilon}$, locally uniform in $A$. Under the product measure, projection onto additive functions leaves $(x-m_\alpha)^\top H(y-m_\beta)$. Thus
\[
D^2\bar\Ll_\varepsilon(A)[H,H]
\ge\frac{b_{A,\varepsilon}}\varepsilon
\lambda_{\min}(\Gamma_\alpha)\lambda_{\min}(\Gamma_\beta)\|H\|_F^2.
\]
This proves local strong convexity and hence strict convexity. Continuity and compactness of $\Kk$ give existence, and strict convexity gives uniqueness. 

The final statement follows since at a positive-definite minimizer, the feasible infinitesimal directions are precisely the trace-free symmetric matrices. 
\end{proof}

\paragraph{The limit problem.}
The pointwise formula~\eqref{eq:pointwise-pal} identifies $-\frac12\log\det A$ as the coefficient of $\varepsilon$ in $\bar\Ll_\varepsilon(A)-\Ll_0(A)$. The resulting limit problem as $\varepsilon\to0$ is
\begin{equation*}\label{eq:limit-problem}
\min_{\substack{A\in\mathbb S_+^d,\ \operatorname{tr}(A)=1}}
-\log\det A\quad\text{subject to}\quad\Ll_0(A)=0,
\tag{$\Jj_0$}
\end{equation*}
where $-\log\det A=+\infty$ for singular $A$. 

\begin{prop}[Limit of the trace-normalized minimizers]\label{prop:limit-aeps}
Under the assumptions of Proposition~\ref{prop:uniqueness-aeps} with $k\ge1$, problem~\eqref{eq:limit-problem} has a unique solution $A_0\succ0$. As $\varepsilon\to0$,
\[
A_\varepsilon\to A_0,\qquad
\frac{\Jj_\varepsilon(A_\varepsilon)}\varepsilon\to -\tfrac12\log\det A_0.
\]
In particular, $A_\varepsilon\succ0$ for all sufficiently small $\varepsilon$. Under~\eqref{eq:hessian-span-symmetric}, the normalized admissible set is the singleton $\{A_0\}$ and $A_0=\hat A$.
\end{prop}

\begin{proof}
We write $F(A)\coloneqq-\tfrac12\log\det A$ on the positive-definite cone, extended by  $+\infty$ at singular matrices.

The feasible set $\Kk\cap\{\Ll_0=0\}$ is compact and convex, since $\Ll_0$ is continuous, convex, and nonnegative. It contains the positive-definite matrix $\hat A$.  On the trace-one set, $F$ is bounded below and tends to $+\infty$ at the singular boundary. Lower semicontinuity therefore gives a positive-definite minimizer $A_0$, and strict convexity of $-\log\det$ gives uniqueness. This is also the minimizer of~\eqref{eq:limit-problem}.

Choose $r>0$ so that the closed ball $\overline B_r(A_0)$ consists of positive-definite matrices, and put
\[
\Cc\coloneqq\overline B_r(A_0)\cap\{A\in\mathbb S_d:\operatorname{tr}A=1\}\subset\Kk.
\]

We prove Gamma-convergence on $\Cc$ using only the pointwise formula~\eqref{eq:pointwise-pal}. For a positive-definite matrix $B$, let $T_B$ be its optimal map and $\phi_B$ be the associated   convex dual potential with $\nabla\phi_B=BT_B$. Denote 
\[
v_B(y)\coloneqq\phi_B^*(By),\qquad
D_B(z,y)\coloneqq v_B(y)-v_B(z)-\langle\nabla v_B(z),y-z\rangle,
\]
where $\phi_B^*$ is the convex conjugate of $\phi_B$.
Since $\nabla v_B(T_B(x)) = Bx$,
the OT dual slack is
$\phi_B(x)+v_B(y)-x^\top By=D_B(T_B(x),y)$. The change of variables $(x,y)\mapsto(T_B(x),y)$ preserves relative entropy with respect to the product marginals and gives
\begin{equation}\label{eq:fixed-target-bias}
b_\varepsilon(B)=\inf_{\gamma\in\Pi(\beta,\beta)}
\left\{\int D_B(z,y)\,d\gamma(z,y)+\varepsilon\KL(\gamma\mid\beta\otimes\beta)\right\}.
\end{equation}
Propositions~\ref{prop:regularity_OT} and~\ref{prop:regularity_potential}, applied also to the reverse problem, give positive-definite Hessians $D^2v_B$ and their continuous dependence on $B$ in the uniform norm. Here $k\ge1$ ensures $C^2$ continuity.

Let $A_\varepsilon'\to A\in\Cc$. For each fixed $s>1$, this continuity and the positive lower bound on $D^2v_A$ imply, for all sufficiently small $\varepsilon$,
\[
D_{A_\varepsilon'}\le sD_A,\qquad
b_\varepsilon(A_\varepsilon')\le s\,b_{\varepsilon/s}(A).
\]
Applying~\eqref{eq:pointwise-pal} only at the fixed matrix $A$ gives
\[
\limsup_{\varepsilon\to0}\left[
\frac{b_\varepsilon(A_\varepsilon')}{\varepsilon}-\frac d2\log(1/\varepsilon)
\right]\le g(A)+\frac d2\log s.
\]
Letting $s\downarrow1$ and using $\bar\Ll_\varepsilon-\Ll_0=-b_\varepsilon-a_\varepsilon$, we obtain
\[
\liminf_{\varepsilon\to0}
\frac{\bar\Ll_\varepsilon(A_\varepsilon')-\Ll_0(A_\varepsilon')}{\varepsilon}
\ge F(A).
\]
If $\Ll_0(A)=0$, nonnegativity of $\Ll_0$ gives the required Gamma-liminf inequality. If $\Ll_0(A)>0$, continuity gives $\Ll_0(A_\varepsilon')/\varepsilon\to+\infty$, and the same lower bound implies divergence of the full objective. At each point $A$ of the zero set, the constant sequence is a recovery sequence by~\eqref{eq:pointwise-pal}. Thus $\bar\Ll_\varepsilon/\varepsilon$ Gamma-converges on $\Cc$ to $F+\iota_{\{\Ll_0=0\}}$.

Every minimizer of the restricted problem therefore converges to its unique limiting minimizer $A_0$, and the minimum values converge to $F(A_0)$. For small $\varepsilon$, these minimizers lie in the relative interior of $\Cc$ in the trace-one hyperplane. Their gradients vanish on trace-free symmetric directions. By  Proposition~\ref{prop:uniqueness-aeps}, they coincide with $A_\epsilon$, the minimizer over  $\Kk$.  Finally, under~\eqref{eq:hessian-span-symmetric}, Theorem~\ref{thm:uniqueness} gives $\Kk\cap\{\Ll_0=0\}=\{\hat A\}$, as claimed.
\end{proof}

\subsection{Bias control under full data}\label{subsec:bias-full-data}
Fixing the trace converts curvature transverse to the admissible ray into quadratic growth in the full matrix error. This lets us estimate recovery directly from the forward value bias.

\begin{thm}[Deterministic bias]\label{thm:bias}
Assume~\ref{ass:domain} and~\ref{ass:densities} with $k=1$, Assumption~\ref{ass:trace-ground-truth}, and the spanning condition~\eqref{eq:hessian-span-symmetric} holds at $\hat A$. Then, the limit solution is $A_0=\hat A$. For all sufficiently small $\varepsilon>0$, the minimizer $A_\varepsilon$ on $\Kk$ is positive definite and
\[
\|A_\varepsilon-A_0\|_F\le C\sqrt\varepsilon,
\]
\end{thm}

The proof is essentially a direct consequence of the following lemma (on the curvature of $\Ll_0$) and the entropic bias bounds from Proposition \ref{prop:entropic-value-bounds}.
\begin{lem}[Quadratic growth on the trace-one slice]\label{lem:trace-growth}
Under the assumptions of Theorem~\ref{thm:bias}, there is $C<\infty$ such that every $A\in\Kk$ satisfies
\[
\|A-A_0\|_F^2\le C\Ll_0(A).
\]
\end{lem}
\begin{proof}
If $d=1$, then $\Kk=\{1\}$ and the claim is immediate; assume henceforth $d\ge2$.
Write $A=tA_0+E$ with $E\perp A_0$. Since $A\succeq0$, $\operatorname{tr}A=1$, and $A_0\succ0$, the coefficient $t$ is positive. The trace constraint gives $t-1=-\operatorname{tr}E$, and hence,
\[
\|A-A_0\|_F\le(1+\sqrt d\,\|A_0\|_F)\|E\|_F.
\]
Theorem~\ref{thm:local_curvature_L0}, or Corollary~\ref{cor:stabilty_S0}, bounds $\|E\|_F^2$ by $C\Ll_0(A)$ in a neighborhood of $A_0$. The local quadratic estimate from Theorem~\ref{thm:local_curvature_L0} makes $A_0$ an isolated zero of $\Ll_0$ on $\Kk$. By convexity, any second zero on $\Kk$ would give a segment of zeros approaching $A_0$, contradicting this estimate. Thus $A_0$ is the unique zero on $\Kk$. Continuity and compactness then give a strictly positive lower bound on $\Ll_0$ outside a smaller neighborhood of $A_0$, while $\|A-A_0\|_F$ is bounded. Increasing $C$ proves the claimed inequality globally. 
\end{proof}

\begin{proof}[Proof of Theorem \ref{thm:bias}]
By Lemma~\ref{lem:trace-growth}, it is sufficient to show that $\Ll_0(A_\varepsilon)=\Oo(\varepsilon)$. From Proposition \ref{prop:limit-aeps}, $A_0=\hat A$, we have $\Ll_0(A_0)=0$.
By minimality of $\Ll_\varepsilon$ at $A_\varepsilon$ and $\Ll_\varepsilon=\Ll_0-b_\varepsilon$,
\begin{equation}\label{eq:population-bias-comparison}
\Ll_0(A_\varepsilon)
\le b_\varepsilon(A_\varepsilon)-b_\varepsilon(A_0).
\end{equation}
On $\Kk$, $\|A_\varepsilon\|_F\le1$, so~\eqref{eq:global-entropic-bias} bounds $b_\varepsilon(A_\varepsilon)$ by $\frac d2\varepsilon\log(1/\varepsilon)+C\varepsilon$. For the second bias term, the pointwise formula~\eqref{eq:pointwise-pal} at the fixed matrix $A_0$ gives
\[
b_\varepsilon(A_0)=\frac d2\varepsilon\log(1/\varepsilon)+O(\varepsilon).
\]
Here the $O(\varepsilon)$ constant is independent of $\varepsilon$, since $A_0$ is fixed. The logarithmic terms cancel, yielding $\Ll_0(A_\varepsilon)\le C\varepsilon$ for all sufficiently small $\varepsilon$. This proves the claimed rate. In particular, $A_\varepsilon\to A_0\succ0$, so $A_\varepsilon$ is positive definite for all sufficiently small $\varepsilon$.
\end{proof}

\subsection{Finite-sample recovery}\label{subsec:variance_control}
We now combine the deterministic entropic bias with fluctuations of the empirical loss to obtain recovery guarantees from finitely many observed pairs.
Let $(x_i,y_i)_{i=1}^n$ be i.i.d.\ samples from $\hat\pi$, and set
\[
\hat\pi_n\coloneqq\frac1n\sum_{i=1}^n\delta_{(x_i,y_i)},\qquad
\hat\alpha_n\coloneqq\frac1n\sum_{i=1}^n\delta_{x_i},\qquad
\hat\beta_n\coloneqq\frac1n\sum_{i=1}^n\delta_{y_i}.
\]
The empirical entropic gap is
\begin{equation*}\label{eq:eps_loss_n}
\Ll_{n,\varepsilon}(A)\coloneqq\langle c_A,\hat\pi_n\rangle
-\inf_{\pi\in\Pi(\hat\alpha_n,\hat\beta_n)}
\{\langle c_A,\pi\rangle+\varepsilon\KL(\pi\mid\hat\alpha_n\otimes\hat\beta_n)\}.
\tag{$\Ll_{n,\varepsilon}$}
\end{equation*}
Using the same deterministic shift $a_\varepsilon$ as for the population loss, define
\begin{equation*}\label{eq:j-eps-n}
\Jj_{n,\varepsilon}(A)\coloneqq\Ll_{n,\varepsilon}(A)-a_\varepsilon+\iota_{\Kk}(A),
\qquad A_\varepsilon^n\in\underset{A\in\Kk}{\argmin}\,\Ll_{n,\varepsilon}(A).
\tag{$\Jj_{n,\varepsilon}$}
\end{equation*}
As before,
the shift is merely for theoretical convenience.

The following theorem combines the bias and sampling bounds to choose the entropic regularization term \(\varepsilon_n\) and obtain finite-sample recovery rates.
For $0<\delta<1$, write
\[
q_d\coloneqq\left\lceil\frac{5d}{4}\right\rceil+3,
\qquad
\Lambda_{d,\delta}\coloneqq
\bigl(1+\sqrt{\log(5/\delta)}\bigr)^{\lceil5d/2\rceil+6}
\bigl(1+\log(5/\delta)\bigr).
\]
\begin{thm}[Finite-sample recovery]\label{thm:statistical}
Under the assumptions of Theorem~\ref{thm:bias}, for $n\ge d+1$ and $\varepsilon>0$ the empirical minimizer $A_\varepsilon^n$ on $\Kk$ exists and is unique almost surely. Put
\[
t_{n,\varepsilon,\delta}\coloneqq\frac{K_{1,\sigma,d}\Lambda_{d,\delta}}
{\varepsilon^{q_d}\sqrt n},
\]
where $\sigma^2$ is a fixed sub-Gaussian parameter for the bounded marginals. There are constants $\varepsilon_0,\eta,C>0$, depending only on the model, such that if
\[
0<\varepsilon\le\varepsilon_0,\qquad
\varepsilon\log(e/\varepsilon)+t_{n,\varepsilon,\delta}\le\eta,
\]
then, with probability at least $1-\delta$, $A_\varepsilon^n\succ0$ and
\begin{equation}\label{eq:statistical-trace-bound}
\begin{aligned}
\|A_\varepsilon^n-A_0\|_F^2&\le C(\varepsilon+t_{n,\varepsilon,\delta}).
\end{aligned}
\end{equation}
In particular, for fixed confidence $0<\delta<1$ and any fixed $C_*>0$, choosing
\[
\varepsilon_n\coloneqq C_*\left(\frac{\Lambda_{d,\delta}}{\sqrt n}\right)^{1/(q_d+1)}
\]
gives, for all sufficiently large $n$, with probability at least $1-\delta$,
\[
\|A_{\varepsilon_n}^n-A_0\|_F
\lesssim\left(\frac{\Lambda_{d,\delta}}{\sqrt n}\right)^{1/(2(q_d+1))} = O\bigl(n^{-1/(4\lceil5d/4\rceil+16)}\bigr).
\]
\end{thm}
\begin{proof}[Proof of Theorem~\ref{thm:statistical}]
Existence follows from continuity and compactness of $\Kk$. For uniqueness, the Hessian kernel of the discrete entropic gap consists of additive cost perturbations $u_i+v_j$~\cite{andradesparsistency}. The first $d+1$ points of each marginal sample are affinely independent almost surely, since the marginal densities are absolutely continuous.  Thus, for $n\ge d+1$,
\[
\Span\{x_i-x_1:2\le i\le n\}
=\Span\{y_j-y_1:2\le j\le n\}=\mathbb R^d.
\]
If $x_i^\top H y_j=u_i+v_j$, mixed differences give
$(x_i-x_1)^\top H(y_j-y_1)=0$, which forces $H=0$. Hence the empirical loss is strictly convex in $A$, and its minimizer on $\Kk$ is unique.

By Theorem~\ref{thm:uniform-probability-ot-cost} and $\|A\|_F\le1$ on $\Kk$, with probability at least $1-\delta$,
\[
\sup_{A\in\Kk}|\Ll_{n,\varepsilon}(A)-\Ll_\varepsilon(A)|
\le t_{n,\varepsilon,\delta}.
\]
Work on this event and abbreviate the right-hand side by $t$. Minimality gives
\[
\Ll_\varepsilon(A_\varepsilon^n)
\le\Ll_{n,\varepsilon}(A_\varepsilon^n)+t
\le\Ll_{n,\varepsilon}(A_0)+t
\le\Ll_\varepsilon(A_0)+2t.
\]
Since $\Ll_0(A_0)=0$, this is exactly
\begin{equation}\label{eq:sample-bias-comparison}
\Ll_0(A_\varepsilon^n)
\le b_\varepsilon(A_\varepsilon^n)-b_\varepsilon(A_0)+2t.
\end{equation}

By~\eqref{eq:global-entropic-bias}, uniformly for $A\in\Kk$,
\[
b_\varepsilon(A)\le\frac d2\varepsilon\log(1/\varepsilon)+C\varepsilon.
\]
The pointwise formula~\eqref{eq:pointwise-pal} at $A_0$ gives the matching logarithmic term for $b_\varepsilon(A_0)$. Substituting into~\eqref{eq:sample-bias-comparison} yields
\[
\Ll_0(A_\varepsilon^n)\le C\varepsilon+2t.
\]
Lemma~\ref{lem:trace-growth} gives~\eqref{eq:statistical-trace-bound}. For $\eta$ sufficiently small, the theorem's condition makes this error smaller than $\lambda_{\min}(A_0)$, so $A_\varepsilon^n\succ0$. Only the uniform upper bound, the pointwise expansion at $A_0$, and the empirical concentration estimate have been used.

Finally, for the stated schedule,
\[
\frac{t_{n,\varepsilon_n,\delta}}{\varepsilon_n}
=\frac{K_{1,\sigma,d}}{C_*^{q_d+1}}.
\]
Thus $t_{n,\varepsilon_n,\delta}=O(\varepsilon_n)$, and the localization condition holds for all sufficiently large $n$. Substitution proves the rates.
\end{proof}

\section{Degenerate settings of iOT}

We study settings in which a single optimal transport plan fails to identify the cost up to scale. In the discrete setting, the admissible costs form a polyhedral cone, which can have dimension greater than one when there are too few samples relative to the matrix dimension (Proposition~\ref{prop:local-invariance} and Remark~\ref{rem:finite_dim}). Moreover, a unique optimal plan gives a full-dimensional plateau in the loss $\Ll_0$ (Proposition~\ref{prop:plateau}), so for $d\ge2$ the cost is not identifiable up to scale from such discrete samples for any sample size, although the plateau narrows empirically as $n$ grows (Example~\ref{exmp:perturbed gaussian}, Figure~\ref{fig:perturb gaussian}). Uniqueness of the plan matters: Example~\ref{exmp:discrete-identifiable-nonunique} shows a non-unique discrete optimal plan that identifies the cost up to positive scaling.

For $d\ge2$, non-identifiability beyond scale can also occur in the continuous setting. For elliptical distributions with a common base, including Gaussians, the admissible costs form a cone of dimension $d(d+1)/2$ (Proposition~\ref{prop:nonidentifiability-elliptical}), even with full knowledge of the transport plan; for $\alpha=\beta=\mathcal N(0,I_2)$ the sampled loss is identically zero (Example~\ref{exmp:perturbed gaussian}). This degeneracy is unstable under perturbations (Proposition~\ref{prop:generic_identifiability}, illustrated in Example~\ref{exmp:perturbed gaussian}), and radial distributions with different bases identify the cost (Proposition~\ref{prop:radial-densities}, Example~\ref{exmp: annulus}).

\subsection{The discrete setting} \label{sec:discrete}

Let $\alpha=\sum_{i=1}^p a_i\delta_{x_i}$ and $\beta=\sum_{j=1}^q b_j\delta_{y_j}$, and let
\[
\mathcal U(a,b) \coloneqq \{P\in\RR_+^{p\times q}: P\mathbf 1=a,\ P^\top\mathbf 1=b\}
\]
be the transportation polytope. For a bilinear cost $c_A(x,y)=-x^\top A y$, define the cost matrix
\[
C(A)_{ij} \coloneqq -x_i^\top A y_j,
\qquad
\langle C(A),P\rangle \coloneqq \sum_{i=1}^p\sum_{j=1}^q C(A)_{ij}P_{ij}.
\]
In this case,~\eqref{eq:OT} can be written as  the linear program
\begin{equation}\label{eq:primal-ot}
\min_{P\in\mathcal U(a,b)}\ \langle C(A),P\rangle.
\end{equation}

\paragraph{Dual formulation and certificates.}
The dual of~\eqref{eq:primal-ot} reads
\begin{equation}\label{eq:dual-ot}
\max_{f\in\RR^p,\ g\in\RR^q}\ a^\top f + b^\top g
\qquad\text{s.t.}\qquad
f_i + g_j \le C(A)_{ij}\quad \forall i,j.
\end{equation}
A plan $\hat P\in\mathcal U(a,b)$ is optimal for cost $C(A)$ if and only if there exist dual potentials $(f,g)$ such that
\begin{equation}\label{eq:kkt-ot}
\begin{cases}
f_i + g_j \le C(A)_{ij} & \forall i,j,\\
f_i + g_j = C(A)_{ij} & \forall (i,j)\in\supp(\hat P),
\end{cases}
\end{equation}
where $\supp(\hat P)\coloneqq \{(i,j):\hat P_{ij}>0\}$.


Given an optimal coupling $\hat P \in \RR^{p\times q}$,
the set of admissible costs for which $\hat P$ is optimal is
\begin{equation}\label{eq:S0-def}
\mathcal \Ss_0(\hat P)
\coloneqq
\{A\in\RR^{d\times d}:\ \hat P\in\mathop{\argmin}_{P\in\mathcal U(a,b)}\langle C(A),P\rangle\}.
\end{equation}
We characterize the non-identifiability below,  giving a direct specialization of the inverse-LP fact that the set of objectives making a fixed primal solution optimal is the (polyhedral) normal cone of the feasible polytope at that solution; intersecting that cone with the bilinear parameterization $A \mapsto C(A)$ yields non-identifiability whenever the restriction map $A\mapsto \bigl(C(A)_{ij}\bigr)_{(i,j)\in\Supp(\hat P)}$ has a kernel direction transverse to $A$, under the strict off-support complementarity hypothesis below.

\begin{prop}[Local invariance of optimality under support-orthogonal perturbations]\label{prop:local-invariance}
Let $a\in\RR_+^p$ and $b\in\RR_+^q$ be probability vectors, and let $\hat P\in\mathcal U(a,b)$.  
Assume there exist $A\in\RR^{d\times d}$ and dual potentials $(f,g)\in\RR^p\times\RR^q$ such that
\begin{equation}\label{eq:kkt-hatP}
\begin{cases}
f_i+g_j < C(A)_{ij} & \forall (i,j) \not\in\Supp(\hat P)\\
f_i+g_j = C(A)_{ij} & \forall (i,j)\in\Supp(\hat P),
\end{cases}
\end{equation}
where $C(A)_{ij}=-x_i^\top A y_j$. In particular, we assume that \emph{strict complementarity holds off the support}.

Let $M_{ij}\coloneqq x_i y_j^\top$. If
\begin{equation}\label{eq:finite_dim_cond}
\mathrm{dim} \Span\enscond{M_{i,j}}{(i,j)\in\Supp(\hat P)} \le d^2-2,
\end{equation}
then $\mathrm{dim}(\Ss_0(\hat P))\geq 2$.

\end{prop}

\begin{proof}
Let $W$ be the orthogonal complement of $\Span\{M_{ij}:(i,j)\in\supp(\hat P)\}$. By~\eqref{eq:finite_dim_cond}, $\dim W\ge2$. For every $H\in W$, the costs $C(A+H)$ and $C(A)$ agree on the support. The finitely many strict off-support inequalities persist for all sufficiently small $\norm H$. Thus the same potentials $(f,g)$ remain dual-feasible and satisfy complementary slackness with $\hat P$. Consequently, $\Ss_0(\hat P)$ contains a neighborhood of $A$ in $A+W$, proving the dimension claim.
\end{proof}

\begin{rem}[Non-identifiability.]\label{rem:finite_dim}
    In general, a basic feasible plan has
$|\supp(P)|\le p+q-1$.
 Indeed, a nondegenerate basic feasible plan corresponds to a spanning tree in the bipartite graph on $p+q$ nodes which has exactly $p+q-1$ edges/support. Thus, when $d^2>p+q$, condition~\eqref{eq:finite_dim_cond} necessarily holds for a basic feasible plan; if, in addition, the strict off-support certificate of Proposition~\ref{prop:local-invariance} exists, that proposition rules out identifiability up to scale.
\end{rem}

\begin{rem}[Sensitivity to the active set.]
 Uniqueness up to scaling is governed by the full normal-cone constraints, including any tight off-support inequalities. A perturbation of the empirical plan can change its support and hence the face of the normal cone. This observation motivates the numerical sensitivity study below, but we do not assert a general discontinuity theorem or unavoidable argmin jumps.
\end{rem}

Proposition~\ref{prop:local-invariance} gives a sufficient condition for non-identifiability in terms of the support of the observed plan. A stronger obstruction holds whenever that plan is the unique optimum: every sufficiently close cost keeps it optimal, so the zero set of the loss contains an open neighborhood of the true cost, regardless of the sample size. This obstruction is specific to unregularized optimal transport, since it uses that the minimum of a linear program is attained at a vertex. Moreover, uniqueness of $\hat P$ is essential here: Example~\ref{exmp:discrete-identifiable-nonunique} below shows that a nonunique discrete optimal plan can identify the cost up to positive scaling.

\begin{prop}[Plateau around the true cost]\label{prop:plateau}
Let $\alpha=\sum_{i=1}^p a_i\delta_{x_i}$ and $\beta=\sum_{j=1}^q b_j\delta_{y_j}$, let $\hat A\in\RR^{d\times d}$, and assume that $\hat P$ is the \emph{unique} optimal plan of~\eqref{eq:primal-ot} for $A=\hat A$. Then there exists $r>0$ such that, for every $A$ with $\norm{A-\hat A}_F<r$, $\hat P$ is the unique optimal plan for $C(A)$. In particular
\[
\Ll_0(A)=0\quad\text{for all }\norm{A-\hat A}_F<r,
\]
so $\Ss_0(\hat P)$ contains an open neighborhood of $\hat A$ in $\RR^{d\times d}$. 
\end{prop}

\begin{proof}
The set $\mathcal U(a,b)$ is a polytope, so it has finitely many vertices. Since $\hat P$ is the unique minimizer of a linear function over it, $\hat P$ is a vertex. Let $P_1,\dots,P_m$ be the other vertices. Uniqueness at $\hat A$ gives, for every $j$,
\[
\Delta_j\coloneqq \langle C(\hat A),P_j-\hat P\rangle=\langle \hat A,K(\hat P)-K(P_j)\rangle_F>0 \qwhereq K(P)\coloneqq \sum_{i,j}P_{ij}x_iy_j^\top.
\]
For any $A$, by linearity and the Cauchy--Schwarz inequality,
\[
\langle C(A),P_j-\hat P\rangle=\Delta_j+\langle A-\hat A,\,K(\hat P)-K(P_j)\rangle_F\ \ge\ \Delta_j-\norm{A-\hat A}_F\,\norm{K(\hat P)-K(P_j)}_F .
\]
Put
\[
r\coloneqq \min_{j:\,K(P_j)\ne K(\hat P)}\ \frac{\Delta_j}{\norm{K(\hat P)-K(P_j)}_F}>0
\]
(with $r=+\infty$ if the set of indices is empty). If $\norm{A-\hat A}_F<r$, then $\langle C(A),P_j-\hat P\rangle>0$ for every $j$, that is, $\hat P$ has strictly smaller cost than all other vertices. Any $P\in\mathcal U(a,b)$ is a convex combination of the vertices, so $\langle C(A),P\rangle>\langle C(A),\hat P\rangle$ unless $P=\hat P$. Hence $\hat P$ is the unique optimal plan and $\Ll_0(A)=0$.
\end{proof}

\begin{rem}[Sampled pairs from an injective gradient map]\label{rem:plateau-sampling}
Let $X_1,\dots,X_n$ be distinct, $\Psi$ convex and differentiable at each $X_i$ with $\nabla\Psi$ injective on $\{X_i\}$, and $Y_i=\nabla\Psi(X_i)$, $a=b=\tfrac1n\mathbf 1$ and $\hat P=\tfrac1n\mathrm{Id}$. By the Birkhoff--von Neumann theorem the vertices of $\mathcal U(a,b)$ are the permutation matrices divided by $n$. For a permutation $\sigma\ne\mathrm{id}$, the Fenchel--Young inequality $X_i^\top Y_{\sigma(i)}\le\Psi(X_i)+\Psi^*(Y_{\sigma(i)})$ is strict whenever $\sigma(i)\neq i$, since $\nabla\Psi$ is injective. Summing over $i$ and using $\sum_i\Psi^*(Y_{\sigma(i)})=\sum_i\Psi^*(Y_i)$ together with equality in Fenchel--Young for $\sigma=\mathrm{id}$ gives
$\sum_iX_i^\top Y_{\sigma(i)}<\sum_iX_i^\top Y_i$. 
The identity is therefore the unique optimal assignment for the true cost, and Proposition~\ref{prop:plateau} applies.

\end{rem}

The population identifiability of Theorem~\ref{thm:uniqueness} is therefore compatible with a sampled loss that is flat on a polyhedral neighborhood of the true cost for every $n$. This is what we observe in practice in Figure~\ref{fig:perturb gaussian}: for each $n$ the sampled loss vanishes on an interval containing $u=\tfrac12$ (approximately symmetric here since the model is invariant under swapping the two coordinates, but not exactly symmetric for a given sample), and for $\varepsilon=0.01$ this plateau covers almost the whole segment. In our experiments the plateau narrows as $n$ grows, but Proposition~\ref{prop:plateau} gives no rate: the radius $r$ depends on the gaps $\Delta_j$ and on the sample, and we do not claim convergence of the zero sets as a theorem.

\begin{exmp}[A nonunique plan can identify the cost]\label{exmp:discrete-identifiable-nonunique}
In $\RR^2$, take source atoms $0,-e_1,-e_2,e_1+e_2$ and target atoms $0,e_1,e_2$. Define the observed plan by sending mass $1/9$ along each of the nine pairs
\[
\begin{array}{rcl}
0&\longrightarrow&0,\ e_1,\ e_2,\\
-e_1&\longrightarrow&0,\ e_2,\\
-e_2&\longrightarrow&0,\ e_1,\\
e_1+e_2&\longrightarrow&e_1,\ e_2.
\end{array}
\]
Let $\alpha,\beta$ be its marginals. This plan is optimal for $\hat A=I_2$: each source sends mass only to targets maximizing $x^\top y$, so no coupling can have a smaller cost. For each panel of Figure~\ref{fig:discrete-identifiability-cycles}, let $P_t$ be the plan obtained by adding $t$ on its solid arrows and subtracting $t$ on its dashed arrows. For every $|t|<1/9$, this plan has the same marginals as $\hat P$ and satisfies
\[
\langle C(I_2),\hat P-P_t\rangle=0.
\]
Thus the forward optimal plan is not unique.

\begin{center}
\includegraphics[width=\linewidth]{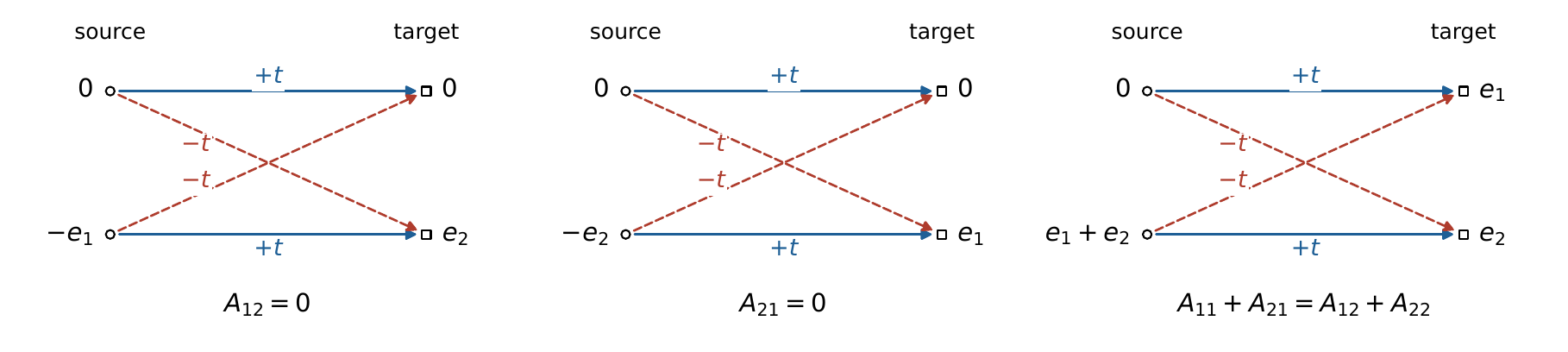}
\captionsetup{hypcap=false}
\captionof{figure}{Each panel shows a feasible mass exchange within the support of the observed plan. Add $t$ on the solid arrows and subtract $t$ on the dashed arrows; both signs of $t$ are feasible for $|t|<1/9$. Optimality therefore forces each exchange to have zero cost change, giving the equality below its panel. Shared pairs in different panels represent the same coupling entry.}\label{fig:discrete-identifiability-cycles}
\end{center}

Now suppose $A\in\Ss_0(\hat P)$. Optimality of $\hat P$ gives $\langle C(A),P_t-\hat P\rangle\ge0$ for every $|t|<1/9$.
For the first panel, the computation is
\begin{align*}
0\leq \langle C(A),P_t-\hat P\rangle
&=t\bigl(c_A(0,0)+c_A(-e_1,e_2)
-c_A(0,e_2)-c_A(-e_1,0)\bigr)=t A_{12},
\end{align*}
for all $\abs{t}\leq 1/9$ and 
so, $A_{12}=0$. The second panel similarly gives $A_{21}=0$, and the third gives $A_{11}+A_{21}=A_{12}+A_{22}$. Hence $A=\lambda I_2$. To fix the sign, move a small amount of mass from the pairs $(0,e_1)$ and $(-e_1,0)$ to $(0,0)$ and $(-e_1,e_1)$. This feasible perturbation changes the cost by $t\lambda$, with $t>0$, so optimality requires $\lambda\ge0$. Conversely, every nonnegative multiple of $I_2$ preserves optimality. Hence
\[
\Ss_0(\hat P)=\{\lambda I_2:\lambda\ge0\},
\]
even among nonsymmetric competing matrices. The trace-one compatible set is therefore the singleton $\{I_2/2\}$, despite nonuniqueness of the forward optimal plan.
\end{exmp}
This example contrasts with Proposition~\ref{prop:plateau}: an observed plan that splits mass can impose equality constraints on the cost, whereas a unique optimal plan remains optimal under small cost perturbations. It does not arise from the injective-map sampling setting of Remark~\ref{rem:plateau-sampling}, where the observed assignment is unique.

\subsection{The continuous setting}\label{sec:Gaussian}
\subsubsection{Elliptical distributions}
In this section, we consider the setting where $\al$ and $\beta$ are elliptical distributions, where the iOT problem can be studied explicitly. 

Let $g:\RR^d\to[0,\infty)$ be a centered spherically symmetric probability density with covariance $I_d$, and let $\Omega$ be invertible. We write $\Ee(\mu,\Omega,g)$ for the law of
\[
X\coloneqq\mu+\Omega Z,\qquad Z\sim g.
\]
Its covariance is $\Sigma\coloneqq\Omega\Omega^\top$ and its Lebesgue density is
\[
f(x)\coloneqq|\det\Omega|^{-1}g\bigl(\Omega^{-1}(x-\mu)\bigr).
\]
Equivalently, if $g(z)=\widetilde g(\|z\|^2)$ for a scalar radial generator $\widetilde g$, then
$f(x)=|\Sigma|^{-1/2}\widetilde g((x-\mu)^\top\Sigma^{-1}(x-\mu))$.
A typical base density is $g_{\rm gauss}(z)\coloneqq(2\pi)^{-d/2}e^{-\|z\|^2/2}$, for which $\Ee(\mu,\Omega,g_{\rm gauss})=\mathcal N(\mu,\Sigma)$.
In this subsection we assume that $\alpha$ and $\beta$ are elliptical with the same base distribution $g$ i.e.
$$
\alpha = \Ee(\mu_1,\Omega_1,g) \qandq  \beta = \Ee(\mu_2,\Omega_2,g)
$$
This specifies only the marginals: one may write
\[
X=\mu_1+\Omega_1Z_1,\qquad
Y=\mu_2+\Omega_2Z_2,
\]
where $Z_1,Z_2$ each have density $g$; no coupling, and in particular no independence or equality between $Z_1$ and $Z_2$, is imposed at this stage.
\begin{prop}\label{prop:nonidentifiability-elliptical}
    Assume that $\alpha$ and $\beta$ have the same base $g$. Given $\hat A$ an invertible matrix, the OT map $T_{\hat A}$ for the cost $c_{\hat A}$ is affine with
    $$
    T_{\hat A}(x) = \mu_2 + \Omega_2 \hat V\hat U^\top\Omega_1^{-1}(x-\mu_1) \qwhereq \Omega_1^\top\hat A\Omega_2 \overset{SVD}{=}\hat U\hat S\hat V^\top
    $$
    Moreover
    $$
    S_0(\hat\pi) = \left\{A: \Omega_1^\top A\Omega_2 =\hat U H\hat V^\top, \; H\in\mathbb S_+^d\right\}.
    $$
\end{prop}
\begin{proof}
    For $\hat A=\Id$, set $\Sigma_i\coloneqq\Omega_i\Omega_i^\top$ and consider the symmetric positive-definite covariance map \cite{Gelbrich1990OnAFellipticalwasserstein}
    $$
    L\coloneqq\Sigma_1^{-1/2}(\Sigma_1^{1/2}\Sigma_2\Sigma_1^{1/2})^{1/2}\Sigma_1^{-1/2},
    \qquad T_{\Id}(x)=\mu_2+L(x-\mu_1).
    $$
    Since $L\Sigma_1L=\Sigma_2$ and $g$ is spherically symmetric, this map pushes $\alpha$ to $\beta$. Since $L\succ0$, it is the gradient of a convex quadratic and hence the Brenier map.
    Now, for a given invertible $\hat A$, after performing the linear change of variable $\hat A^\top\#\alpha$ one recovers the Euclidean transport where the Brenier map is affine. Since the change of variable is linear, $T_{\hat A}$ is also affine with $T_{\hat A}(x)=b+Mx$. The map $T_{\hat A}$ solves the following optimization problem
    $$
    \max_{T\#\alpha=\beta}\EE[X^\top \hat A T(X)]
    $$
    Matching first moments gives $T(x) = \mu_2 + M(x-\mu_1)$. Matching second moments gives $M\Sigma_1 M^\top=\Sigma_2$.
    After canceling the constant terms, the optimization problem becomes,
    $$
    \max_{M\Sigma_1 M^\top=\Sigma_2}\Tr(\hat AM\Sigma_1)
    $$
    Write $M=\Omega_2 Q\Omega^{-1}_1$ for some matrix $Q$. Then the constraint becomes 
    $$
    M\Sigma_1 M^\top=\Sigma_2\iff Q\in \textbf{O}_d
    $$
    Rewriting the optimization in term of $Q\in\textbf{O}_d$ gives
    $$
    \max_{Q\in\textbf{O}_d} \Tr(\Omega_1^\top \hat A\Omega_2 Q)
    $$
    which is solved by 
    $$
    Q^*\coloneqq\hat V\hat U^\top \qwhereq \Omega_1^\top\hat A\Omega_2 \overset{SVD}{=}\hat U\hat S\hat V^\top
    $$
    Therefore 
    $$
    T_{\hat A}(x) = \mu_2 + \Omega_2 \hat V\hat U^\top\Omega_1^{-1}(x-\mu_1) \qwhereq \Omega_1^\top\hat A\Omega_2 \overset{SVD}{=}\hat U\hat S\hat V^\top
    $$
    Put $\hat Q\coloneqq\hat V\hat U^\top$. For any competing $A$, including singular matrices, Cauchy--Schwarz in the singular-vector coordinates of $M\coloneqq\Omega_1^\top A\Omega_2$ bounds the centered score of every coupling by the sum of the singular values of $M$, attained by an orthogonal coupling of the spherical base. Thus $A$ is compatible with the observed plan precisely when $\hat Q$ maximizes
    $Q\mapsto\operatorname{Tr}(\Omega_1^\top A\Omega_2Q)$ over $\mathbf O_d$.
    Writing $M=\Omega_1^\top A\Omega_2$ and $Q=\hat Q R$, this is equivalent to the identity maximizing
    $R\mapsto\operatorname{Tr}(M\hat Q R)$. The equality case in the orthogonal Procrustes problem says exactly that $M\hat Q$ is symmetric positive semidefinite. Equivalently,
    \[
    M=\hat U H\hat V^\top\quad\text{for some }H\in\mathbb S_+^d.
    \]
    This includes singular compatible costs and the zero cost. If uniqueness of the competing forward map is required, then $H\succ0$.
 \end{proof}

Proposition~\ref{prop:nonidentifiability-elliptical} demonstrates that for $d\ge2$ the elliptical setting is non-identifiable beyond scale when the bases are the same: the admissible cone is linearly isomorphic to $\mathbb S_+^d$ and therefore has dimension $d(d+1)/2$. Nonetheless, Proposition~\ref{prop:generic_identifiability} shows that this non-uniqueness is unstable under compactly supported smooth perturbations. We present in Figure \ref{fig:perturb gaussian} numerical illustrations of this effect with simple, fully supported perturbations for computational simplicity.

\begin{exmp}[Numerical illustration of the identifiability instability and plateau]\label{exmp:perturbed gaussian}
In this simulation we have considered the following two-dimensional perturbed model : 
\begin{align*}
    \alpha&=\Nn(0,I_2),  \quad \Psi_\varepsilon(x)\coloneqq \frac{1}{2}|x|^2 + \varepsilon \left(\cos x_1 + \cos x_2 + \cos(x_1+x_2)\right), \\
    T_\varepsilon(x) \coloneqq \nabla \Psi_\varepsilon(x) &= x-\varepsilon \left(\begin{array}{c}
         \sin x_1 +\sin (x_1+x_2) \\
         \sin x_2+ \sin (x_1+x_2)
    \end{array}\right), \quad \beta_\varepsilon \coloneqq (T_\varepsilon)_\#\alpha \\
    \pi_\varepsilon &\coloneqq (I_d,T_\varepsilon)\#\alpha, \quad A(u) \coloneqq \left(\begin{array}{cc}
        u & 0 \\
        0 & 1-u
    \end{array}\right), \quad 0\le u\le 1.
\end{align*}

\noindent The true cost is $\hat A=I_2$, or equivalently its trace-one representative $A(\frac{1}{2})=\frac{1}{2}I_2$. For $0\le\varepsilon<1/3$, the map $T_\varepsilon$ is optimal for this cost between $\alpha$ and $\beta_\varepsilon$. For $0<\varepsilon<1/3$, it also satisfies the identifiability condition of Theorem~\ref{thm:uniqueness}.

Indeed, the Hessian of $\Psi_\varepsilon$ is 
$$
\nabla^2\Psi_\varepsilon(x) = I_2-\varepsilon \begin{pmatrix}
    \cos x_1 + \cos (x_1+x_2) & \cos(x_1+x_2) \\
    \cos(x_1+x_2) & \cos x_2 +\cos(x_1+x_2)
\end{pmatrix}.
$$
Writing as $P(x)$ the matrix with cosines terms, we get for all unit vector $v\in\RR^2$ and $x\in\RR^2$
$$
v^\top P(x) v = v_1^2\cos x_1 + v_2^2\cos x_2 + \cos(x_1+x_2)(v_1+v_2)^2 \le 1+2=3
$$
Then, uniformly on $\RR^2$
$$
\nabla^2\Psi_\varepsilon(x)\succeq (1-3\varepsilon)I_2, \quad \forall x\in\RR^2.
$$
Therefore, for $0\le\varepsilon<1/3$, $\Psi_\varepsilon$ is uniformly convex, which ensures that $T_\varepsilon$ is an optimal Brenier map for the cost $A(\frac{1}{2})$ between $\alpha$ and $\beta_\varepsilon$. Moreover, the Hessians $\left\{\nabla^2\Psi_\varepsilon(0,0),\nabla^2\Psi_\varepsilon (\pi,0),\nabla^2\Psi_\varepsilon(0,\pi)\right\}$ span the symmetric matrices for $\varepsilon\ne0,2$.
Indeed, computing the determinant of the 3 hessians in the canonical basis $\{E_{11},E_{22},E_{12}\}$ of $\mathbb S^2$ gives
$$
\det\left\{\begin{pmatrix}
    1-2\varepsilon & -\varepsilon \\
    -\varepsilon & 1-2\varepsilon
\end{pmatrix}, \begin{pmatrix}
    1+2\varepsilon & \varepsilon \\
    \varepsilon & 1
\end{pmatrix}, \begin{pmatrix}
    1 & \varepsilon \\
    \varepsilon & 1+2\varepsilon
\end{pmatrix}\right\} = 4\varepsilon^2(\varepsilon-2)
$$

Therefore, for $0<\varepsilon<1/3$, Theorem~\ref{thm:uniqueness} gives $\Ss_0(\pi_\varepsilon) = \{\lambda I_2: \;\lambda\ge0\}$. The population gap has a unique zero on the trace-one slice at $u=1/2$. These marginals have unbounded support, so the bounded-domain hypotheses of Theorem~\ref{thm:curvature} do not directly apply to this example. Figure~\ref{fig:perturb gaussian} illustrates the dependence on sample size and perturbation strength through the empirical gap defined below.

$$
    {\rm Given } \quad (X_i,T_\varepsilon(X_i)) \overset{i.i.d.}{\sim} \pi_\varepsilon \quad {\rm write } \quad
    \Ll_{n,0}^\varepsilon(A) \coloneqq\max_{\sigma}\frac{1}{n}\sum_{i=1}^nX_i^\top A T_\varepsilon(X_{\sigma(i)}) -\frac{1}{n}\sum_{i=1}^nX_i^\top AT_\varepsilon(X_i) 
$$

\noindent Notice that for $\varepsilon=0$ we have $T_\varepsilon = \Id$ and for any positive semi-definite matrix $A$ and permutation $\sigma$ 
$$
\sum _i X_i^\top AX_i - \sum_i X_i^\top A X_{\sigma(i)} = \frac{1}{2}\sum_i (X_i-X_{\sigma(i)})^\top A(X_i-X_{\sigma(i)})\ge0
$$
The identity pairing $\sigma(i)=i$ is therefore optimal for every such $A$. Then for $\varepsilon=0$ the loss is identically equal to zero on the axis $u\mapsto A(u)$, ie
$$
\Ll_{n,0}^{\varepsilon=0}(A(u)) = 0 \quad \forall n \ge 1, 0\le u\le 1
$$
In Figure~\ref{fig:perturb gaussian} we plot $u\mapsto \Ll_{n,0}^\varepsilon(A(u))$ for $\varepsilon \in \{0.01,0.1,0.3\}$ and $n\in\{50,100,1000,2000\}$. The baseline $\varepsilon=0$ is zero up to floating-point error and is omitted from the figure. For $n=50$ and $n=100$, curves show means over five independent draws with bands of one empirical standard deviation; for $n=1000$ and $n=2000$, each curve is a single realization. The small loss scale at $\varepsilon=0.01$ reflects a weak perturbation of the degenerate model together with finite-sample plateaus; it is not by itself evidence of a floating-point precision limit. Larger perturbations produce more visible gaps away from $u=1/2$, while Proposition~\ref{prop:plateau} still applies by Remark~\ref{rem:plateau-sampling}. The displayed realizations illustrate these effects without asserting a rate or monotonicity in $n$. Exact assignments are evaluated on the specified grid and joined by straight segments; no entropic regularization or curve smoothing is used.

\begin{figure}[h]
    \centering
    \includegraphics[width=\linewidth]{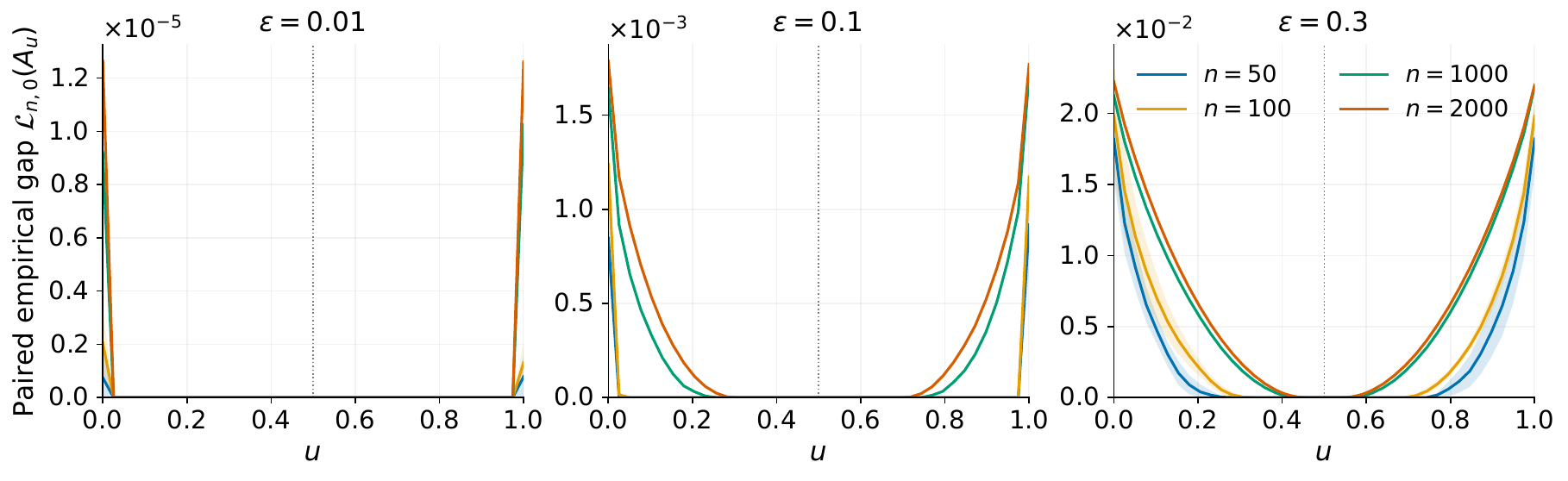}
    \caption{Graphs of $u\mapsto\Ll_{n,0}^{\varepsilon}(A(u))$ for $A(u)=\diag(u,1-u)$, $\varepsilon\in\{0.01,0.1,0.3\}$, and $n\in\{50,100,1000,2000\}$. Here $\varepsilon$ is the cosine perturbation parameter and the loss is unregularized. The first two sample sizes show the mean over five independent draws with bands of one empirical standard deviation; the two large sample sizes use one realization each. The vertical scale differs across panels.}
    \label{fig:perturb gaussian}
\end{figure}

\end{exmp}
\subsubsection{Radial densities}
Radial distributions can identify the cost when their optimal transport is not a dilation. 
Here, we consider $\alpha,\beta$ to have the same means but different bases (no condition on their variance). Proposition~\ref{prop:radial-densities} gives an almost-everywhere condition on the radial transport profile that ensures identifiability; in dimensions $d\ge2$, it yields the essential spanning condition~\eqref{eq:hessian-span-symmetric}.

\begin{prop}[Identifiability for nonlinear radial transport]\label{prop:radial-densities}
Let $\alpha,\beta\in\mathcal P_2(\RR^d)$ have radial Lebesgue densities, suppose the density of $\alpha$ is positive almost everywhere on $\RR^d$, and let $\hat A=\Id$. The optimal coupling $\hat\pi=(\Id,T)_\#\alpha$ has a representative
\[
T(x)=b(\|x\|)\frac{x}{\|x\|},\qquad x\ne0,
\]
with $b:(0,\infty)\to(0,\infty)$ nondecreasing. If $b(r) \neq r b'(r)$ on a set of positive Lebesgue measure, then
\[
\Ss_0(\hat\pi)=\{\lambda\Id:\lambda\ge0\}.
\]
\end{prop}
\begin{rem}
If $b$ is locally absolutely continuous on $(0,\infty)$, then the condition that $b(r)\neq rb'(r)$ on a set of positive Lebesgue measure is equivalent to $b$ not being of the form $b(r)=cr$. Indeed, $w(r)\coloneqq b(r)/r$ is locally absolutely continuous and $w'(r)=(rb'(r)-b(r))/r^2$ almost everywhere. Thus $b(r)=rb'(r)$ almost everywhere if and only if $w$ is constant, which is exactly the dilation case.
\end{rem}

\begin{proof}
    Since $\alpha,\beta$ have radial densities, $f(Rx)=f(x)$ and $g(Rx)=g(x)$ for all $R\in\textbf{O}_d$. Given a map $T$ that pushes forward $\alpha$ to $\beta$, for all $R\in\textbf{O}_d$ the map $\Tilde{T}_R(x)\coloneqq R^{-1}T(Rx)$ pushes $\alpha$ to $\beta$ as well. 
    Moreover, $\tilde T_R$ doesn't perturb the cost, i.e.,
    $$
    -\int x^\top T(x) d\alpha(x) = -\int x^\top \tilde T_R(x) d\alpha(x)
    $$
    
    Since the OT map is unique, 
    $$
    RT(x)=T(Rx)
    $$
    It follows that $T(x)=b(\norm{x})x/\norm{x}$. Note also that by Brenier's theorem $T = \nabla \phi$ for some convex function $\phi$, and so, $b(t) = \dotp{T(tu)}{u} = \frac{d}{dt} \phi(tu)$ is the derivative of a convex function, so  $b$ is non-decreasing. In particular, $b$ is differentiable almost everywhere. Moreover $b>0$: if $b(r_0)=0$ for some $r_0>0$, monotonicity gives $T=0$ on $B(0,r_0)$, so $\beta$ would have an atom at $0$, contradicting that $\beta$ has a density. Let $N$ be the Lebesgue-null set on which $b$ is not differentiable.

    Given $x\in\RR^d\setminus\{0\}$, let $r\coloneqq\norm{x}$ and $u\coloneqq x/r$. Then, for all $x \not\in \enscond{x}{\norm{x}\in N}$,
    $$
    \nabla T(x)=b'(r)uu^\top + \frac{b(r)}{r}(\Id-uu^\top) = \left(b'(r)-\frac{b(r)}{r}\right)uu^\top+\frac{b(r)}{r}\Id.
    $$
    Moreover, $\nabla T$ is defined $\alpha$-almost everywhere since $\enscond{x}{\norm{x}\in N}$ is Lebesgue null and $\alpha$ has a density.

For $d\ge2$, let $E$ be any Lebesgue-null set containing the exceptional points of the Hessian formula.
Since $a(r)\coloneqq b'(r)-b(r)/r\neq0$ on a set of positive measure, we can choose $r$ such that $a(r)\neq0$ and
    $\nabla T(ru)=a(r)uu^\top+b(r)\Id/r$.
 By polar Fubini and continuity in $u$, we may choose this radius so that every matrix $a(r)uu^\top+b(r)\Id/r$ lies in $\Span\{\nabla T(x):x\notin E\}$.
    Taking differences at that radius shows that the matrices $uu^\top-vv^\top$ belong to
    $\Span\{\nabla T(x):x\in\mathbb R^d\setminus E\}$ for every $u,v\in\mathbb S^{d-1}$. These differences span the codimension-one subspace of trace-zero symmetric matrices. Moreover, $\nabla T(ru)$ has positive trace $b'(r)+(d-1)b(r)/r>0$, since $b'\ge0$ and $b(r)>0$. Its span therefore contains that codimension-one subspace and a matrix outside it, so
    \[
    \Span\{\nabla T(x):x\in\mathbb R^d\setminus E\}=\mathbb S_d.
    \]
    Theorem~\ref{thm:uniqueness} now shows that the full admissible set is $\{\lambda\Id:\lambda\geq0\}$, and every nonzero admissible cost is a positive multiple of $\Id$. The case $d=1$ follows from Remark~\ref{rem:oned-identifiability}.
\end{proof}

\begin{exmp}[Numerical illustration of identifiability for radial distribution with different bases] 
\label{exmp: annulus}

For $\alpha=\mathcal N(0,I_2)$ and $\beta$ uniform on the annulus $\{a\le\|y\|\le c\}$ ($a=0.8$, $c=0.85$, $\Delta\coloneqq c^2-a^2$), radial quantile matching gives for $\hat A=I_2$
\[
b(r)=\sqrt{a^2+\Delta(1-e^{-r^2/2})},\qquad b'(r)=\frac{\Delta r e^{-r^2/2}}{2b(r)},\qquad
\frac{rb'(r)}{b(r)}\le\frac{\Delta}{e\,a^2}\approx0.047<1,
\]
so $b(r)\ne rb'(r)$ for every $r>0$ and Proposition~\ref{prop:radial-densities} gives $\Ss_0(\hat\pi)=\{\lambda I_2:\lambda\ge0\}$.
Note that $b(0^+)=a>0$ (so $T$ is not continuous at the origin) and that $\mathrm{Cov}(\beta)=0.340625\,I_2\neq I_2$.
We sample $X_i\sim\alpha$, set $Y_i\coloneqq T(X_i)$ with $T(x)=b(\|x\|)x/\|x\|$, and compute $\Ll_{n,0}$; the curve is the mean over 5 independent draws (shaded: $\pm$ one standard deviation).
The population cost is identifiable with infinite sample, whereas for each $n$ the sampled loss has a plateau containing $u=\tfrac12$ (Proposition~\ref{prop:plateau}, Remark~\ref{rem:plateau-sampling}); in our experiments it narrows as $n$ grows.

\begin{figure}[h]
    \centering
    \includegraphics[width=\linewidth]{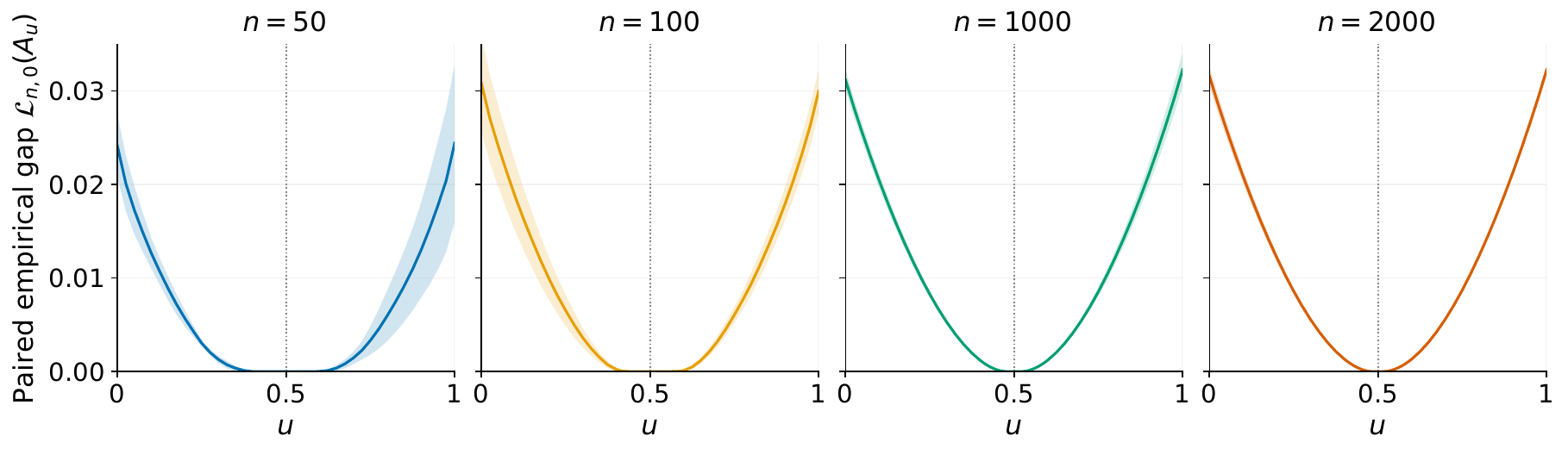}
    \caption{Graphs of $u\mapsto\Ll_{n,0}(A(u))$ for $A(u)=\diag(u,1-u)$, with paired observations from the standard Gaussian source and its optimal map to the uniform annulus with radii $0.8$ and $0.85$. Each panel shows the mean over five independent draws and a band of one empirical standard deviation, for $n\in\{50,100,1000,2000\}$. All panels use the same vertical scale.}
    \label{fig:annulus gap loss}
\end{figure}
\begin{figure}[ht]
    \centering
    \includegraphics[width=\linewidth]{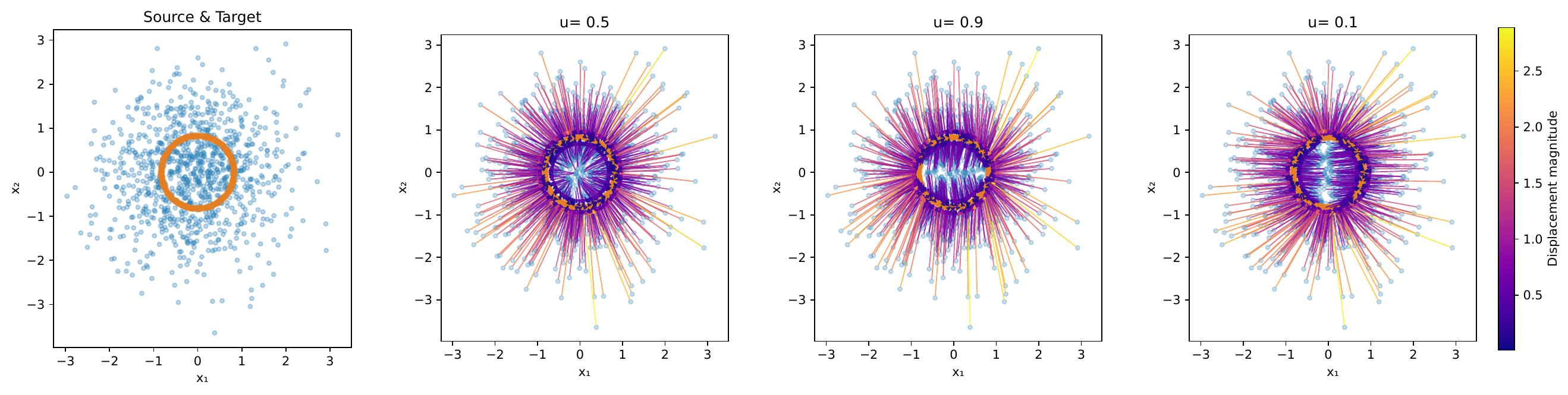}
    \caption{Transport maps $T_A$ for different matrices $A=\diag(u,1-u)$, illustrated by empirical OT assignments $X_i\mapsto Y_{\sigma_A(i)}$ between independently sampled marginals. The source is a standard Gaussian and the target is uniform on the annulus with radii $0.8$ and $0.85$. The left panel shows the two clouds of $n=1000$ points; the other panels show all assignment arrows for $u=0.5,0.9,0.1$. The clouds are fixed across panels, and colours indicate Euclidean displacement magnitude on a shared scale.}
    \label{fig:ot-map-annulus}
\end{figure}
Figure~\ref{fig:ot-map-annulus} is a separate forward illustration with independently sampled marginals, using seed $42$. Figure~\ref{fig:annulus gap loss} uses paired observations from the population plan to evaluate the inverse gap loss. The arrows in Figure~\ref{fig:ot-map-annulus} show how the optimal assignment on fixed empirical marginals changes with the cost; they are not claimed to be exact population maps.
\end{exmp}

The accompanying code is self-contained. Its entry point, \nolinkurl{figures/reproduce.py}, generates the four vector PDF figures used here from the shared map and assignment implementation and saves all sampled arrays and loss values. For each sample size $n$ in the gap-loss experiments, a local NumPy \texttt{default\_rng(n)} generates successive independent arrays of shape $(n,2)$ for the repetitions. In the cosine experiment the same source array is reused across the three perturbations within a repetition. The cosine grid is the union of 40 equally spaced points on $[0,1]$ and 11 on $[0.4,0.6]$; the annulus grid uses 41 and 21 points, respectively. The stored grids retain the original notebook's floating-point convention. All forward objectives use equal weights $1/n$ and unregularized maximum-score assignments; the rank-one endpoint objectives can equivalently be obtained by sorting. The loss normalization, exact seeds, grids, repetition counts, software versions, and raw arrays are recorded with the outputs. Standard deviations use the number of repetitions as denominator and are not confidence intervals. The repository also supplies exploratory Jupyter notebooks and a browser experiment comparing these true-plan observations with a different protocol in which independent marginal samples are matched by empirical OT at the true cost.

\section*{Conclusion}

This work establishes a sharp contrast between inverse OT regimes: for $d\ge2$, uniquely optimal discrete plans and same-base elliptical models fail to identify the cost beyond scale, while smooth continuous marginals can restore local well-posedness through transport curvature. Theorem~\ref{thm:curvature} characterizes the second-order geometry of OT with respect to the cost function. As a result, Theorems~\ref{thm:uniqueness},~\ref{thm:local_curvature_L0},~\ref{thm:bias}, and~\ref{thm:statistical} provide a complete chain from identifiability to quantitative recovery.

Beyond theory, these findings are directly relevant for practical iOT deployment. In realistic pipelines, one works with finite samples and entropic surrogates; our analysis clarifies how $\epsilon$ controls a bias--conditioning tradeoff and how this interacts with sampling noise. This gives principled guidance for choosing regularization and for interpreting recovered costs.

A natural next direction is to extend the identifiability and statistical theory from bilinear costs to general smooth parametrizations $c_\theta$. At a compatible base where the gap is twice differentiable, the parameter Hessian is the pullback of the cost-space Hessian along $\theta\mapsto c_\theta$. Injectivity of the tangent map modulo the transport-invariant kernel characterized in Theorem~\ref{thm:curvature} gives a sufficient condition for quadratic local stability. When this condition fails, qualitative identifiability can still hold and may require higher-order analysis. This perspective could yield general rank conditions and recovery rates for nonlinear cost families beyond the bilinear setting studied here.

Applications such as single-cell genomics, where OT-type couplings are used to relate cell populations across time or conditions \cite{schiebinger2019optimal,bunne2022proximal}, provide one motivation for such extensions. In these settings, inverse formulations are often used as exploratory tools for learning costs between latent states or features \cite{samaran2024scconfluence}. The present results provide a starting point towards clarifying the mathematical obstructions that any such interpretation must confront, including symmetry, discretization, and regularization-induced degeneracy.


\section*{Acknowledgments}

The work of G. Peyr\'e was supported by the French government under the management of Agence Nationale de la Recherche as part of the ``Investissements d’avenir'' program, reference ANR-19-P3IA-0001 (PRAIRIE 3IA Institute) and by the European Research Council (ERC project WOLF).

\bibliographystyle{abbrv}
\bibliography{references}

\appendix

\section{Sample complexity of entropic OT}\label{app:sample-complexity-entropic-ot}

We extend the expectation estimates of Mena and Niles-Weed~\cite{mena2019statistical} to a high-probability bound, uniform over bounded cost matrices. The linear part of the entropic gap is controlled separately by Bernstein's inequality. Throughout this section, $(X_i,Y_i)_{i=1}^n$ are i.i.d.\ with law $\hat\pi$ and marginals $\alpha,\beta$; $X_i$ and $Y_i$ need not be independent. We abbreviate the empirical marginals by $\alpha_n,\beta_n$.

\begin{thm}[Sample complexity of the entropic gap loss]\label{thm:uniform-probability-ot-cost}
Assume that $\alpha$ and $\beta$ are $\sigma^2$-sub-Gaussian. For $B>0$, $0<\varepsilon\le1$, $0<\delta<1$, and $n\ge1$, with probability at least $1-\delta$,
\[
\sup_{\|A\|\le B}|\Ll_\varepsilon(A)-\Ll_{n,\varepsilon}(A)|
\le\frac{K_{B,\sigma,d}\Lambda_{d,\delta}}{\varepsilon^{q_d}\sqrt n}.
\]
Here $q_d$ and $\Lambda_{d,\delta}$ are defined in Section~\ref{subsec:variance_control}, and $K_{B,\sigma,d}$ is independent of $n,\varepsilon,\delta$.
\end{thm}

We prove uniform bounds for the two terms
\begin{equation}\label{eq:linear-term}
\sup_{\|A\|\le B}\left|\int x^\top Ay\,d\hat\pi
-\frac1n\sum_{i=1}^nX_i^\top AY_i\right|
\end{equation}
and
\begin{equation}\label{eq:ot-term}
\sup_{\|A\|\le B}
|\OT^\varepsilon_{\alpha,\beta}(c_A)-\OT^\varepsilon_{\alpha_n,\beta_n}(c_A)|.
\end{equation}
All constants denoted by $C_d$ may change from line to line and depend only on $d$.

\subsection{Sub-Gaussian variables and the linear term}\label{subsec:subgaussian-tools}

\begin{defn}[Sub-Gaussian scale]
A random vector $X\in\mathbb R^d$, or its law $P$, is $\sigma^2$-sub-Gaussian if
\[
\EE\exp\!\left(\frac{\|X\|^2}{2d\sigma^2}\right)\le2.
\]
The scale $\sigma>0$ is an upper bound, not necessarily the smallest admissible value. No centering is required; in particular, the definition also applies to empirical measures.
\end{defn}

\begin{defn}[Sub-exponential Orlicz norm]
For a scalar random variable $Z$, set
\[
\|Z\|_{\psi_1}\coloneqq\inf\{c>0:\EE\exp(|Z|/(2c))\le2\}.
\]
We call $Z$ sub-exponential if this norm is finite.
\end{defn}

We use Bernstein's inequality in its full form: for i.i.d.\ variables $Z_i$ with $\|Z_i-\EE Z_i\|_{\psi_1}\le K$, with probability at least $1-\delta$,
\[
\left|\frac1n\sum_{i=1}^nZ_i-\EE Z_1\right|
\lesssim K\left(\sqrt{\frac{\log(2/\delta)}n}+\frac{\log(2/\delta)}n\right).
\]
See Sections~2.7--2.8 of~\cite{Vershynin_2018}. Retaining the second term avoids a restriction on $n$ relative to the confidence level.

\begin{proof}[Bound for~\eqref{eq:linear-term}]
For $Z=X^\top AY$, Young's and Cauchy--Schwarz inequalities give
\[
\EE\exp\!\left(\frac{|Z|}{4d\sigma^2\|A\|_{op}}\right)
\le\left(\EE e^{\|X\|^2/(4d\sigma^2)}\right)^{1/2}
\left(\EE e^{\|Y\|^2/(4d\sigma^2)}\right)^{1/2}\le2
\]
when $A\ne0$; the case $A=0$ is immediate. Thus
$\|Z-\EE Z\|_{\psi_1}\lesssim d\sigma^2\|A\|_{op}$, without independence between $X$ and $Y$.
To obtain a simultaneous bound in $A$, apply Bernstein's inequality to the $d^2$ entries of
\[
D_n\coloneqq\frac1n\sum_{i=1}^nX_iY_i^\top-\EE[XY^\top].
\]
A union bound and $|\langle A,D_n\rangle|\le B\|D_n\|_F$ yield
\[
\sup_{\|A\|\le B}|\langle A,D_n\rangle|
\le C_d B\sigma^2\left(\sqrt{\frac{\log(2d^2/\delta)}n}
+\frac{\log(2d^2/\delta)}n\right).
\]
Applying the same argument to $XX^\top$ and $YY^\top$ controls the empirical second-moment matrices of both marginals. A further finite union bound controls all three matrices on one event.
\end{proof}

\subsection{High-probability concentration for entropic OT}\label{subsec:sample-complexity-entropic-ot}

Write $S_\varepsilon(P,Q)$ for entropic OT with cost $\frac12\|x-y\|^2$, and put $s\coloneqq\lceil d/2\rceil+1$. Let $\Ff_s$ be the class of $C^s$ functions satisfying
\[
|f(x)|\le C_{s,d}(1+\|x\|^2),\qquad
|D^a f(x)|\le C_{s,d}(1+\|x\|^s)\quad(1\le|a|\le s).
\]
For a measure $P$, write $\|P-P_n\|_{\Ff_s}\coloneqq\sup_{f\in\Ff_s}|(P-P_n)f|$.
Corollary~2 and the derivative bounds of~\cite{mena2019statistical} imply the pathwise estimate
\[
|S_1(P,Q)-S_1(P_n,Q_n)|
\lesssim_d(1+\tilde\sigma^{3s})
\bigl(\|P-P_n\|_{\Ff_s}+\|Q-Q_n\|_{\Ff_s}\bigr),
\]
whenever $P,Q,P_n,Q_n$ share sub-Gaussian scale $\tilde\sigma$. Their entropy estimate also gives
\[
\EE\|P-P_n\|_{\Ff_s}\le\frac{C_d(1+\sigma^{d+2})}{\sqrt n}
\]
for a $\sigma^2$-sub-Gaussian law $P$. We combine this expectation bound with the following upper-tail inequality.

\begin{prop}[Adamczak's inequality, sub-exponential case]\label{prop:adamczak-tail-bound}
Let $X_i$ be independent and $\mathcal F$ a countable class with $\EE f(X_i)=0$. Set
\[
Z\coloneqq\sup_{f\in\mathcal F}\left|\sum_i f(X_i)\right|,\quad
v^2\coloneqq\sup_{f\in\mathcal F}\sum_i\EE f(X_i)^2,\quad
U\coloneqq\left\|\max_i\sup_{f\in\mathcal F}|f(X_i)|\right\|_{\psi_1}.
\]
If $U<\infty$, then for a universal $C$ and all $t\ge0$,
\[
\PP(Z\ge2\EE Z+t)\le e^{-t^2/(4v^2)}+3e^{-t/(CU)}.
\]
This is Theorem~4 of~\cite{adamczak2008tailinequalitysupremaunbounded} with $a=\tau=1$ and $\eta=1/2$, enlarging the expectation term to $2\EE Z$. The normalization of $\psi_1$ changes only $C$.
\end{prop}

\begin{prop}[Concentration at unit regularization]\label{prop:intermediate-probability-ot-cost}
Let $P,Q$ be $\sigma^2$-sub-Gaussian, and let $P_n,Q_n$ be their empirical marginals from i.i.d.\ pairs. Put $p_d\coloneqq\lceil5d/2\rceil+6$ and $L_\delta\coloneqq1+\log(1/\delta)$. With probability at least $1-4\delta$,
\begin{align*}
|S_1(P,Q)-S_1(P_n,Q_n)|
&\le C_d\bigl(1+[\sigma(1+\sqrt{L_\delta})]^{p_d}\bigr)\\
&\quad\times\left[\frac1{\sqrt n}
+(1+\sigma^2)\left(\sqrt{\frac{L_\delta}n}
+\frac{\log(n+1)L_\delta}n\right)\right].
\end{align*}
\end{prop}
\begin{proof}
For the centered class $f-Pf$, the envelope satisfies
\[
\sup_{f\in\Ff_s}|f(X)-Pf|
\le C_d(1+\|X\|^2+\EE\|X\|^2).
\]
Its $\psi_1$ norm is at most $C_d(1+\sigma^2)$. Also,
\[
\sup_{f\in\Ff_s}\operatorname{Var}(f(X))
\le\sup_{f\in\Ff_s}\EE f(X)^2\le C_d(1+\sigma^2)^2.
\]
Here we used $\EE\|X\|^{2k}\le2(2d\sigma^2)^k k!$. Lemma~\ref{lem:max-exponential} therefore bounds the maximal envelope norm by $C_d(1+\sigma^2)\log(n+1)$. Proposition~\ref{prop:countable-supremum} justifies the countable-class hypothesis. Applying Proposition~\ref{prop:adamczak-tail-bound} gives, with probability at least $1-\delta$,
\[
\|P-P_n\|_{\Ff_s}
\le\frac{C_d(1+\sigma^{d+2})}{\sqrt n}
+C_d(1+\sigma^2)\left(\sqrt{\frac{L_\delta}n}
+\frac{\log(n+1)L_\delta}n\right).
\]
The same estimate holds for $Q$. Lemma~\ref{lem:tilde-sigma}, applied to each marginal, gives a common scale at most $C\sigma(1+\sqrt{L_\delta})$ outside two further events of probability $\delta$. The union bound thus costs $4\delta$, regardless of dependence between the two marginals. Combining the pathwise estimate with these four events proves the claim, since $3s+d+2\le p_d$.
\end{proof}

\begin{proof}[Proof of Theorem~\ref{thm:uniform-probability-ot-cost}]
Set $P_A\coloneqq(A^\top)_\#\alpha$ and $Q\coloneqq\beta$. The bilinear and quadratic costs satisfy
\[
\OT^\varepsilon_{\alpha,\beta}(c_A)
=S_\varepsilon(P_A,Q)-\frac12\int\|A^\top x\|^2\,d\alpha
-\frac12\int\|y\|^2\,d\beta.
\]
This identity also holds for singular $A$: pushforward decreases relative entropy, and lifting a coupling using the conditional law of $x$ given $A^\top x$ attains equality. It holds for the empirical measures as well. The second-moment bounds above control the two additional terms uniformly in $\|A\|\le B$.

To make the OT estimate simultaneous in $A$, put $b\coloneqq1\vee B$ and work with the fixed transformed data
\[
U_i\coloneqq\frac b{\sqrt\varepsilon}X_i,\qquad
V_i\coloneqq\frac1{\sqrt\varepsilon}Y_i,
\qquad \rho\coloneqq\frac{b\sigma}{\sqrt\varepsilon}.
\]
Both laws have sub-Gaussian scale $\rho$. For $T_A\coloneqq A^\top/b$, we have $\|T_A\|_{op}\le1$. The chain rule shows that $f\circ T_A\in C_d\Ff_s$ for every $f\in\Ff_s$, uniformly in $A$. Consequently,
\[
\sup_{\|A\|\le B}\|(T_A)_\#(P-P_n)\|_{\Ff_s}
\le C_d\|P-P_n\|_{\Ff_s},
\]
where $P$ is the law of $U_i$. Contractions also preserve an upper sub-Gaussian scale. Thus the four events in the preceding proof, applied to the fixed data $U_i,V_i$, control every $A$ at once.

Finally, $S_\varepsilon(P_A,Q)=\varepsilon S_1((\varepsilon^{-1/2}I)_\#P_A,(\varepsilon^{-1/2}I)_\#Q)$. Apply the preceding bound with scale $\rho$ and multiply by $\varepsilon$. Since $p_d/2\le q_d$, $0<\varepsilon\le1$, and $\log(n+1)\lesssim\sqrt n$,
\[
\sup_{\|A\|\le B}
|S_\varepsilon(P_A,Q)-S_\varepsilon(P_{A,n},Q_n)|
\le\frac{K_{B,\sigma,d}}{\varepsilon^{q_d}\sqrt n}
(1+\sqrt{L_\delta})^{p_d}(1+L_\delta).
\]
Indeed, the prefactor contributes at most $\varepsilon^{-q_d}$, and $\varepsilon(1+\rho^2)$ is bounded by $1+b^2\sigma^2$. Add the event controlling the linear and second-moment terms, at cost $\delta$, and replace $\delta$ by $\delta/5$. The resulting confidence factor is bounded by a dimension-dependent multiple of $\Lambda_{d,\delta}$, proving the theorem.
\end{proof}

\subsection{Technical lemmas}

\begin{prop}[Countable reduction]\label{prop:countable-supremum}
There is a countable subset $\tilde{\Gg}_s\subset\Ff_s$ such that for every sub-Gaussian law $P$ and every finite sample,
\[
\sup_{f\in\Ff_s}|(P_n-P)f|
=\sup_{g\in\tilde{\Gg}_s}|(P_n-P)g|.
\]
\end{prop}
\begin{proof}
The space $C(\mathbb R^d)$ with uniform convergence on compact sets is second countable, so its subspace $\Ff_s$ has a countable dense subset $\tilde{\Gg}_s$. For each $f\in\Ff_s$, choose $g_m\in\tilde{\Gg}_s$ converging locally uniformly to $f$. Then $P_ng_m\to P_nf$ for every finite sample, while $Pg_m\to Pf$ by dominated convergence and the common envelope $C_{s,d}(1+\|x\|^2)$.
\end{proof}

\begin{lem}\label{lem:max-exponential}
For nonnegative sub-exponential variables $Y_1,\dots,Y_n$, with $K\coloneqq\max_i\|Y_i\|_{\psi_1}$,
\[
\|\max_iY_i\|_{\psi_1}\lesssim K\log(n+1).
\]
\end{lem}
\begin{proof}
Markov's inequality and the union bound give
$\PP(\max_iY_i>t)\le\min\{1,2ne^{-t/(2K)}\}$.
Integrating this bound in the tail representation of the exponential moment yields the claim.
\end{proof}

\begin{lem}\label{lem:tilde-sigma}
Let $P$ be $\sigma^2$-sub-Gaussian and $P_n$ its empirical measure. Define
\[
\tilde\sigma\coloneqq\inf\{r\ge\sigma:P_n\text{ is }r^2\text{-sub-Gaussian}\}.
\]
With probability at least $1-\delta$,
\[
\sigma\le\tilde\sigma\le\sigma\left(1+\sqrt{\frac{\log(1/\delta)}{\log2}}\right).
\]
\end{lem}
\begin{proof}
For $t>0$, set $\lambda\coloneqq((\sigma+t)/\sigma)^2\ge1$. Jensen's and Markov's inequalities give
\begin{align*}
\PP(\tilde\sigma>\sigma+t)
&=\PP\left(\frac1n\sum_i e^{\|X_i\|^2/(2d(\sigma+t)^2)}>2\right)\\
&\le2^{-\lambda}\EE\left[\frac1n\sum_i
 e^{\lambda\|X_i\|^2/(2d(\sigma+t)^2)}\right]
\le2^{1-\lambda}\le e^{-t^2\log2/\sigma^2}.
\end{align*}
Substituting $t=\sigma\sqrt{\log(1/\delta)/\log2}$ proves the claim.
\end{proof}

\section{Entropic OT bias}\label{app:entropic-value-bounds}
\newcommand{\Law}{\mathrm{Law}}
\newcommand{\Var}{\mathrm{Var}}
We first give a general criterion for an upper bound on the entropic bias, and then specialize it to the bilinear cost.

For random variables $R,T$, their mutual information is
\[
I(R;T)\coloneqq\KL(\Law(R,T)\mid\Law(R)\otimes\Law(T)).
\]
This quantity is symmetric in $R,T$. We use the following properties.
\begin{itemize}
\item[(i)] \emph{Data processing.} If $R\to T\to W$ is a Markov chain, then $I(R;W)\le I(T;W)$.
\item[(ii)] \emph{Conditional independence.} If $R$ and $T$ are conditionally independent given $Z$, then $R\to Z\to T$ is a Markov chain, so $I(R;T)\le I(R;Z)$.
\item[(iii)] \emph{Gaussian-channel bound}~\cite[Section~9.1]{cover2006elements}. If $U$ is an $\RR^q$-valued random variable with finite second moment and $G\sim\Nn(0,I_q)$ is independent of $U$, then, for $t>0$,
\[
I(U;U+\sqrt tG)\le\frac q2\log\left(1+\frac{\Var(U)}{qt}\right),
\qquad \Var(U)\coloneqq\EE\|U-\EE U\|^2.
\]
\end{itemize}

\begin{prop}[Cross-difference criterion]\label{prop:cross-difference-bias}
Let $\alpha,\beta$ be probability measures on $\RR^d$, and let $c$ be a Borel cost satisfying
\[
|c(x,y)|\le f(x)+g(y),\qquad
0\le f\in L^1(\alpha),\quad 0\le g\in L^1(\beta).
\]
Suppose that $\pi\in\Pi(\alpha,\beta)$ is an optimal coupling for $c$, and write
\[
b_\varepsilon(c)\coloneqq\OT^\varepsilon_{\alpha,\beta}(c)-\OT_{\alpha,\beta}(c).
\]
For $p\coloneqq(x,y)$ and $p'\coloneqq(x',y')$, define the cross-difference
\[
\delta_c(p,p')\coloneqq c(x,y')+c(x',y)-c(x,y)-c(x',y').
\]
Assume that a Borel map $S:\RR^d\times\RR^d\to\RR^q$, $q\ge1$, belongs to $L^2(\pi)$ and, for some $L\ge0$, satisfies
\begin{equation}\label{eq:cross-difference-control}
\delta_c(p,p')\le L\|S(p)-S(p')\|^2
\qquad\text{for }\pi\otimes\pi\text{-almost every }(p,p').
\end{equation}
Set $V\coloneqq\Var(S(X,Y))$ for $(X,Y)\sim\pi$. Then, for every $t,\varepsilon>0$,
\begin{equation}\label{eq:general-bias-channel}
0\le b_\varepsilon(c)\le Lqt+
\frac{q\varepsilon}{2}\log\left(1+\frac V{qt}\right).
\end{equation}
Consequently, with $C_S\coloneqq2LV/q$,
\begin{equation}\label{eq:general-cross-bias}
0\le b_\varepsilon(c)\le\frac{q\varepsilon}{2}
\left[1+\log\left(1+\frac{C_S}{\varepsilon}\right)\right].
\end{equation}
If $LV=0$, then $b_\varepsilon(c)=0$ for all $\varepsilon>0$.
\end{prop}

\begin{proof}
Let $P\coloneqq(X,Y)\sim\pi$, set $U\coloneqq S(P)$ and $Z\coloneqq U+\sqrt tG$, with $G\sim\Nn(0,I_q)$ independent of $P$. Given $Z$, let $P'\coloneqq(X',Y')$ be an independent copy of $P$ under its conditional law, and put $U'\coloneqq S(P')$. Then $\widetilde\pi_t\coloneqq\Law(X,Y')$ has marginals $\alpha,\beta$.

The positive Gaussian likelihood implies $\Law(P\mid Z=z)\ll\pi$ for almost every $z$. Thus $\Law(P,P')\ll\pi\otimes\pi$, and~\eqref{eq:cross-difference-control} applies to the conditional copies. Conditional exchangeability gives
\begin{align*}
\int c\,d\widetilde\pi_t-\int c\,d\pi
&=\frac12\EE\delta_c(P,P')\\
&\le\frac L2\EE\|U-U'\|^2
=L\EE\|U-\EE[U\mid Z]\|^2
\le Lqt.
\end{align*}
The equality uses conditional independence and
\[
\EE[\|U-U'\|^2\mid Z]=2\EE[\|U-\EE[U\mid Z]\|^2\mid Z].
\]
The last inequality follows by using $Z$ as an estimator of $U$, since $\EE\|U-Z\|^2=qt$.

Since $X$ and $Y'$ are conditionally independent given $Z$, properties (i)--(iii) yield
\begin{align*}
\KL(\widetilde\pi_t\mid\alpha\otimes\beta)
=I(X;Y')&\le I(X;Z)\le I(P;Z)=I(U;Z)\\
&\le\frac q2\log\left(1+\frac V{qt}\right).
\end{align*}
Here $I(P;Z)=I(U;Z)$ because $U$ is a function of $P$ and $P\to U\to Z$ is a Markov chain: data processing gives both inequalities. Using $\widetilde\pi_t$ as a competitor proves~\eqref{eq:general-bias-channel}; the lower bound follows from nonnegativity of relative entropy. If $L>0$, choose $t=\varepsilon/(2L)$ to obtain~\eqref{eq:general-cross-bias}. If $L=0$, send $t\to\infty$ in~\eqref{eq:general-bias-channel}; if $V=0$, send $t\downarrow0$. Both cases give zero bias.
\end{proof}

\begin{proof}[Proof of Proposition~\ref{prop:entropic-value-bounds}]
The bilinear cost satisfies the integrability assumption of Proposition~\ref{prop:cross-difference-bias}, since
\[
|c_A(x,y)|\le\frac{\|A\|_{op}}2(\|x\|^2+\|y\|^2),
\]
and an optimal coupling exists under second moments. Let $(X,Y)\sim\pi$ be such a coupling, and set
\[
s_{A\beta}^2\coloneqq\int\|A(y-m_\beta)\|^2\,d\beta(y)
\le\|A\|_{op}^2s_\beta^2.
\]
Suppose first that $s_\alpha s_{A\beta}>0$. For $a>0$, take the $d$-dimensional feature
\[
S_a(x,y)\coloneqq ax+a^{-1}Ay.
\]
Writing $h\coloneqq x-x'$ and $\ell\coloneqq y-y'$, we have
\[
\delta_{c_A}(p,p')=h^\top A\ell
\le\frac14\|ah+a^{-1}A\ell\|^2
=\frac14\|S_a(p)-S_a(p')\|^2.
\]
Moreover, the $L^2$ triangle inequality gives
\[
\Var(S_a(X,Y))\le(a s_\alpha+a^{-1}s_{A\beta})^2.
\]
Choose $a\coloneqq\sqrt{s_{A\beta}/s_\alpha}$. Proposition~\ref{prop:cross-difference-bias} then applies with $q=d$, $L=1/4$, and $V\le4s_\alpha s_{A\beta}$, so~\eqref{eq:general-cross-bias} yields
\[
0\le b_\varepsilon(A)\le\frac{d\varepsilon}{2}
\left[1+\log\left(1+\frac{2s_\alpha s_{A\beta}}{d\varepsilon}\right)\right].
\]
This implies~\eqref{eq:explicit-bilinear-upper}. If $s_\alpha s_{A\beta}=0$, either $X$ or $AY$ is almost surely constant; the expected cost is independent of the coupling and the product coupling gives $b_\varepsilon(A)=0$.
Finally, for $\|A\|_F\le B$ and $0<\varepsilon\le1$,
\[
1+\frac{2\|A\|_{op}s_\alpha s_\beta}{d\varepsilon}
\le\frac{1+2Bs_\alpha s_\beta/d}{\varepsilon},
\]
proving~\eqref{eq:global-entropic-bias}.
\end{proof}

\subsection{Entropic bias bound for general costs}
The cross-difference criterion also gives an explicit upper bound when $\partial_xc$ is Lipschitz in the target variable, uniformly in the source variable.

\begin{cor}[Lipschitz cross-variable control]\label{cor:cross-lipschitz-bias}
Assume the integrability and existence hypotheses of Proposition~\ref{prop:cross-difference-bias}, and let $\alpha,\beta\in\mathcal P_2(\RR^d)$. Suppose that, for each $y\in\operatorname{spt}(\beta)$, $c(\cdot,y)$ is continuously differentiable on an open convex neighborhood $\Omega_x$ of $\operatorname{spt}(\alpha)$ and that, for some $M\ge0$,
\[
\|\partial_xc(x,y)-\partial_xc(x,y')\|\le M\|y-y'\|
\qquad(x\in\Omega_x,\ y,y'\in\operatorname{spt}(\beta)).
\]
With $s_\alpha,s_\beta$ denoting the centered second-moment square roots defined in Proposition~\ref{prop:entropic-value-bounds}, for every $\varepsilon>0$,
\begin{equation}\label{eq:cross-lipschitz-bias}
0\le b_\varepsilon(c)\le d\varepsilon
\left[1+\log\left(1+\frac{Ms_\alpha s_\beta}{d\varepsilon}\right)\right].
\end{equation}
If $Ms_\alpha s_\beta=0$, then $b_\varepsilon(c)=0$ for all $\varepsilon>0$.
\end{cor}

\begin{proof}
Integration along the $x$-segment gives, for $h\coloneqq x-x'$ and $\ell\coloneqq y-y'$,
\[
\delta_c(p,p')\le M\|h\|\|\ell\|
\le\frac M2\|S_a(p)-S_a(p')\|^2,
\qquad S_a(x,y)\coloneqq(ax,a^{-1}y)\in\RR^{2d}.
\]
For this feature, $V=a^2s_\alpha^2+a^{-2}s_\beta^2$. If $s_\alpha s_\beta>0$, choosing $a\coloneqq\sqrt{s_\beta/s_\alpha}$ gives $V=2s_\alpha s_\beta$. Applying~\eqref{eq:general-cross-bias} with $q=2d$ and $L=M/2$ proves~\eqref{eq:cross-lipschitz-bias}, including zero bias when $M=0$. If either centered second moment vanishes, one marginal is a point mass, so the only coupling is the product coupling and the bias is zero.
\end{proof}

\begin{rem}[Comparison with sharper upper bounds]
Corollary~\ref{cor:cross-lipschitz-bias} gives coefficient $d$ under second moments alone, with the advantage that that the proof is elementary. Under additional assumptions, the coefficient can be replaced by $d/2$: This holds for locally $C^{1,1}$ costs on a product of open convex neighborhoods of the supports and compactly supported marginals with bounded densities, without any nondegeneracy assumption~\cite[Proposition~3.4]{carlier2023entropicrate}. It also holds for compactly supported arbitrary marginals when the cost is $C^2$ on such neighborhoods and $D^2_{xy}c$ is invertible there~\cite[Theorem~3.8 and Lemma~3.15]{eckstein2024quantization}. In either case,
\[
b_\varepsilon(c)\le\frac d2\varepsilon\log(1/\varepsilon)+O(\varepsilon).
\]
\end{rem}

\section{Useful linear algebra and optimization lemmas}

\begin{lem}\label{lem:projection-at}

For $U,V\in\RR^{d\times d}$, define the orthogonal projection onto $\{U\}^{\perp}$ by
\[
P_U^\perp V\coloneqq
\begin{cases}
V-\dfrac{\dotp{V}{U}}{\norm{U}^2}U,&U\ne0,\\[1ex]
V,&U=0.
\end{cases}
\]
Given $A,A_0\in\RR^{d\times d}\setminus \ens{0}$, set $A_t\coloneqq A_0+t(A-A_0)$ for $t\in[0,1]$. Then
\[
\inf_{t\in[0,1]}\norm{P_{A_t}^\perp(A-A_0)}^2
=\min\left\{
\norm{P_{A_0}^\perp A}^2,\,
\norm{P_A^\perp A_0}^2
\right\} = \min\{\norm{A}^2,\norm{A_0}^2\}
\left(
1-\dotp{\frac{A}{\norm A}}{\frac{A_0}{\norm{A_0}}}^{\,2}
\right).
\]
\end{lem}

\begin{proof}
Set $B\coloneqq A-A_0$. If $A$ and $A_0$ are collinear, the identities follow directly, including when an endpoint is zero. Otherwise $A_t\ne0$ on $[0,1]$, and
\[
\inf_{t\in[0,1]}\norm{P_{A_t}^\perp B}^2
=\norm B^2-\max_{t\in[0,1]}g(t),
\qquad
g(t)\coloneqq\frac{\dotp{B}{A_t}^2}{\norm{A_t}^2}.
\]
Differentiation gives
\[
g'(t)=\frac{2\dotp{B}{A_t}}{\norm{A_t}^4}
\left(\norm B^2\norm{A_0}^2-\dotp{B}{A_0}^2\right).
\]
The factor in parentheses is positive, and $\dotp{B}{A_t}$ is increasing in $t$. Thus $g$ is monotone or decreases and then increases, so its maximum occurs at an endpoint. Since $P_{A_0}^\perp B=P_{A_0}^\perp A$ and $P_A^\perp B=-P_A^\perp A_0$, this proves the first identity. Expanding the two endpoint projections gives the normalized formula.
\end{proof}

\section{A H\"older composition estimate}\label{app:holder-composition}
The following special case of de la Llave--Obaya~\cite[Proposition~6.2(ii)--(iii) and its proof]{delallave_obaya1999composition} is included with a proof to make explicit the exponent loss used in Section~\ref{sec:curvature_OT}.

\begin{lem}[Composition with a gap in H\"older exponents]\label{lem:holder-composition-gap}
Let $\Omega\subset\RR^d$ and $U\subset\RR^p$ be bounded convex open sets, and fix $m\in\NN_0$, $0<\theta<\kappa<1$ and $M>0$. Set $r\coloneqq\max\{m,1\}$. For $a\in C^{m,\kappa}(\bar U;\RR^q)$ and maps $T_i\in C^{r,\theta}(\bar\Omega;\RR^p)$ satisfying
\[
T_i(\bar\Omega)\subset\bar U,\qquad
\|T_i\|_{C^{r,\theta}}\le M,\qquad
\delta\coloneqq\|T_1-T_2\|_{C^{r,\theta}}\le1,
\]
we have
\begin{equation}\label{eq:holder-composition-gap}
\|a\circ T_1-a\circ T_2\|_{C^{m,\theta}}
\le C\|a\|_{C^{m,\kappa}}\delta^{\kappa-\theta}.
\end{equation}
The constant $C$ depends only on the fixed domains, dimensions, exponents, $m$ and $M$.
\end{lem}
\begin{proof}
For $0\le j\le m$, put $b\coloneqq D^j a$ and $Q_b\coloneqq b\circ T_1-b\circ T_2$. The definition of $C^{m,\kappa}$, together with the mean-value estimate on the convex set $U$ for $j<m$, gives $\|b\|_{C^{0,\kappa}}\le C\|a\|_{C^{m,\kappa}}$. Since $r\ge1$, the maps $T_i$ have uniformly bounded Lipschitz constants. Hence
\[
\|Q_b\|_\infty\le C\|a\|_{C^{m,\kappa}}\delta^\kappa,
\qquad
[Q_b]_{C^{0,\kappa}}\le C\|a\|_{C^{m,\kappa}}.
\]
The elementary interpolation inequality
\[
[f]_{C^{0,\theta}}
\le C\|f\|_\infty^{1-\theta/\kappa}
[f]_{C^{0,\kappa}}^{\theta/\kappa}
\]
follows by bounding $|f(x)-f(y)|$ by both $2\|f\|_\infty$ and $[f]_{C^{0,\kappa}}|x-y|^\kappa$. Together with $\delta\le1$, it yields
\begin{equation}\label{eq:holder-coefficient-difference}
\|Q_b\|_{C^{0,\theta}}
\le C\|a\|_{C^{m,\kappa}}\delta^{\kappa-\theta}.
\end{equation}
For $m=0$, this is the desired estimate.

For $m\ge1$, Taylor's formula on the convex set $\bar U$ and continuity of the derivatives justify the chain rule up to the boundary. It expresses every derivative of $a\circ T_i$ of order at most $m$ as a sum of contractions of $D^j a\circ T_i$ with products of derivatives of $T_i$. The coefficient compositions are bounded in $C^{0,\theta}$ by $C\|a\|_{C^{m,\kappa}}$. The products of derivatives of $T_i$ are bounded in that space, and their differences are $O(\delta)$. Subtracting the two expressions, using~\eqref{eq:holder-coefficient-difference} and the H\"older product inequality, gives
\[
\|a\circ T_1-a\circ T_2\|_{C^{m,\theta}}
\le C\|a\|_{C^{m,\kappa}}
\bigl(\delta^{\kappa-\theta}+\delta\bigr).
\]
Since $\delta\le1$, this proves~\eqref{eq:holder-composition-gap}. The positive power $\kappa-\theta$ ensures continuity as the inner maps approach one another.
\end{proof}

\end{document}